\documentclass[a4paper,pdftex]{amsart}
\usepackage{amssymb}
\usepackage{tikz-cd}

\newtheorem{thm}{Theorem}[section]
\newtheorem{prop}[thm]{Proposition}
\newtheorem{lem}[thm]{Lemma}
\newtheorem{cor}[thm]{Corollary}
\theoremstyle{definition}
\newtheorem{defn}[thm]{Definition}
\newtheorem*{claim}{Claim}
\theoremstyle{remark}
\newtheorem{rem}[thm]{Remark}

\newcommand{\Z}{\mathbb{Z}}

\newcommand{\R}{\mathbb{R}}
\newcommand{\C}{\mathbb{C}}
\newcommand{\GL}{\mathrm{GL}}
\newcommand{\aff}{\mathrm{aff}}
\newcommand{\red}{\mathrm{red}}
\newcommand*{\alggrp}[1]{\boldsymbol{#1}}
\DeclareMathOperator{\Hom}{Hom}

\DeclareMathOperator{\Ker}{Ker}
\DeclareMathOperator{\Imm}{Im}

\DeclareMathOperator{\supp}{supp}

\DeclareMathOperator{\val}{val}
\DeclareMathOperator{\Stab}{Stab}

\numberwithin{equation}{section}

\usepackage{CJKutf8}

\title{Modulo $p$ parabolic induction of pro-$p$-Iwahori Hecke algebra}
\author{Noriyuki Abe}
\address{Creative Research Institution (CRIS), Hokkaido University, N21, W10, Kita-ku, Sapporo, Hokkaido 001-0021, Japan}
\email{abenori@math.sci.hokudai.ac.jp}

\begin{document}
\begin{abstract}
We study the structure of parabolic inductions of a pro-$p$-Iwahori Hecke algebra.
In particular, we give a classification of irreducible modulo $p$ representations of pro-$p$-Iwahori Hecke algebra in terms of supersingular representations.
Since supersingular representations are classified by Ollivier and Vign\'eras, it completes the classification of irreducible modulo $p$ representations.
\end{abstract}
\maketitle

\section{Introduction}
\label{sec:Introduction}
Let $G$ be the group of $F$-valued points of a connected reductive group over a $p$-adic field $F$ and $K'$ its open compact subgroup.
Then $K'$-biinvariant functions with compact support on $G$ forms an algebra via the convolution product.
We call this algebra the Hecke algebra.
When we investigate the representation theory of $G$ over a characteristic zero field, this algebra has an important role.
One of the most important case is when $K'$ is an Iwahori subgroup.
In this case, the category of representations of the Hecke algebra is equivalent to a block of the category of smooth representations of $G$.

We are interested in the representations of $G$ over a characteristic $p$ field.
In this setting, it is natural to consider a \emph{pro-$p$-Iwahori subgroup} since any non-zero representations of $G$ over a characteristic $p$ field has a non-zero vector fixed by a pro-$p$-Iwahori subgroup.
The corresponding Hecke algebra is called a \emph{pro-$p$-Iwahori Hecke algebra} and the structure of this algebra is studied by Vign\'eras~\cite{Vigneras-prop}.
This structure theorem is used in \cite{arXiv:1412.0737} in which we gave a classification of irreducible admissible modulo $p$ representations of $G$ in terms of supersingular representations.

Another motivation to consider a pro-$p$-Iwahori Hecke algebra is to study Galois representations.
It is conjectured by Vign\'eras~\cite{MR2122539} and proved by Ollivier~\cite{MR2728487} that there is the ``numerical modulo $p$ Langlands correspondence'' between modulo $p$ Galois representations and modulo $p$ representations of a pro-$p$-Iwahori Hecke algebra of $\GL_n$.
It asserts that the number of $n$-dimensional modulo $p$ Galois representations with a fixed determinant and the number of $n$-dimensional modulo $p$ supersingular representations of a pro-$p$-Iwahori Hecke algebra of $\GL_n$ with a fixed action of the center of $\GL_n$ coincide with each other.
Recently, Gro\ss e-Kl\"onne~\cite{Grosse-Klonne-pro-p-Iwahori-Hecke-and-Galois-rep} constructed a functor from the category of modulo $p$ representations of a pro-$p$-Iwahori Hecke algebra of $\GL_n$ to the category of modulo $p$ Galois representations which gives a realization of the numerical modulo $p$ Langlands correspondence.

In this paper, we study the representations, especially parabolic inductions of a pro-$p$-Iwahori Hecke algebra.
Our main theorem is the classification of the irreducible modulo $p$ representations of a pro-$p$-Iwahori Hecke algebra in terms of supersingular representations.
The supersingular representations are classified by Ollivier~\cite{arXiv:1211.5366} (split case) and Vign\'eras~\cite{Vigneras-prop-III} (general).

We state our main result.
Let $C$ be an algebraically closed field of characteristic $p$.
Fix a pro-$p$-Iwahori subgroup of $G$ and let $\mathcal{H}$ be the corresponding Hecke algebra with coefficients in $C$.
We consider $\mathcal{H}$-modules.
Here, modules always are right modules unless otherwise stated.
Let $P$ be a parabolic subgroup, $M$ its Levi subgroup and $\mathcal{H}_M$ a pro-$p$-Iwahori Hecke algebra of $M$.
Then for a representation $\sigma$ of $\mathcal{H}_M$, the parabolic induction $I_P(\sigma)$ is defined (Definition~\ref{defn:parabolic induction}).

The statement of the main theorem is almost the same as the group case~\cite{MR3143708,arXiv:1412.0737}.
To state it, fix a minimal parabolic subgroup $B$ and its Levi subgroup $Z$.
Then we have a root system $\Phi$ and the set of simple roots $\Pi$.
For $\alpha\in\Pi$, let $G_\alpha'$ be the group generated by the root subgroups of $G$ corresponding to $\pm\alpha$ and put $Z'_\alpha = Z\cap G'_\alpha$.
Then for a standard Levi subgroup $M\supset Z$ and a representation $\sigma$ of $\mathcal{H}_M$ we define
\[
	\Pi(\sigma) = \{\alpha\in\Pi\mid \langle\Pi_M,\check{\alpha}\rangle = 0,\ \text{$T^M_\lambda$ is identity for any $\lambda\in Z'_\alpha$}\}\cup \Pi_M,
\]
where $T^M_\lambda\in \mathcal{H}_M$ is the function supported on the double coset containing $\lambda$ with respect to the pro-$p$-Iwahori subgroup of $M$ (Iwahori-Matsumoto basis) and $\Pi_M$ is the set of simple roots of $M$.
Let $P(\sigma)$ be the corresponding parabolic subgroup.
Now consider triples $(P,\sigma,Q)$ such that
\begin{itemize}
\item $P$ is a parabolic subgroup with Levi subgroup $M$.
\item $\sigma$ is a supersingular representation of $M$.
\item $Q$ is a parabolic subgroup such that $P\subset Q\subset P(\sigma)$.
\end{itemize}
For such a pair, we define an $\mathcal{H}$-module $I(P,\sigma,Q)$ as follows.
Let $M(\sigma)\supset Z$ be the Levi subgroup of $P(\sigma)$.
We can prove that there is the ``extension'' $e(\sigma)$ of $\sigma$ to $\mathcal{H}_{M(\sigma)}$ (Proposition~\ref{prop:extending}).
Then for each $Q'\supset Q$, we have an embedding $I_{Q'}(e(\sigma))\hookrightarrow I_Q(e(\sigma))$.
Now let
\[
	I(P,\sigma,Q) = I_Q(e(\sigma))/\sum_{Q'\supsetneq Q}I_{Q'}(e(\sigma)).
\]
Then the main theorem of this paper is the following.
\begin{thm}\label{thm:main theorem, introduction, not precisely}
The corresponding $(P,\sigma,Q)\mapsto I(P,\sigma,Q)$ gives a bijection between such triples and isomorphism classes of irreducible representations of $\mathcal{H}$.
\end{thm}

The study is based on the structure theory of pro-$p$-Iwahori Hecke algebras studied by Vign\'eras~\cite{Vigneras-prop}.
Fix a maximal split torus in $Z$ and denote (the group of $F$-valued points of) it by $S$.
Let $I(1)$ be the pro-$p$-Iwahori subgroup. (It corresponds to a certain chamber in the apartment corresponding to $S$, see the next section for the detail.)
Then $Z$ is the centralizer of $S$ in $G$.
Put $\widetilde{W}(1) = N_G(S)/(Z\cap I(1))$.
Then we have the Bruhat decomposition $\widetilde{W}(1) = I(1)\backslash G/I(1)$.
Hence there is a basis of $\mathcal{H}$ which is indexed by $\widetilde{W}(1)$.
The basis is called \emph{Iwahori-Matsumoto basis}.
Similar to the affine Hecke algebra, this basis satisfies the braid relations and quadratic relations.
Like the affine Hecke algebra, we have another basis $\{E(\widetilde{w})\mid \widetilde{w}\in \widetilde{W}(1)\}$, called \emph{Bernstein basis}.
Put $\Lambda(1) = Z/(Z\cap I(1))\subset \widetilde{W}(1)$.
Then $\mathcal{A} = \bigoplus_{\lambda \in \Lambda(1)}CE(\lambda)$ is a subalgebra.
In general, $\mathcal{A}$ is not commutative.
However, it is almost commutative and the study of representations of $\mathcal{A}$ is not so difficult.

The basic tactics to study the representations is via the restriction to $\mathcal{A}$.
Hence it is important to know about irreducible representations of $\mathcal{A}$.
An important invariant of an irreducible representation $X$ of $\mathcal{A}$ is its support $\supp X$ defined by $\supp X = \{\lambda\in\Lambda(1)\mid XE(\lambda)\ne 0\}$.
Then we can prove that it is the closure of a facet (Proposition~\ref{prop:irred rep of A factors chi_Theta}).
More precisely, the following holds.
Let $X_*(S)$ be the group of cocharacters of $S$ and set $V = X_*(S)\otimes\R$.
Then the finite Weyl group $W$ acts on $V$ as a reflection group.
Hence we have a notion of a facet in $V$.
We also have a group homomorphism $\nu \colon \Lambda(1)\to V$.
Let $\Theta$ be a subset of $\Pi$ and put $\Lambda_\Theta(1) = \{\lambda\in\Lambda(1)\mid \langle \nu(\lambda),\alpha\rangle = 0\ (\alpha\in\Theta)\}$.
For $w\in W$, fix a representative $n_w\in \widetilde{W}(1)$.
Let $C[\Lambda_\Theta(1)] =  \bigoplus_{\tau\in\Lambda_\Theta(1)}C\tau_\lambda$ be the group algebra of $\Lambda_\Theta(1)$.
Define $w\chi_\Theta\colon \mathcal{A}\to C[\Lambda_\Theta(1)]$ by
\[
	w\chi_\Theta(E(\lambda)) =
	\begin{cases}
	\tau_{n_w^{-1}\lambda n_w} & (\text{$w^{-1}(\nu(\lambda))$ is dominant and $n_w^{-1}\lambda n_w\in \Lambda_\Theta(1)$}),\\
	0 & (\text{otherwise}).
	\end{cases}
\]
Then we can prove that any irreducible representation factors through $w\chi_\Theta$ for some $w$ and $\Theta$.
For the proof of our main theorem, we study the modules $w\chi_\Theta\otimes_\mathcal{A}\mathcal{H}$.

We study these modules in Section~\ref{sec:Intertwining operators} by constructing intertwining operators.
Let $s\in W$ be a simple reflection and assume that $sw > w$.
Then we construct a homomorphism $w\chi_\Theta\otimes_\mathcal{A}\mathcal{H}\to sw\chi_\Theta\otimes_\mathcal{A}\mathcal{H}$ (Proposition~\ref{prop:construction of intertwining operators}) which is always injective (Proposition~\ref{prop:injectivity of intertwining operator}).
Moreover, often it is an isomorphism.
In fact, if we put $\Pi_w = \{\alpha\in\Pi\mid w(\alpha) > 0\}$, then $\Pi_w = \Pi_{sw}$ implies that $w\chi_\Theta\otimes_\mathcal{A}\mathcal{H}\to sw\chi_\Theta\otimes_\mathcal{A}\mathcal{H}$ is an isomorphism (Theorem~\ref{thm:std mod depends only on Delta_w}).
Moreover, we have $w_1\chi_\Theta\otimes_\mathcal{A}\mathcal{H}\simeq w_2\chi_\Theta\otimes_\mathcal{A}\mathcal{H}$ if $\Pi_{w_1} = \Pi_{w_2}$.
If $\Pi_{w} \ne \Pi_{sw}$, $w\chi_\Theta\otimes_\mathcal{A}\mathcal{H}\to sw\chi_\Theta\otimes_\mathcal{A}\mathcal{H}$ is not an isomorphism.
We can construct an intertwining operator $sw\chi_\Theta\otimes_\mathcal{A}\mathcal{H}\to w\chi_\Theta\otimes_\mathcal{A}\mathcal{H}$ in opposite direction and calculate the compositions.

In Section~\ref{sec:Classification theorem}, we give a definition and fundamental properties of parabolic inductions.
One of the most important properties to prove our main theorem is the structure of a parabolic induction as an $\mathcal{A}$-module (Lemma~\ref{lem:I(P,sigma,Q) as A-mod}).
For example, if $P$ is a minimal parabolic subgroup and $P(\sigma) = P$, then we can prove $I_P(\sigma)|_{\mathcal{A}} = \bigoplus_{w\in W}wX$ for some irreducible representation $X$ of $\mathcal{A}$.
In this case, using results in Section~\ref{sec:Intertwining operators}, we can prove that for any $w,w'\in W$, we have $wX\otimes_\mathcal{A}\mathcal{H}\simeq w'X\otimes_\mathcal{A}\mathcal{H}$.
Hence if $\pi$ is a submodule of $I_P(\sigma)$ and $\pi|_{\mathcal{A}}$ contains $wX$ for some $w\in W$, then $\pi|_{\mathcal{A}}$ contains $wX$ for any $w\in W$.
Therefore we get $\pi = I_P(\sigma)$, so $I_P(\sigma) = I(P,\sigma,P)$ is irreducible.
This is a part of our main theorem.
Using such arguments, we prove our main theorem in Section~\ref{sec:Classification theorem}.

We have nothing about the relations between the representations of the group in this paper.
We hope to address this question in a future work.
\subsection*{Acknowledgment}
I had many discussion with Marie-France Vign\'eras on the structure of pro-$p$-Iwahori Hecke algebras.
I thank her for reading the manuscript and giving helpful comments.
This work was supported by JSPS KAKENHI Grant Number 26707001.

\section{Notation and Preliminaries}
\subsection{Notation}\label{subsec:Notation}
We will use the slightly different notation from the introduction.
Our main reference is \cite{Vigneras-prop}.

Let $F$ be a non-archimedean local field, $\mathcal{O}$ its ring of integers, $\kappa$ its residue field and $p$ the characteristic of $\kappa$.
Let $\alggrp{G}$ be a connected reductive group over $F$.
As usual, the group of its valued points is denoted by $\alggrp{G}(F)$.
Fix a maximal split torus $\alggrp{S}$ of $\alggrp{G}$ and a minimal parabolic subgroup $\alggrp{B}$ which contains $\alggrp{S}$.
Then the centralizer $\alggrp{Z}$ of $\alggrp{S}$ in $\alggrp{G}$ is a Levi subgroup of $\alggrp{B}$.
Let $\alggrp{U}$ be the unipotent radical of $\alggrp{B}$, $\alggrp{\overline{B}} = \alggrp{Z}\alggrp{\overline{U}}$ the opposite parabolic subgroup of $\alggrp{B}$.
Take a special point from the apartment attached to $\alggrp{S}$ and let $K$ be the special parahoric subgroup corresponding to this point.
Then there is a connected reductive group $\alggrp{G}_\kappa$ over $\kappa$ and the surjective homomorphism $K\to \alggrp{G}_\kappa(\kappa)$.
The kernel of $K\to \alggrp{G}_\kappa(\kappa)$ is a pro-$p$ group.
The image of $\alggrp{Z}(F)\cap K$ (resp.~$\alggrp{U}(F)\cap K$, $\alggrp{\overline{U}}(F)\cap K$) is the $\kappa$-valued points of an algebraic subgroup $\alggrp{Z}_\kappa$ (resp.~$\alggrp{U}_\kappa$, $\alggrp{\overline{U}}_\kappa$) of $\alggrp{G}_\kappa$.
Put $Z_\kappa = \alggrp{Z}_\kappa(\kappa)$.
Let $I(1)$ be the inverse image of $\alggrp{U}_\kappa(\kappa)$ by $K\to \alggrp{G}_\kappa(\kappa)$.
Then this is a \emph{pro-$p$-Iwahori subgroup} of $\alggrp{G}(F)$.
The Hecke algebra attached to $(\alggrp{G}(F),I(1))$ is called a \emph{pro-$p$-Iwahori Hecke algebra}.
We study the modules of this algebra over an algebraically closed filed of characteristic $p$.
However, we will not use this algebra directly due to a technical reason. % (See the argument in the proof of Lemma~\ref{lem:simple Bernstein relations}.)
Instead of this algebra, we use another algebra defined by generators and relations which will be introduced later.
Under a suitable specialization, the algebra becomes the pro-$p$-Iwahori Hecke algebra.

The torus $\alggrp{S}$ gives a root datum $(X^*(\alggrp{S}),\Phi,X_*(\alggrp{S}),\check{\Phi})$ and $\alggrp{B}$ gives a positive system $\Phi^+\subset \Phi$.
%These data also define an affine root system $\Phi_\aff$ and its positive system $\Phi_\aff^+$.
Let $W = N_{\alggrp{G}(F)}(\alggrp{S}(F))/\alggrp{Z}(F)$ be the finite Weyl group and set $\widetilde{W} = N_{\alggrp{G}(F)}(\alggrp{S}(F))/(\alggrp{Z}(F)\cap K)$ where $N_{\alggrp{G}(F)}(\alggrp{S}(F))$ is the normalizer of $\alggrp{S}(F)$ in $\alggrp{G}(F)$.
Then we have the surjective homomorphism $\widetilde{W}\to W$ and its kernel $\alggrp{Z}(F)/(\alggrp{Z}(F)\cap K)$ is denoted by $\Lambda$.
Set $V = X_*(\alggrp{S})\otimes_\Z\R$.
Then $V$ is identified with the apartment corresponding to $\alggrp{S}$ and $\widetilde{W}$ acts on $V$ as an affine transformations.
In particular, $\alggrp{Z}(F)$ acts on $V$ as a translation.
Hence we have a group homomorphism $\nu\colon \alggrp{Z}(F)\to V$.
It is characterized by $\alpha(\nu(z)) = -\val(\alpha(z))$ for $\alpha\in X^*(\alggrp{S})$ and $z\in \alggrp{S}(F)$.
%\begin{rem}
%In \cite{Vigneras-prop}, this homomorphism is denoted by $v$ and $\nu$ is used for $-v$.
%Hence the conventions about positive and negative are interchanged.
%Since we take $I(1)$ (not $I(1)$ which is the inverse image of $U_\kappa$), we can use results in \cite{Vigneras-prop} without modifications.\Caution{後で考える}
%\end{rem}
Since $\nu$ annihilates  a compact subgroup ($V$ has no non-trivial compact subgroup), $\nu$ factors thorugh $\alggrp{Z}(F)\to \Lambda$.
The induced homomorphism $\Lambda\to V$ is denoted by the same letter $\nu$.
Since we have a positive system $\Phi^+$, we have the set of simple reflections $S_\aff\subset\widetilde{W}$.
Let $\widetilde{W}_\aff$ be the group generated by $S_\aff$.
Then there exists a reduced root datum $\Sigma\subset V^* = \Hom_\R(V,\R)$ such that $\widetilde{W}_\aff$ is the group generated by the reflections with respect to $\{v\in V\mid \alpha(v) + k = 0\}$ where $\alpha\in\Sigma$ and $k\in\Z$.
Since we fixed a positive system $\Phi^+\subset \Phi$, this determines a positive system $\Sigma^+\subset \Sigma$.
Let $\Delta\subset\Sigma^+$ be the set of simple roots.
There is a natural bijection between $\Delta$ and the set of simple roots in $\Phi^+$. (The set of simple roots in $\Phi^+$ is denoted by $\Pi$ in Section~\ref{sec:Introduction}. We will never use $\Pi$.)
More precisely, for $\alpha\in\Delta$, there exists a real number $r > 0$ such that $r\alpha$ is a simple root for $\Phi^+$.
%There is a bijection between $\Sigma$ and the set of reduced roots $\Phi^\red$ in $\Phi$.
%We also have a bijection between $\Sigma_\aff = \Sigma\times\Z$ and $\Phi^\red_\aff$ preserving the positive system.
Define the positive system $\Sigma_\aff^+$ of $\Sigma_\aff$ by $\Sigma_\aff^+ = \Sigma^+\times\Z_{\ge 0}\cup \Sigma\times\Z_{>0}$.
For $\alpha\in\Sigma$ (resp.\ $\widetilde{\alpha}\in\Sigma_\aff$), let $s_\alpha\in W$ (resp.\ $s_{\widetilde{\alpha}}\in \widetilde{W}$) be the corresponding reflection.
An element in $v\in V$ is called dominant (resp.~anti-dominant) if for any $\alpha\in\Sigma^+$, $\langle v,\alpha\rangle \ge 0$ (resp.\ $\langle v,\alpha\rangle \le 0$).
We call $v$ regular if $\langle v,\alpha\rangle\ne 0$ for any $\alpha\in\Sigma$.

Let $\alpha\in\Sigma$.
Then there exists a unique element $\alpha'\in \Phi^\red\cap \R_{>0}\alpha$.
The root $\alpha'$ gives a unipotent subgroup of $\alggrp{G}$ which is denoted by $\alggrp{U}_\alpha$.
Let $G'_\alpha$ be the group generated by $\alggrp{U}_\alpha(F)$ and $\alggrp{U}_{-\alpha}(F)$.
We also denote this group by $G'_s$ where $s = s_\alpha$.
Notice that it is not the group of $F$-valued points of an algebraic subgroup in general.
Set $Z'_\alpha = G'_\alpha\cap \alggrp{Z}(F)$ and the image of $Z'_\alpha$ in $\Lambda$ is denoted by $\Lambda'_\alpha$.
It is also denoted by $\Lambda'_{s}$ where $s = s_\alpha$.
Let $(\alpha,k)\in \Sigma_\aff$ and $r\in\R_{>0}$ such that $\alpha' = r\alpha$.
Put $U_{(\alpha,k)} = U_{\alpha' + rk}$ where the notation is in \cite[3.5]{Vigneras-prop}.

As usual, the length function of $\widetilde{W}$ is denoted by $\ell$.
It is defined by $\ell(\widetilde{w}) = \#(\Sigma_\aff^+\cap \widetilde{w}\Sigma_\aff^-)$.
We have the following length formula~\cite[Proposition~1.23]{MR0185016} for $\lambda\in \Lambda$ and $w\in W$:
\begin{equation}\label{eq:length formula}
\ell(\lambda w) = \sum_{\alpha\in\Sigma^+,w^{-1}(\alpha) > 0}\lvert\langle\nu(\lambda),\alpha\rangle\rvert + \sum_{\alpha\in\Sigma^+,w^{-1}(\alpha) < 0}\lvert\langle\nu(\lambda),\alpha\rangle - 1\rvert.
\end{equation}
Notice that we have an embedding $W\hookrightarrow \widetilde{W}$ since we fixed a special point and $\widetilde{W} = W \ltimes \Lambda$.
Put $S = S_\aff\cap W$.
This is the set of simple reflections in $W$ and $(W,S)$ is a Coxeter system.

Put $\widetilde{W}(1) = N_{\alggrp{G}(F)}(\alggrp{S}(F))/(\alggrp{Z}(F)\cap I(1))$.
Then we have a bijection $\widetilde{W}(1)\simeq I(1)\backslash \alggrp{G}(F)/I(1)$.
Hence the pro-$p$-Iwahori Hecke algebra $\mathcal{H}(\alggrp{G}(F),I(1))$ has a basis parameterized by $\widetilde{W}(1)$.
There is a surjective homomorphism $\widetilde{W}(1)\to \widetilde{W}$.
The kernel is isomorphic to $Z_\kappa$.
We will regard $Z_\kappa$ as a subgroup of $\widetilde{W}(1)$.
The kernel of $\widetilde{W}(1)\to W$ is $\alggrp{Z}(F)/(\alggrp{Z}(F)\cap I(1))$ and it is denoted by $\Lambda(1)$.
The composition $\Lambda(1)\to \Lambda\xrightarrow{\nu}V$ is also denoted by $\nu$.
We call $\lambda\in \Lambda(1)$ dominant (resp.\ anti-dominant, regular) if $\nu(\lambda)$ is dominant (resp.\ anti-dominant, regular).

For $s\in S_\aff$, let $q_s$ be the indeterminates such that if $s$ and $s'$ are conjugate to each other, then $q_s = q_{s'}$.
Let $\Z [q_s\mid s\in S_\aff/\mathord{\sim}]$ be the polynomial algebra with these indeterminate here $\sim$ means the equivalence relation defined via the conjugation by $\widetilde{W}$.
It will be denoted by $\Z[q_s]$ for short.
Via a reduced expression, we extend $s\mapsto q_s$ to $\widetilde{W}\ni \widetilde{w}\to q_{\widetilde{w}}\in \prod_{s\in S_\aff/\mathord{\sim}}q_s^{\Z_{\ge 0}}$.
By $\widetilde{W}(1)\to \widetilde{W}$ we have $q_{\widetilde{w}}$ for any $\widetilde{w}\in \widetilde{W}(1)$.
Since $Z_\kappa\subset \widetilde{W}(1)$ is a normal subgroup, the adjoint action of $\widetilde{w}\in \widetilde{W}(1)$ gives a homomorphism $\Z[q_s^{1/2}][Z_\kappa]\to \Z[q_s^{1/2}][Z_\kappa]$.
Denote this action by $c\mapsto \widetilde{w}\cdot c$.
For each $\widetilde{s}$ which is a lift of $s\in S_\aff$, fix $c_{\widetilde{s}}\in \Z[Z_\kappa]$ such that if $\widetilde{w}\widetilde{s}\widetilde{w}^{-1} = \widetilde{s}'$ for $\widetilde{w}\in \widetilde{W}$, then $\widetilde{w}\cdot c_{\widetilde{s}} = c_{\widetilde{s}'}$ and $c_{t\widetilde{s}} = tc_{\widetilde{s}}$ for $t\in Z_\kappa$.
Then Vign\'eras~\cite[4.3]{Vigneras-prop} defined an algebra over $\Z[q_s]$.
Under the specialization $q_s\mapsto \#(I(1)sI(1)/I(1))$ and a suitable choice of $c_{\widetilde{s}}$ (which will be explained later), this algebra becomes $\mathcal{H}(\alggrp{G}(F),I(1))$.
Let $\mathcal{H}_\Z$ be the scalar extension of this algebra to $\Z[q_s^{1/2}]$.
This is the main object of this paper.
Later, we will put $\mathcal{H} = \mathcal{H}_\Z\otimes_\Z C$ for a field $C$ of characteristic $p$.
This is a $C[q_s^{1/2}]$-algebra.

As a $\Z[q_s^{1/2}]$-module, $\mathcal{H}_\Z$ is free and has a basis indexed by $\widetilde{W}(1)$.
The basis is denoted by $\{T_{\widetilde{w}}\}_{\widetilde{w}\in\widetilde{W}(1)}$, namely we have $\mathcal{H}_\Z = \bigoplus_{\widetilde{w}\in\widetilde{W}(1)}\Z[q_s^{1/2}]T_{\widetilde{w}}$.
The multiplications in $\mathcal{H}_\Z$ is described as follows.
Define $\ell\colon \widetilde{W}(1)\to \Z_{\ge 0}$ by the composition $\widetilde{W}(1)\to \widetilde{W}\xrightarrow{\ell}\Z_{\ge 0}$.
The multiplication in $\mathcal{H}$ is defined as follows~\cite[4.3]{Vigneras-prop}.
\begin{itemize}
\item the braid relations: $T_{\widetilde{w}_1}T_{\widetilde{w}_2} = T_{\widetilde{w}_1\widetilde{w}_2}$ if $\ell(\widetilde{w}_1) + \ell(\widetilde{w}_2) = \ell(\widetilde{w}_1\widetilde{w}_2)$.
\item the quadratic relations: $T_{\widetilde{s}}^2 = q_sT_{\widetilde{s}} + c_{\widetilde{s}}T_{\widetilde{s}}$ for $\widetilde{s}\in \widetilde{W}(1)$ which is a lift of $s\in S_\aff$.
\end{itemize}
In particular, $Z_\kappa\ni t\mapsto T_t\in \mathcal{H}_\Z$ gives an embedding $\Z[q_s^{1/2}][Z_\kappa]\hookrightarrow \mathcal{H}_\Z$.

It is convenient to fix a representative $n_s$ of $s\in S_\aff$ in $\widetilde{W}(1)$ as follows.
Consider the apartment attached to $\alggrp{S}$.
Then the elements in the apartment fixed by $s$ is a hyperplane.
Take a facet $\mathcal{F}$ contained in this hyperplane whose codimension in $V$ is one and consider the parahoric subgroup $K_\mathcal{F}$ attached to this facet.
Then we take $n_s$ from the group generated by $K_\mathcal{F}\cap \alggrp{U}(F)$ and $K_\mathcal{F}\cap \alggrp{\overline{U}}(F)$.
Moreover, we can (and do) take $n_s$ such that $\{n_s\}_{s\in S}$ satisfies the braid relations~\cite{MR0206117}. (See also \cite[Proposition~3.4]{Vigneras-prop}.)
Hence we define $n_w$ for $w\in W$ via a reduced expression of $w$.

Let $G'$ be the group generated by $\bigcup_{\alpha\in\Sigma}G'_\alpha$.
It is also the group generated by $\alggrp{U}(F)$ and $\overline{\alggrp{U}}(F)$.
Then $\widetilde{W}_\aff$ is the image of $N_{\alggrp{G}(F)}(\alggrp{S}(F))\cap G'$.
Let $\widetilde{W}_\aff(1)$ be the image of $N_{\alggrp{G}(F)}(\alggrp{S}(F))\cap G'$ in $\widetilde{W}(1)$. (This group is different from the group denoted by $W^\aff(1)$ in \cite{Vigneras-prop}.)
This is a normal subgroup of $\widetilde{W}(1)$ since $G'$ is normal in $\alggrp{G}(F)$.
Put $\Lambda'(1) = \widetilde{W}_\aff(1)\cap \Lambda(1)$.
This is the image of $G'\cap \alggrp{Z}(F)$ and it is again normal in $\widetilde{W}(1)$ since it is an intersection of normal subgroups.
We introduce the Bruhat order on $\widetilde{W}(1)$ as follows.
Let $\widetilde{w}_1,\widetilde{w}_2\in \widetilde{W}(1)$ and denote its image in $\widetilde{W}$ by $w_1,w_2$.
Then $\widetilde{w}_1 < \widetilde{w}_2$ if and only if $w_1 < w_2$ and $\widetilde{w}_1\in \widetilde{w}_2\widetilde{W}_\aff(1)$.
Here the order in $\widetilde{W}$ is the usual Bruhat order.
Assume that for $\widetilde{s}\in \widetilde{W}_\aff(1)$ which is a lift of a simple reflection, we have $c_{\widetilde{s}}\in C[Z_\kappa\cap \widetilde{W}_\aff(1)]$. (This assumption is satisfied if $c_{\widetilde{s}}$ comes from the group, see subsection~\ref{subsec:Bernstein basis}.)
Then we have $T_{\widetilde{w}_1}T_{\widetilde{w}_2}\in \sum_{\widetilde{v}_1\le \widetilde{w}_1,\widetilde{v}_2\le\widetilde{w}_2,\widetilde{v} \le \widetilde{v}_1\widetilde{v}_2}\Z[q_s^{1/2}]T_{\widetilde{v}}$.
\begin{rem}
In \cite[5.3]{Vigneras-prop}, the order is defined as follows: let $\widetilde{w}_1,\widetilde{w}_2\in\widetilde{W}(1)$ with the image $w_1,w_2$ in $\widetilde{W}$ respectively, then $\widetilde{w}_1 < \widetilde{w}_2$ if and only if $w_1 < w_2$.
If $\widetilde{w}_1 < \widetilde{w}_2$ in the sense of this paper, then $\widetilde{w}_1 < \widetilde{w}_2$ in the sense of Vign\'eras.
However the converse is not true.
\end{rem}

\subsection{Properties of $\widetilde{W}_\aff(1)$}
The subgroup $\widetilde{W}_\aff(1)$ has the similar properties to $\widetilde{W}_\aff$.
The first property is the Bruhat decomposition.
\begin{lem}
We have $G'(\alggrp{Z}(F)\cap I(1)) = \coprod_{\widetilde{w}\in\widetilde{W}_\aff(1)}I(1)\widetilde{w}I(1)$.
\end{lem}
\begin{proof}
Put $N = N_{\alggrp{G}(F)}(\alggrp{S}(F))$.
We have $\alggrp{G}(F) = \coprod_{\widetilde{w}\in\widetilde{W}(1)}I(1)\widetilde{w}I(1)$~\cite[Proposition~3.35]{Vigneras-prop}.
Let $g\in G'(\alggrp{Z}(F)\cap I(1))$ and take $n\in N$ such that $g\in I(1)nI(1)$.
We have the decomposition $I(1) = (\overline{\alggrp{U}}(F)\cap I(1))(\alggrp{Z}(F)\cap I(1))(\alggrp{U}(F)\cap I(1))$.
Since $\alggrp{U}(F),\overline{\alggrp{U}}(F)\subset G'$, $I(1)$ is contained in $G'(\alggrp{Z}(F)\cap I(1))$.
Hence we have $n\in G'(\alggrp{Z}(F)\cap I(1))\cap N$.
Since $\alggrp{Z}(F)\cap I(1)\subset N$, we have $G'(\alggrp{Z}(F)\cap I(1))\cap N = (G'\cap N)(\alggrp{Z}(F)\cap I(1))$.
Therefore the image of $n$ in $\widetilde{W}(1)$ is in $\widetilde{W}_\aff(1)$.
\end{proof}

\begin{lem}[{\cite[III.16.\ Proposition]{arXiv:1412.0737}}]\label{lem:structure of Lambda'_s}
For $s = s_\alpha \in S$, we have an exact sequence
\[
	1\to \Lambda'_s(1)\cap Z_\kappa\to \Lambda'_s(1)\xrightarrow{\nu}\Z\check{\alpha}\to 0.
\]
\end{lem}
%\begin{proof}
%Let $\lambda\in \Lambda'_s(1)$ such that $\nu(\lambda) = 0$.
%Take a representative $z\in \alggrp{Z}(F)\cap G'_s$ of $\lambda$.
%Then we have $z\in \widetilde{K}$ where $\widetilde{K}$ is the maximal compact subgroup corresponding to our fixed special point.
%Let $w_G$ be the Kottwitz map.
%Since $w_G$ is trivial on $\alggrp{U}_\alpha(F)$ and $\alggrp{U}_{-\alpha}(F)$, it is also trivial on $G'_s$.
%Hence $z\in \widetilde{K}\cap \Ker w_G = K$.
%%%%%%%%%%%%%%%%%%%%%%%%%%%%%%%%%%%%%%%%%%%%%%%%%%%%%%%%%%%%%%
%\end{proof}

We have the generators of $\widetilde{W}_\aff(1)$.
\begin{lem}
The subgroup $\widetilde{W}_\aff(1)$ of $\widetilde{W}(1)$ is generated by $\bigcup_{s\in S}\Lambda_s'(1)$ and $\{n_s\mid s\in S\}$.
\end{lem}
\begin{proof}
Let $\widetilde{W}'_s(1)$ be the subgroup of $\widetilde{W}(1)$ generated by $\Lambda_s'(1)$ and $n_s$.
Then we have $G'_s  \subset I(1)\widetilde{W}_{s}'(1)I(1)$ by the Bruhat decomposition of $G'_s$.
Let $\widetilde{W}'(1)$ be the group generated by $\bigcup_{s\in S}\Lambda_s'(1)$ and $\{n_s\mid s\in S\}$.
Since $G'$ is generated by $\bigcup_{s\in S}G'_s$, we have $G'\subset I(1)\widetilde{W}'(1)I(1)$.
Comparing the Bruhat decomposition of $G'$, we get $\widetilde{W}_\aff(1)\subset \widetilde{W}'(1)$.
\end{proof}
Using this lemma, we get the following property on $\Lambda'(1)$.
\begin{lem}\label{lem:generators of Lambda'}
The subgroup $\Lambda'(1)$ of $\Lambda(1)$ is generated by $\bigcup_{s\in S}\Lambda'_s(1)$ and $\Lambda'(1)\cap Z_\kappa$ is generated by $\bigcup_{s\in S}(\Lambda'_s(1)\cap Z_\kappa)$.
\end{lem}
\begin{proof}
Let $\Lambda'(1)_0$ be the subgroup generated by $\bigcup_{s\in S}\Lambda'_s(1)$.
We prove that this subgroup is stable under the action of $n_w$ for $w\in W$.
It is sufficient to prove that $n_s\lambda n_s^{-1}\in \Lambda_s'(1)\lambda$ for $\lambda\in\Lambda(1)$ and $s\in S$.
Let $z\in \alggrp{Z}(F)$ be a representative of $\lambda$.
We have $n_s zn_s^{-1}z^{-1} = (n_szn_s^{-1})z^{-1}\in \alggrp{Z}(F)$.
Since $G'_s$ is normalized by $\alggrp{Z}(F)$, we have $n_szn_s^{-1}z^{-1} = n_s(zn_s^{-1}z^{-1})\in G'_s$.
Hence $n_szn_s^{-1}z^{-1}\in G'_s\cap \alggrp{Z}(F)$.
Therefore we have $n_s\lambda n_s^{-1}\lambda^{-1}\in \Lambda'_s(1)$.
Hence by the above lemma, any element in $\widetilde{W}_\aff(1)$ has the form $\lambda n_{s_1}\dotsm n_{s_l}$ where $\lambda\in \Lambda'(1)_0$ and $s_i\in S$.
Take the minimal $i$ such that $s_1\dotsm s_i < s_1\dotsm s_{i - 1}$.
Put $w = s_1\dotsm s_{i - 1}$.
Then $n_w = n_{ws_i}n_{s_i}$ and we have $\lambda n_{s_1}\dotsm n_{s_l} = \lambda n_{ws_i}n_{s_i}^2n_{s_{i + 1}}\dotsm n_{s_l} = \lambda (n_{ws_i}n_{s_i}^2n_{ws_i}^{-1})n_{ws_i}n_{s_{i + 1}}\dotsm n_{s_l}\in \Lambda'(1)_0n_{ws_i}n_{s_{i + 1}}\dotsm n_{s_l}$.
Therefore, we can take $\lambda$ and $s_1,\dots,s_l$ such that $s_1\dotsm s_l$ is a reduced expression.
Hence $\lambda n_{s_1}\dotsm n_{s_l} = \lambda n_{s_1\dotsm s_l}$.
It is in $\Lambda(1)$ if and only if $s_1\dotsm s_l = 1$.
Therefore, $\Lambda'(1) = \widetilde{W}_\aff(1)\cap \Lambda(1)$ is $\Lambda'(1)_0$.

Let $s_1,\dots,s_n\in S$ such that $S = \{s_1,\dots,s_n\}$ where $n = \#S$.
Since $\Lambda'_{s_i}(1)$ is normal in $\Lambda(1)$, we have $\Lambda'(1) = \Lambda'_{s_1}(1)\dotsm \Lambda'_{s_n}(1)$.
Let $\lambda\in \Lambda'(1)\cap Z_\kappa$ and take $\lambda_i\in \Lambda'_{s_i}(1)$ such that $\lambda = \lambda_1\dotsm \lambda_n$.
Then $0 = \nu(\lambda) = \sum_i\nu(\lambda_i)$.
Take $\alpha_i\in \Delta$ such that $s_i = s_{\alpha_i}$.
Then $\nu(\lambda_i)\in \Z\check{\alpha_i}$.
Since $\{\check{\alpha}_1,\dots,\check{\alpha}_n\}$ is linearly independent, we have $\nu(\lambda_i) = 0$.
Hence $\lambda_i\in \Lambda'_{s_i}(1)\cap Z_\kappa$ by Lemma~\ref{lem:structure of Lambda'_s}.
\end{proof}

\subsection{Bernstein basis}
\label{subsec:Bernstein basis}
For investigating the representations of pro-$p$-Iwahori Hecke algebra, another basis, called \emph{Bernstein basis} is important.

Let $\Delta'\subset \Sigma$ be a set of simple roots corresponding to a positive system of $\Sigma$.
Vign\'eras~\cite{Vigneras-prop} defined a $\Z[q_s^{1/2}]$-basis $\{E_{\Delta'}(\widetilde{w})\}_{\widetilde{w}\in \widetilde{W}(1)}$ of $\mathcal{H}_\Z$.
If $c_{\widetilde{s}}\in C[Z_\kappa\cap \widetilde{W}_\aff(1)]$ for any $\widetilde{s}\in \widetilde{W}_\aff(1)$ which is a lift of a simple affine reflection, then $E_{\Delta'}(\widetilde{w})\in T_{\widetilde{w}} + \sum_{\widetilde{v} < \widetilde{w}} \Z[q_s]T_{\widetilde{v}}$. (The same proof for \cite[Corollary~5.6]{Vigneras-prop} is applicable.)
\begin{rem}
In fact, Vign\'eras defined a basis $\{E_o(\widetilde{w})\mid \widetilde{w}\in \widetilde{W}(1)\}$ for any orientation $o$.
we can attache a spherical orientation $o_{\Delta'}$ \cite[Definition~5.16]{Vigneras-prop} and one can consider a basis $\{E_{o_{\Delta'}}(\widetilde{w})\}$.
We put $E_{\Delta'}(\widetilde{w}) = E_{o_{\Delta'}}(\widetilde{w})$.
\end{rem}
Here are some examples of $E_{\Delta'}(\widetilde{w})$.
(See \cite[5.3]{Vigneras-prop}.)
For $\widetilde{w}\in \widetilde{W}(1)$, $T_{\widetilde{w}^{-1}}$ is invertible in $\mathcal{H}[q_s^{\pm 1/2}]$ and $q_{\widetilde{w}}T_{\widetilde{w}^{-1}}^{-1}\in \mathcal{H}$.
We denote this element by $T_{\widetilde{w}}^*$.
\begin{prop}\label{prop:example of E}
Let $\Delta'$ be a set of simple roots.
\begin{enumerate}
\item If $\lambda\in \Lambda(1)$ is dominant with respect to $\Delta'$, then $E_{\Delta'}(\lambda) = T_\lambda$.
\item If $\lambda\in \Lambda(1)$ is anti-dominant with respect to $\Delta'$, then $E_{\Delta'}(\lambda) = T_\lambda^*$.
\item Let $w\in W$.
We have $E_{-\Delta}(n_w) = T_{n_w}$ and $E_{\Delta}(n_w) = E_{w(-\Delta)}(n_w) = T_{n_w}^*$.
In particular, for $s\in S$, $E_{\Delta}(n_s) = E_{s(-\Delta)}(n_s) = T_{n_s}^* = T_{n_s} - c_{n_s}$.
\end{enumerate}
\end{prop}
For $\widetilde{w}\in \widetilde{W}(1)$, put $\widetilde{w}(\Delta') = w(\Delta')$ for the set of simple roots $\Delta'$ attached to a positive system where $w\in W$ is the image of $\widetilde{w}$ in $W$.

\begin{prop}[{\cite[Theorem~5.25]{Vigneras-prop}}]\label{prop:multiplication formula}
Let $\widetilde{w},\widetilde{w}'\in \widetilde{W}(1)$.
Then we have
\[
	E_{\Delta'}(\widetilde{w}\widetilde{w}') = q_{\widetilde{w}}^{-1/2}q_{\widetilde{w}'}^{-1/2}q_{\widetilde{w}\widetilde{w}'}^{1/2}E_{\Delta'}(\widetilde{w})E_{\widetilde{w}^{-1}(\Delta')}(\widetilde{w}').
\]
\end{prop}
In particular, if $\lambda_1,\lambda_2\in \Lambda(1)$, then $E_{\Delta'}(\lambda_1)E_{\Delta'}(\lambda_2) = q_{\lambda_1}^{1/2}q_{\lambda_2}^{1/2}q_{\lambda_1\lambda_2}^{-1/2}E_{\Delta'}(\lambda_1\lambda_2)$.
By \eqref{eq:length formula}, we have $\ell(\lambda_1\lambda_2^{-1}\lambda_1^{-1}\lambda_2) = 0$.
Therefore we have $q_{\lambda_1\lambda_2^{-1}\lambda_1^{-1}\lambda_2} = 1$ and $q_{\lambda_1\lambda_2\lambda_1^{-1}}q_{\lambda_1\lambda_2^{-1}\lambda_1^{-1}\lambda_2} = q_{\lambda_2}$.
Namely $q_{\lambda_1\lambda_2\lambda_1^{-1}} = q_{\lambda_2}$.
Hence we get
\begin{equation}\label{eq:E(lambda_1) and E(lambda_2)}
\begin{split}
E_{\Delta'}(\lambda_1)E_{\Delta'}(\lambda_2) & = q_{\lambda_1}^{1/2}q_{\lambda_1\lambda_2\lambda_1^{-1}}^{1/2}q_{\lambda_1\lambda_2}^{-1/2} E_{\Delta'}((\lambda_1\lambda_2\lambda_1^{-1})\lambda_1)\\ & = E_{\Delta'}(\lambda_1\lambda_2\lambda_1^{-1})E_{\Delta'}(\lambda_1).
\end{split}
\end{equation}

In this paper, we use $E_{-\Delta}(\widetilde{w})$ mainly.
Set $E(\widetilde{w}) = E_{-\Delta}(\widetilde{w})$.
Using the above properties, we have the following description of $E(\widetilde{w})$. (One can regard this description as a definition of $E(\widetilde{w})$.)
Take anti-dominant $\lambda_1,\lambda_2\in \Lambda(1)$ and $w\in W$ such that $\widetilde{w} = \lambda_1\lambda_2^{-1} n_w$.
Then we have
\[
	E(\lambda_1\lambda_2^{-1}n_w) = q_{\lambda_1\lambda_2^{-1}n_w}^{1/2}q_{\lambda_1}^{-1/2}q_{\lambda_2}^{1/2}q_{n_w}^{-1/2}T_{\lambda_1}T_{\lambda_2}^{-1}T_{n_w}
\]
Put $\theta(\lambda) = q_\lambda^{-1/2}E(\lambda)\in \mathcal{H}_\Z[q_s^{\pm 1/2}]$ for $\lambda\in \Lambda(1)$.
Then for anti-dominant $\lambda\in \Lambda(1)$, we have $\theta(\lambda) = q_\lambda^{-1/2}T_\lambda$ and $\theta(\lambda_1 \lambda_2^{-1}) = \theta(\lambda_1)\theta(\lambda_2)^{-1}$ for any $\lambda_1,\lambda_2\in \Lambda(1)$.
We have
\[
	E(\lambda n_w) = q_{\lambda n_w}^{1/2}q_{n_w}^{-1/2}\theta(\lambda) T_{n_w}.
\]
The element $\theta(\lambda)$ is denoted by $\widetilde{E}_{o_{-\Delta}}(\lambda)$ in \cite[Lemma~5.44]{Vigneras-prop}.

\subsection{Bernstein relations}\label{subsec:Bernstein relations}
Let $\widetilde{\alpha}\in \Sigma_\aff$ be an affine root.
Consider the apartment attached to $\alggrp{S}$ and take a facet $\mathcal{F}$ of codimension one in $V$ contained in the hyperplane on which $\widetilde{\alpha}$ is null.
Let $K_\mathcal{F}$ be the parahoric subgroup attached to $\mathcal{F}$.
Then it defines a connected reductive subgroup over $\kappa$.
Let $G_{\mathcal{F},\kappa}$ be the group of $\kappa$-valued points of this reductive group, $U_{\mathcal{F},\kappa}$ (resp.~$\overline{U}_{\mathcal{F},\kappa})$ the image of $K_{\mathcal{F}}\cap \alggrp{U}(F)$ (resp.\ $K_{\mathcal{F}}\cap \overline{\alggrp{U}}(F)$).
Define $Z_{\widetilde{\alpha},\kappa}$ as the intersection of $Z_\kappa$ with the group generated by $U_{\mathcal{F},\kappa}$ and $\overline{U}_{\mathcal{F},\kappa}$.
\begin{lem}\label{lem:Z_(a,k) is canonical}
The subgroup $Z_{\widetilde{\alpha},\kappa}\subset Z_\kappa$ does not depend on a choice of $\mathcal{F}$.
\end{lem}
\begin{proof}
Let $\alpha\in \Sigma$ be a root corresponding to $\widetilde{\alpha}$, $U_{\widetilde{\alpha},\kappa}$ (resp.\ $\overline{U}_{\widetilde{\alpha},\kappa}$) the image of $\alggrp{U}_\alpha(F)\cap K_\mathcal{F}$ (resp.\ $\alggrp{U}_{-\alpha}(F)\cap K_\mathcal{F}$) in $G_{\mathcal{F},\kappa}$.
The root system of $G_{\mathcal{F},\kappa}$ is $\{\pm \widetilde{\alpha}\}$.
Hence $Z_{\widetilde{\alpha},\kappa}$ is the intersection of $Z_\kappa$ with the group generated by $U_{\widetilde{\alpha},\kappa}$ and $\overline{U}_{\widetilde{\alpha},\kappa}$.
Set $r = \min\{n\in\Z\mid \text{$\alpha(x) + n\ge 0$ for all $x\in \mathcal{F}$}\}$.
We have $\alggrp{U}_\alpha(F)\cap K_\mathcal{F} = U_{(\alpha,r)}$ \cite[(45)]{Vigneras-prop}.
Since $\alpha(x)$ only depends on $\widetilde{\alpha}$, $r$ does not depend on a choice of $\mathcal{F}$.
Hence $K_\mathcal{F}\cap \alggrp{U}_\alpha(F)$ does not depend on $\mathcal{F}$.
\end{proof}

For $s\in S_\aff$, take a simple affine root $\widetilde{\alpha}$ such that $s = s_{\widetilde{\alpha}}$.
We take $c_{n_s}$ coming from the Hecke algebra attached to $(G,I(1))$ \cite[4.2]{Vigneras-prop}.
We have $c_{n_s}\in\Z_{\ge 0}[Z_{\widetilde{\alpha},\kappa}]$ and
\[
	c_{n_s} \equiv -(\#Z_{\widetilde{\alpha},\kappa})^{-1}\sum_{t\in Z_{\widetilde{\alpha},\kappa}}T_t\pmod{p}.
\]
In particular, $c_{n_s}\pmod {p}$ does not depend on a choice of $n_s$.
It satisfies $n_s\cdot c_{n_s} = c_{n_s}$.

Put $n_w(\lambda) = n_w\lambda n_w^{-1}$ for $w\in W$.
The following lemma is a reformulation of the Bernstein relations of Vign\'eras~\cite[Theorem~5.38]{Vigneras-prop}.
We refer this lemma as Bernstein relations in this paper.
\begin{lem}\label{lem:Bernstein relations}
Let $\lambda\in \Lambda(1)$, $s = s_\alpha\in S$.
There exist $\mu_{n_s}(k)\in \Lambda'_s(1)$ and $c_{n_s,k}\in \Z[Z_\kappa]$ such that, if $\langle\nu(\lambda),\alpha\rangle \ge 0$, then
\begin{align*}
\theta(n_s(\lambda))T_{n_s} - T_{n_s}\theta(\lambda)
& = \sum_{k = 0}^{\langle \nu(\lambda),\alpha\rangle - 1}\theta(n_s(\lambda)\mu_{n_s}(k))c_{n_s,k}\\
& = \sum_{k = 1}^{\langle \nu(\lambda),\alpha\rangle}c_{n_s,k}\theta(\mu_{n_s}(-k)\lambda).
\end{align*}
If $\langle\nu(\lambda),\alpha\rangle < 0$, then
\begin{align*}
\theta(n_s(\lambda))T_{n_s} - T_{n_s}\theta(\lambda)
& = -\sum_{k = 0}^{-\langle\nu(\lambda),\alpha\rangle-1}c_{n_s,-k}\theta(\mu_{n_s}(k)\lambda)\\
& = -\sum_{k = 1}^{-\langle\nu(\lambda),\alpha\rangle}\theta(n_s(\lambda)\mu_{n_s}(-k))c_{n_s,-k}.
\end{align*}
These elements satisfy the following.
We have $c_{n_s,k} \in \Z[Z_{(\alpha,k),\kappa}]$ and
\[
	c_{n_s,k} \equiv -(\#Z_{(\alpha,k),\kappa})^{-1}\sum_{t\in Z_{(\alpha,k),\kappa}}T_t\pmod{p}.
\]
%The cardinality of $Z_{(\alpha,k),\kappa}$ only depends on $k\pmod{2}$.
We have $\mu_{n_s}(0) = 1$, $\nu(\mu_{n_s}(k)) = k\check{\alpha}$ for $k\in\Z$, $n_s^{-1}(\mu_{n_s}(k)) = \mu_{n_s}(-k)$ for $k\in\Z_{\ge 0}$, $c_{n_s,0} = c_{n_s}$ and $\mu_{n_s}(-1)^{-1}\cdot c_{{n_s},1} = c_{{n_s},-1}$.
\end{lem}
\begin{proof}
First assume that $\langle\nu(\Lambda(1)),\alpha\rangle = \Z$.
Let $\lambda_s\in\Lambda(1)$ such that $\langle\nu(\lambda_s),\alpha\rangle = -1$.
Then by \cite[Theorem~5.38]{Vigneras-prop}, if $n = \langle\nu(\lambda),\alpha\rangle > 0$, we have
\[
	\theta(n_s(\lambda))T_{n_s} - T_{n_s}\theta(\lambda) = \sum_{k = 0}^{n - 1}\theta(n_s(\lambda)n_s(\lambda_s)^k\lambda_s^{-k})(\lambda_s^k\cdot c_{n_s}).
\]
Take $\lambda = \lambda_s^{-n}$ in this formula.
Then
\begin{equation}\label{eq:basic Bernstein relation, easy case}
\theta(n_s(\lambda_s)^{-n})T_{n_s} - T_{n_s}\theta(\lambda_s^{-n}) = \sum_{k = 0}^{n - 1}\theta(n_s(\lambda_s)^{-n + k}\lambda_s^{-k})(\lambda_s^k\cdot c_{n_s})
\end{equation}
We have $\theta(n_s(\lambda_s)^{-n + k}\lambda_s^{-k})(\lambda_s^k\cdot c_{n_s}) = (n_s(\lambda_s)^{-n + k}\cdot c_{n_s})\theta(n_s(\lambda_s)^{-n + k}\lambda_s^{-k})$.
Hence replacing $k$ with $n - k$, we have
\begin{equation}\label{eq:basic Bernstein relation, easy case2}
\theta(n_s(\lambda_s)^{-n})T_{n_s} - T_{n_s}\theta(\lambda_s^{-n}) = \sum_{k = 1}^n(n_s(\lambda_s)^{-k}\cdot c_{n_s})\theta(n_s(\lambda_s)^{-k}\lambda_s^{-n+k}).
\end{equation}
Since $\langle \nu(\lambda_s^n\lambda),\alpha\rangle = 0$, we have $T_{n_s}\theta(\lambda_s^n\lambda) = \theta(n_s(\lambda_s^n\lambda))T_{n_s}$ \cite[Lemma~5.34, 5.35]{Vigneras-prop}.
Hence multiplying by $\theta(\lambda_s^n\lambda)$ on the right, we get
\begin{equation}\label{eq:Bernstein relation, easy case, replacement target 1-1}
\theta(n_s(\lambda))T_{n_s} - T_{n_s}\theta(\lambda) =
\sum_{k = 1}^n(n_s(\lambda_s)^{-k}\cdot c_{n_s})\theta(n_s(\lambda_s)^{-k}\lambda_s^{k}\lambda).
\end{equation}
Multiply by $\theta(n_s(\lambda_s)^n)$ on the left and by $\theta(\lambda_s^n)$ on the right to \eqref{eq:basic Bernstein relation, easy case} and \eqref{eq:basic Bernstein relation, easy case2}.
From \eqref{eq:basic Bernstein relation, easy case}, we have
\begin{align*}
T_{n_s}\theta(\lambda_s^n) - \theta(n_s(\lambda_s)^n)T_{n_s}
& =
\sum_{k = 0}^{n - 1}\theta(n_s(\lambda_s)^k\lambda_s^{-k})(\lambda_s^k\cdot c_{n_s})\theta(\lambda_s^n)\\
& = \sum_{k = 0}^{n - 1}(n_s(\lambda_s)^k\cdot c_{n_s})\theta(n_s(\lambda_s)^k\lambda_s^{n-k}).
\end{align*}
Multiplying by $\theta(\lambda_s^{-n}\lambda)$ on the right, if $-n = \langle\nu(\lambda),\alpha\rangle < 0$, then
\begin{equation*}
T_{n_s}\theta(\lambda) - \theta(n_s(\lambda))T_{n_s} = \sum_{k = 0}^{n - 1}(n_s(\lambda_s)^{k}\cdot c_{n_s})\theta(n_s(\lambda_s)^{k}\lambda_s^{-k}\lambda)\\
\end{equation*}
Similarly, from \eqref{eq:basic Bernstein relation, easy case2}, we have
\[
T_{n_s}\theta(\lambda_s^{n}) - \theta(n_s(\lambda_s)^n)T_{n_s} = \sum_{k = 1}^n\theta(n_s(\lambda_s)^{n-k}\lambda_s^{k})(\lambda_s^{-k}\cdot c_{n_s}).
\]
and, if $-n = \langle\nu(\lambda),\alpha\rangle < 0$, then
\begin{equation}\label{eq:Bernstein relation, easy case, replacement target 1-2}
T_{n_s}\theta(\lambda) - \theta(n_s(\lambda))T_{n_s} = \sum_{k = 1}^n\theta(n_s(\lambda)n_s(\lambda_s)^{-k}\lambda_s^{k})(\lambda_s^{-k}\cdot c_{n_s}).
\end{equation}
In \eqref{eq:Bernstein relation, easy case, replacement target 1-1} and \eqref{eq:Bernstein relation, easy case, replacement target 1-2}, we may replace $\lambda_s$ with $n_s^{-1}(\lambda_s)^{-1}$ since $\langle\nu(n_s^{-1}(\lambda_s)^{-1}),\alpha\rangle = -1$.
We get, if $n = \langle \nu(\lambda),\alpha\rangle > 0$, then
\[
	\theta(n_s(\lambda))T_{n_s} - T_{n_s}\theta(\lambda) = \sum_{k = 1}^n(\lambda_s^{k}\cdot c_{n_s})\theta(\lambda_s^{k}n_s^{-1}(\lambda_s)^{-k}\lambda)
\]
and, if $-n = \langle\nu(\lambda),\alpha\rangle < 0$, then
\[
	T_{n_s}\theta(\lambda) - \theta(n_s(\lambda))T_{n_s} =
	\sum_{k = 1}^n\theta(n_s(\lambda)\lambda_s^{k}n_s^{-1}(\lambda_s)^{-k})(n_s^{-1}(\lambda_s)^{k}\cdot c_{n_s}).
\]
Now, for $k\in\Z_{\ge 0}$, put
\begin{align*}
\mu_{n_s}(k) & = n_s(\lambda_s)^k\lambda_s^{-k},& \mu_{n_s}(-k) & = \lambda_s^kn_s^{-1}(\lambda_s)^{-k},\\
c_{n_s,k} &  = \lambda_s^k\cdot c_{n_s}, & c_{n_s,-k} & = n_s(\lambda_s)^k\cdot c_{n_s}.
\end{align*}
Then if $k\in\Z_{\ge 0}$, $c_{n_s,k}\in \Z[\lambda_s^k\cdot Z_{(\alpha,0),\kappa}] = \Z[Z_{\lambda_s^k(\alpha,0),\kappa}] = \Z[Z_{(\alpha,k),\kappa}]$.
Similarly, we have $c_{n_s,-k}\in\Z[Z_{(\alpha,-k),\kappa}]$.
Since $n_s^2\in Z_\kappa$ and $Z_\kappa$ is commutative, $n_s^{-1}(\lambda_s)^k\cdot c_{n_s} = n_s^{-2}\cdot (n_s(\lambda_s)^k\cdot (n_s^2\cdot c_{n_s})) = n_s(\lambda_s)^k\cdot c_{n_s} = c_{n_s,-k}$.
Therefore these elements satisfy the Bernstein relations.
To check the other conditions is easy.

Next assume that $\langle\nu(\Lambda(1)),\alpha\rangle = 2\Z$.
Let $s'$ be the reflection corresponding to the highest root in the component of $\Sigma$ containing $s$, $\alpha'$ the highest root and $\widetilde{\alpha}' = (\alpha',0)\in \Sigma_\aff$.
Take $\widetilde{w}\in \widetilde{W}(1)$ such that $\lambda_s = n_s\widetilde{w}n_{s'}\widetilde{w}^{-1}\in \Lambda(1)$ satisfies $\ell(\lambda_s) = 2\ell(\widetilde{w}) + 2$ and $\nu(\lambda_s) = -\check{\alpha}$~\cite[Lemma~5.15]{Vigneras-prop}.
Put $c' = \widetilde{w}\cdot c_{n_{s'}}$.
By the condition, in $\widetilde{W}$, we have $s_{\widetilde{w}(\widetilde{\alpha}')} = s \lambda_s = s_{(\alpha,-1)}$
Hence we have $\widetilde{w}(\widetilde{\alpha}') = (\alpha,-1)$.

By similar arguments using \cite[Theorem~5.38]{Vigneras-prop}, if $2n = \langle \nu(\lambda),\alpha\rangle \ge 0$, then $\theta(n_s(\lambda))T_{n_s} - T_{n_s}\theta(\lambda)$ is equal to
\begin{equation*}
\sum_{k = 0}^{n - 1}\theta(n_s(\lambda))(\theta(n_s(\lambda_s)^kn_s^2\lambda_s^{-k-1})(\lambda_s^{k + 1}\cdot c') + \theta(n_s(\lambda_s)^k\lambda_s^{-k})(\lambda_s^k\cdot c_{n_s}))
\end{equation*}
and
\begin{equation}\label{eq:Bernstein relation, difficult case, replacement target 1}
\sum_{k = 1}^n ((n_s(\lambda_s)^{-k}\cdot c')\theta(n_s(\lambda_s)^{-k}n_s^2\lambda_s^{k - 1}) + (n_s(\lambda_s)^{-k}\cdot c_{n_s})\theta(n_s(\lambda_s)^{-k}\lambda_s^k))\theta(\lambda).
\end{equation}
If $-2n = \langle \nu(\lambda),\alpha\rangle < 0$, $T_{n_s}\theta(\lambda) - \theta(n_s(\lambda))T_{n_s}$ is equal to
\begin{equation*}
\sum_{k = 0}^{n - 1}((n_s(\lambda_s)^k\cdot c')\theta(n_s(\lambda_s)^kn_s^2\lambda_s^{-k-1}) + (n_s(\lambda_s)^k\cdot c_{n_s})\theta(n_s(\lambda_s)^k\lambda_s^{-k}))\theta(\lambda)\\
\end{equation*}
and
\begin{equation}\label{eq:Bernstein relation, difficult case, replacement target 2}
\sum_{k = 1}^n\theta(n_s(\lambda))(\theta(n_s(\lambda_s)^{-k}n_s^2\lambda_s^{k - 1})(\lambda_s^{-k+1}\cdot c') + \theta(n_s(\lambda_s)^{-k}\lambda_s^k)(\lambda_s^{-k}\cdot c_{n_s})).
\end{equation}
Replace $n_s$ with $n_s^{-1}$ and $n_{s'}$ with $n_{s'}^{-1}$ in \eqref{eq:Bernstein relation, difficult case, replacement target 1}.
Then $\lambda_s$ becomes $n_s^{-1}(\lambda_s)^{-1}$, $c_{n_s}$ becomes $c_{n_s^{-1}} = n_s^{-2}c_{n_s}$ and $c'$ becomes $\widetilde{w}\cdot c_{n_{s'}^{-1}} = c'\widetilde{w}\cdot (n_{s'}^{-2}) = c'(\lambda_s^{-1}n_s\lambda_s^{-1}n_s)$.
Hence it gives
\begin{align*}
& \theta(n_s^{-1}(\lambda))T_{n_s^{-1}} - T_{n_s^{-1}}\theta(\lambda)\\
& = \sum_{k = 1}^n((n_s^{-2}(\lambda_s)^{k}\cdot (c'\lambda_s^{-1}n_s\lambda_s^{-1}n_s))\theta(n_s^{-2}(\lambda_s)^{k}n_s^{-2}n_s^{-1}(\lambda_s)^{-k + 1})\\
&\quad + (n_s^{-2}(\lambda_s)^{k}\cdot (n_s^{-2}c_{n_s}))\theta(n_s^{-2}(\lambda_s)^{k}n_s^{-1}(\lambda_s)^{-k}))\theta(\lambda).
\end{align*}
Since $\lambda_s^{-1}n_s\lambda_s^{-1}n_s = \widetilde{w}n_{s'}^{-2}\widetilde{w}\in Z_\kappa$, it commutes with $n_s^{-2}$.
Hence we have
\begin{align*}
&(n_s^{-2}(\lambda_s)^{k}\cdot (c'(\lambda_s^{-1}n_s\lambda_s^{-1}n_s)))\theta(n_s^{-2}(\lambda_s)^{k}n_s^{-2}n_s^{-1}(\lambda_s)^{-k + 1})\\
& = (n_s^{-2}(\lambda_s)^{k}\cdot c')\theta(n_s^{-2}(\lambda_s)^{k}(\lambda_s^{-1}n_s\lambda_s^{-1}n_s)n_s^{-2}n_s^{-1}(\lambda_s)^{-k + 1})\\
& = (n_s^{-2}(\lambda_s)^{k}\cdot c')\theta(n_s^{-2}(\lambda_s)^{k}n_s^{-2}(\lambda_s^{-1}n_s\lambda_s^{-1}n_s)n_s^{-1}(\lambda_s)^{-k + 1})\\
& = (n_s^{-2}(\lambda_s)^{k}\cdot c')\theta(n_s^{-2}(\lambda_s)^{k}(n_s^{-2}\lambda_s^{-1}n_s^2)(n_s^{-1}\lambda_s^{-1}n_s)n_s^{-1}(\lambda_s)^{-k + 1})\\
& = (n_s^{-2}(\lambda_s)^{k}\cdot c')\theta(n_s^{-2}(\lambda_s)^{k-1}n_s^{-1}(\lambda_s)^{-k}).
\end{align*}
Similarly, we have
\begin{align*}
&(n_s^{-2}(\lambda_s)^{k}\cdot (n_s^{-2}c_{n_s}))\theta(n_s^{-2}(\lambda_s)^{k}n_s^{-1}(\lambda_s)^{-k})\\
&=
(n_s^{-2}(\lambda_s)^{k}\cdot c_{n_s})\theta(n_s^{-2}(\lambda_s)^{k}n_s^{-2}n_s^{-1}(\lambda_s)^{-k}).
\end{align*}
As in the case of $\langle \nu(\Lambda(1)),\alpha\rangle = \Z$, since $Z_\kappa$ is commutative, we have $n_s^{-2}(\lambda_s)^{k}\cdot c' = \lambda_s^k\cdot c'$ and $n_s^{-2}(\lambda_s)^k\cdot c_{n_s} = \lambda_s^k\cdot c_{n_s}$.
Therefore $\theta(n_s^{-1}(\lambda))T_{n_s^{-1}} - T_{n_s^{-1}}\theta(\lambda)$ is equal to
\[
\sum_{k = 1}^n((\lambda_s^{k}\cdot c')\theta(n_s^{-2}(\lambda_s)^{k-1}n_s^{-1}(\lambda_s)^{-k})
 + (\lambda_s^{k}\cdot c_{n_s})\theta(n_s^{-2}(\lambda_s)^{k}n_s^{-2}n_s^{-1}(\lambda_s)^{-k}))\theta(\lambda).
\]
Multiply by $T_{n_s^2}$ on the left.
Since $n_s^{2}\in Z_\kappa$, $T_{n_s^2}\theta(n_s^{-1}(\lambda)) = \theta(n_s^2)\theta(n_s^{-1}(\lambda)) = \theta(n_s^2n_s^{-1}(\lambda)) = \theta(n_s(\lambda))\theta(n_s^2) = \theta(n_s(\lambda))T_{n_s^2}$.
We also remark that $T_{n_s^2}$ commutes with $\lambda_s^k\cdot c'$ and $\lambda_s^k\cdot c_{n_s}$ since everything is in $\Z[Z_\kappa]$.
Hence, if $2n = \langle \nu(\lambda),\alpha\rangle \ge 0$, then $\theta(n_s(\lambda))T_{n_s} - T_{n_s}\theta(\lambda)$ is equal to
\[
	\sum_{k = 1}^n((\lambda_s^{k}\cdot c')\theta(\lambda_s^{k-1}n_s^2n_s^{-1}(\lambda_s)^{-k}) + (\lambda_s^{k}\cdot c_{n_s})\theta(\lambda_s^{k}n_s^{-1}(\lambda_s)^{-k}))\theta(\lambda).
\]
The same argument for \eqref{eq:Bernstein relation, difficult case, replacement target 2} implies that, if $-2n = \langle \nu(\lambda),\alpha\rangle < 0$, then $T_{n_s}\theta(\lambda) - \theta(n_s(\lambda))T_{n_s}$ is equal to
\[
	\sum_{k = 1}^n\theta(n_s(\lambda))(\theta(\lambda_s^{k-1}n_s^2n_s^{-1}(\lambda_s)^{-k})(n_s(\lambda_s)^{k - 1}\cdot c') + \theta(\lambda_s^kn_s^{-1}(\lambda_s)^{-k})(n_s(\lambda_s)^k\cdot c_{n_s}).
\]
Put
\begin{align*}
\mu_{n_s}(2k) & = n_s(\lambda_s)^k\lambda_s^{-k}, & \mu_{n_s}(2k + 1) & = n_s(\lambda_s)^kn_s^2\lambda_s^{-k-1} & (k\ge 0),\\
\mu_{n_s}(-2k)& = \lambda_s^kn_s^{-1}(\lambda_s)^{-k}, & \mu_{n_s}(-2k + 1) & = \lambda_s^{k-1} n_s^2n_s^{-1}(\lambda_s)^{-k} & (k\ge 1),\\
c_{n_s,2k} & = \lambda_s^k\cdot c_{n_s}, & c_{n_s,2k+1} & = \lambda_s^{k + 1}\cdot c' & (k\ge 0),\\
c_{n_s,-2k} & = n_s(\lambda_s)^{k}\cdot c_{n_s}, & c_{n_s,-2k+1} & = n_s(\lambda_s)^{k - 1}\cdot c' & (k\ge 1).
\end{align*}
We have $\mu_{n_s}(-1) = n_s^2n_s^{-1}(\lambda_s)^{-1}$ and $c_{n_s,1} = \lambda_s\cdot c'$.
Hence $\mu_{n_s}(-1)^{-1}\cdot c_{n_s,1} = (n_s^{-1}\lambda_s n_s^{-1}\lambda_s)\cdot c'$.
We have $n_s^{-1}\lambda_s = \widetilde{w}n_{s'}\widetilde{w}^{-1}$.
Hence $n_s^{-1}\lambda_s n_s^{-1}\lambda_s = \widetilde{w}n_{s'}^2\widetilde{w}^{-1}\in Z_\kappa$.
Therefore $\mu_{n_s}(-1)^{-1}\cdot c_{n_s,1} = c' = c_{n_s,-1}$.
It is easy to check the other conditions.
\end{proof}

Since $n_s^{-1}$ is also a lift of $s\in S$ in $\widetilde{W}(1)$, we have the Bernstein relations for $n_s^{-1}$.
\begin{lem}\label{lem:Bernstein relation, n_s <-> n_s^-1}
Put $\mu_{n_s^{-1}}(k) = \mu_{n_s}(k)n_s^{-2}$ and $c_{n_s^{-1},k} = c_{n_s,k}$.
Then they satisfy the Bernstein relations for $n_s^{-1}$.
\end{lem}
\begin{proof}
Assume that $\langle\nu(\lambda),\alpha\rangle > 0$.
Replacing $\lambda$ with $n_s^{-2}(\lambda)$ in Lemma~\ref{lem:Bernstein relations}, we have
\begin{align*}
\theta(n_s^{-1}(\lambda))T_{n_s} - T_{n_s}\theta(n_s^{-2}(\lambda))
& = \sum_{k = 0}^{\langle \nu(\lambda),\alpha\rangle - 1}\theta(n_s^{-1}(\lambda)\mu_{n_s}(k))c_{n_s,k}\\
& = \sum_{k = 1}^{\langle \nu(\lambda),\alpha\rangle}c_{n_s,k}\theta(\mu_{n_s}(-k)n_s^{-2}(\lambda)).
\end{align*}
We have $\theta(n_s^{-2}(\lambda))\theta(n_s^{-2}) = \theta(n_s^{-2}\lambda) = T_{n_s^{-2}}\theta(\lambda)$.
Since $n_s^{-2},c_{n_s,k}\in C[Z_\kappa]$, these commute with each other.
Hence multiplying by $\theta(n_s^{-2}) = T_{n_s^{-2}}$ on the right of the both sides, we get
\begin{align*}
\theta(n_s^{-1}(\lambda))T_{n_s^{-1}} - T_{n_s^{-1}}\theta(\lambda)
& = \sum_{k = 0}^{\langle \nu(\lambda),\alpha\rangle - 1}\theta(n_s^{-1}(\lambda)\mu_{n_s}(k)n_s^{-2})c_{n_s,k}\\
& = \sum_{k = 1}^{\langle \nu(\lambda),\alpha\rangle}c_{n_s,k}\theta(\mu_{n_s}(-k)n_s^{-2}\lambda).
\end{align*}
The same argument implies, if $\langle\nu(\lambda),\alpha\rangle < 0$, we have
\begin{align*}
\theta(n_s^{-1}(\lambda))T_{n_s^{-1}} - T_{n_s^{-1}}\theta(\lambda)
& = -\sum_{k = 0}^{-\langle\nu(\lambda),\alpha\rangle-1}c_{n_s,-k}\theta(\mu_{n_s}(k)n_s^{-2}\lambda)\\
& = -\sum_{k = 1}^{-\langle\nu(\lambda),\alpha\rangle}\theta(n_s(\lambda)\mu_{n_s}(-k)n_s^{-2})c_{n_s,-k}.
\end{align*}
We get the lemma.
\end{proof}

Let $C$ be an algebraically closed field of characteristic $p$ and set $\mathcal{H} = \mathcal{H}_\Z\otimes_\Z C$.
This is a $C[q_s^{1/2}]$-algebra.
As an element of $\mathcal{H}$, $c_{n_s}$ and $c_{n_s,k}$ does not depend on a choice of $n_s$ and we denote it by $c_s$ and $c_{s,k}$.
The field $C$ is a $C[q_s^{1/2}]$-algebra via $q_s^{1/2}\mapsto 0$.
Then $\mathcal{H}\otimes_{C[q_s^{1/2}]} C$ is isomorphic to the Hecke algebra for $(\alggrp{G}(F),I(1))$ with coefficients in $C$.
The following lemma is useful for calculations in $\mathcal{H}\otimes_{C[q_s^{1/2}]}C$.
\begin{lem}\label{lem:lemma on length function}
Let $\lambda,\mu\in \Lambda(1)$ and $\alpha\in\Sigma^+$ such that $\langle\nu(\lambda),\alpha\rangle > 0$.
Assume that there exists $k\in\Z_{\ge 0}$ such that $k\le \langle\nu(\lambda),\alpha\rangle$ and $\nu(\mu) = \nu(\lambda) - k\check{\alpha}$.
Then we have $\ell(\lambda) \ge \ell(\mu) + 2\min\{k,\langle\nu(\lambda),\alpha\rangle - k\}$.
Moreover, equality holds if and only if $k = 0$ or $k = \langle \nu(\lambda),\alpha\rangle$ or $w(\alpha)$ is simple for some $w\in W$ such that $n_w(\lambda)$ is dominant.
\end{lem}
The elements appeared in the Bernstein relations (e.g.\ $\lambda\mu_{n_s}(-k)$) satisfies the condition of $\mu$ in this lemma.
\begin{proof}
Since $\ell(\mu) = \ell(n_s(\mu))$ and $\nu(n_s(\mu)) = \nu(\lambda) - (\langle \nu(\lambda),\alpha\rangle - k)\check{\alpha}$, replacing $\mu$ with $n_s(\mu)$ if necessary, we may assume $k\le \langle \nu(\lambda),\alpha\rangle /2$.
Take $v\in W$ such that $\nu(n_v(\mu))$ is dominant.
We may assume $v(\alpha) > 0$ since $\langle \nu(n_v(\mu)),v(\alpha)\rangle = \langle\nu(\mu),\alpha\rangle = \langle \nu(\lambda),\alpha\rangle - 2k \ge 0$.
Set $\rho = (1/2)\sum_{\beta\in\Sigma^+}\beta$.
Then we have
\begin{align*}
\ell(\mu) & = \ell(n_v(\mu))\\
& = \sum_{\gamma\in\Sigma^+}\lvert \langle \nu(n_v(\mu)),\gamma\rangle\rvert\\
& = \sum_{\gamma\in\Sigma^+}\langle \nu(n_v(\mu)),\gamma\rangle\\
& = \sum_{\gamma\in\Sigma^+}\langle \nu(n_v(\lambda)),\gamma\rangle -  k\sum_{\gamma\in\Sigma^+}\langle v(\check{\alpha}),\gamma\rangle\\
& \le \sum_{\gamma\in\Sigma^+}\lvert\langle \nu(n_v(\lambda)),\gamma\rangle\rvert - k\sum_{\gamma\in\Sigma^+}\langle v(\check{\alpha}),\gamma\rangle\\
& = \ell(n_v(\lambda)) -  2k\langle v(\check{\alpha}),\rho\rangle\\
& = \ell(\lambda) -  2k\langle v(\check{\alpha}),\rho\rangle
\end{align*}
Since $v(\check{\alpha})\in\Sigma^+$, we have $\langle v(\check{\alpha}),\rho\rangle\ge 1$.
Hence we have $\ell(\mu) \le \ell(\lambda) - 2k$.

Assume $k \ne 0$.
By the above argument, if the equality holds, then $n_v(\lambda)$ is dominant and $v(\alpha)$ is simple.
Conversely, assume that there exists $v\in W$ such that $n_v(\lambda)$ is dominant and $v(\alpha)$ is simple.
We have $\langle \nu(n_v(\mu)),v(\alpha)\rangle = \langle \nu(\mu),\alpha\rangle = \langle \nu(\lambda),\alpha\rangle - 2k\ge 0$.
If $\beta\in\Delta\setminus\{v(\alpha)\}$, then $\langle\nu(n_v(\mu)),\beta\rangle = \langle\nu(n_v(\lambda)),\beta\rangle - k\langle v(\alpha),\beta\rangle\ge 0$.
Hence $n_v(\mu)$ is dominant.
By the above argument, the equality holds.
\end{proof}

One application of the above lemma and the Bernstein relations is the following ``the simple Bernstein relations at $q = 0$''.
(See \cite[Corollary~5.53]{Vigneras-prop}.)
\begin{lem}\label{lem:simple Bernstein relations}
Let $\lambda\in\Lambda(1)$, $s = s_\alpha\in S$.
In $\mathcal{H}\otimes C$, we have the following.
If $\langle \nu(\lambda),\alpha\rangle > 0$, then
\[
	E(n_s(\lambda))(T_{n_s} - c_s) = T_{n_s}E(\lambda).
\]
If $\langle \nu(\lambda),\alpha\rangle < 0$, then
\[
	E(n_s(\lambda))T_{n_s} = (T_{n_s} - c_s)E(\lambda).
\]
\end{lem}
\begin{proof}
If $\langle\nu(\lambda),\alpha\rangle > 0$, by the Bernstein relations,
\[
	E(n_s(\lambda))T_{n_s} - T_{n_s}E(\lambda) = \sum^{\langle\nu(\lambda),\alpha\rangle - 1}_{k = 0}q_\lambda^{1/2}q_{n_s(\lambda)\mu_{n_s}(k)}^{-1/2}E(n_s(\lambda)\mu_{n_s}(k))c_{s,k}.
\]
The lemma follows from the following claim in this case.
\begin{claim}
We have $q_\lambda^{1/2}q_{n_s(\lambda)\mu_{n_s}(k)}^{-1/2}\in C[q_s^{1/2}]$.
If $k\ne 0$, then under $C[q_s^{1/2}]\to C$, $q_\lambda^{1/2}q_{n_s(\lambda)\mu_{n_s}(k)}^{-1/2}$ goes to $0$.
\end{claim}
\begin{proof}[Proof of Claim]
In the above equation, the left hand side is in $\mathcal{H}$.
The images of by $\nu$ of $(n_s(\lambda)\mu_{n_s}(k))_k$ are distinct by Lemma~\ref{lem:Bernstein relations}.
Recall that $\{E(w)\}_{w\in \widetilde{W}(1)}$ is a $C[q_s^{1/2}]$-basis of $\mathcal{H}$.
Therefore $q_\lambda^{1/2}q_{n_s(\lambda)\mu_{n_s}(k)}^{-1/2}\in C[q_s^{1/2}]$.
Hence there exists $n_s \in\Z_{\ge 0}$ for $s\in S_\aff/\mathord{\sim}$ such that $q_\lambda^{1/2}q_{n_s(\lambda)\mu_{n_s}(k)}^{-1/2}\in \prod_s q_s^{n_s/2}$.
We have $\sum_{s\in S_\aff/\mathord{\sim}}n_s = \ell(\lambda) - \ell(n_s(\lambda)\mu_{n_s}(k))$.
By the above lemma, the right hand side is grater than or equal to $2k$.
Hence if $k\ne 0$, there exists $s$ such that $n_s > 0$.
The claim follows.
\end{proof}
The same argument implies the lemma if $\langle \nu(\lambda),\alpha\rangle < 0$.
\end{proof}

If $\langle\nu(\lambda),\alpha\rangle = 0$, we have $T_{n_s}\theta(\lambda) = \theta(n_s(\lambda))T_{n_s}\in \mathcal{H}[q_s^{\pm 1/2}]$.
We have a slightly more general properties.
If $k = 0$, it is \cite[IV.17.~Lemma]{arXiv:1412.0737}.
\begin{lem}\label{lem:when <lambda,alpha>=0, relation with c'(-k)}
Let $\lambda\in\Lambda$, $\alpha\in\Delta$, $s = s_\alpha$ and $k\in\Z_{\ge 0}$.
Assume that $\langle\nu(\lambda),\alpha\rangle = 0$.
Then we have $\theta(\mu_{n_s}(-k))c_{s,-k}\theta(\lambda) = \theta(n_s(\lambda))\theta(\mu_{n_s}(-k))c_{s,-k}$ in $\mathcal{H}[q_s^{\pm 1/2}]$.
In particular $c_s\theta(\lambda) = \theta(n_s(\lambda))c_s$ and hence $(T_{n_s} - c_s)E(\lambda) = E(n_s(\lambda))(T_{n_s} - c_s)$ in $\mathcal{H}$.
\end{lem}
\begin{proof}
Take $\lambda_0$ from the center of $\Lambda(1)$ such that $\langle \nu(\lambda_0),\alpha\rangle < -k$.
Put $n = -\langle \nu(\lambda_0),\alpha\rangle$.
Since $T_{n_s}\theta(\lambda) = \theta(n_s(\lambda))T_{n_s}$, we have
\begin{align*}
\sum_{l = 0}^n\theta(n_s(\lambda_0)\mu_{n_s}(-l))c_{s,-l}\theta(\lambda) & = (T_{n_s}\theta(\lambda_0) - \theta(n_s(\lambda_0))T_{n_s})\theta(\lambda)\\
& = T_{n_s}\theta(\lambda)\theta(\lambda_0) - \theta(n_s(\lambda_0))T_{n_s}\theta(\lambda)\\
& = \theta(n_s(\lambda))T_{n_s}\theta(\lambda_0) - \theta(n_s(\lambda_0))\theta(n_s(\lambda))T_{n_s}\\
& = \theta(n_s(\lambda))(T_{n_s}\theta(\lambda_0) - \theta(n_s(\lambda_0))T_{n_s})\\
& = \sum_{l = 0}^n\theta(n_s(\lambda))\theta(n_s(\lambda_0)\mu_{n_s}(-l))c_{s,-l}\\
& = \sum_{l = 0}^n\theta(n_s(\lambda_0))\theta(n_s(\lambda))\theta(\mu_{n_s}(-l))c_{s,-l}.
\end{align*}
Hence we have
\begin{equation}\label{eq:before projection, commutative relation when orthogonal to alpha}
\sum_{l = 0}^n\theta(\mu_{n_s}(-l))c_{s,-l}\theta(\lambda) = \sum_{l = 0}^n\theta(n_s(\lambda))\theta(\mu_{n_s}(-l))c_{s,-l}.
\end{equation}
This is an equality in $\bigoplus_{\lambda'\in\Lambda(1)}\Z\theta(\lambda')$.
We have
\[
	\theta(\mu_{n_s}(-l))c_{s,-l}\theta(\lambda),\theta(n_s(\lambda))\theta(\mu_{n_s}(-l))c_{s,-l} \in \bigoplus_{\nu(\lambda') = \nu(\lambda) - l\check{\alpha}}C[q_s^{\pm 1/2}]\theta(\lambda').
\]
Hence projecting the both sides of \eqref{eq:before projection, commutative relation when orthogonal to alpha} to $\bigoplus_{\nu(\lambda') = \nu(\lambda) - k\check{\alpha}}C[q_s^{\pm 1/2}]\theta(\lambda')$, we get the first statement of the lemma.
Putting $k = 0$, we have $c_s\theta(\lambda) = \theta(n_s(\lambda))c_s$.
Hence the second statement follows.
\end{proof}

These formulas are very simple.
However, it is too simple even for studying representations over $C$.
So later we use the original Bernstein relations (Lemma~\ref{lem:Bernstein relations}).
For specializing $q_s\to 0$, Lemma~\ref{lem:lemma on length function} and an argument in the proof of Lemma~\ref{lem:simple Bernstein relations} is useful and it will be used later.

Finally, we prove the following lemma which will be used later.
This lemma is discovered in the study of \cite{arXiv:1412.0737}.
\begin{lem}[{\cite[IV.24.~Proposition]{arXiv:1412.0737}}]\label{lem:c_s,c_{s,-1},mu_n_s and Lambda_s'(1) cap Z_k}
Let $\alpha\in \Delta$ and put $s = s_\alpha$.
The groups $Z_{(\alpha,0),\kappa}$ and $Z_{(\alpha,-1),\kappa}$ generates $\Lambda'_s(1)\cap Z_\kappa$.
In particular, for a non-trivial character $\psi\colon \Lambda_s'(1)\cap Z_\kappa\to C^\times$ we have $\psi(c_sc_{s,-1}) = 0$.
\end{lem}
\begin{proof}
Replacing $\alggrp{G}$ with the algebraic group generated by $\alggrp{U}_\alpha$ and $\alggrp{U}_{-\alpha}$, we may assume that $\alggrp{G}$ has a semisimple $F$-rank $1$.
Put $s_0 = s_{(\alpha,0)}$, $s_1 =  s_{(\alpha,-1)}$.
Then we have $S_\aff = \{s_0,s_1\}$ and the map $s_0\mapsto n_{s_0}$, $s_1\mapsto n_{s_1}$ extends to the map $\widetilde{W}_\aff\to G'$ such that the braid relations hold. (Since the semisimple $F$-rank is $1$, a reduced expression of any element in $\widetilde{W}_\aff$ is unique.)
We denote this map by $\widetilde{w}\mapsto n_{\widetilde{w}}$.
Let $Z_0$ (resp.\ $Z_1$) be the inverse image of $Z_{(\alpha,0),\kappa}$ (resp.\  $Z_{(\alpha,-1),\kappa}$) in $\alggrp{Z}(F)\cap K$ and $H_0$ the group generated by $Z_{0}$ and $Z_{1}$.
Put $H = \bigcup_{w\in \widetilde{W}_\aff}I(1)n_w H_0I(1)$.
Since we have $\widetilde{W}_\aff(1) = \{n_w\mid w\in \widetilde{W}_\aff\}Z_\kappa$, we have $G'(\alggrp{Z}(F)\cap I(1)) = \bigcup_{\widetilde{w}\in \widetilde{W}_{\aff}(1)}I(1)\widetilde{w}I(1) = \bigcup_{w\in \widetilde{W}_\aff}I(1)n_w(\alggrp{Z}(F)\cap K\cap G')I(1) = H(\alggrp{Z}(F)\cap K\cap G')$.
By relations between the Iwahori-Matsumoto basis (braid relations and quadratic relations), $\sum_{\widetilde{w}\in H}\C T_{\widetilde{w}}\subset \mathcal{H}_\Z\otimes \C$ is a subalgebra where we regard $\C$ as a $\Z[q_s^{1/2}]$-algebra via $q_s^{1/2}\mapsto \#(I(1)sI(1)/I(1))^{1/2}$.
By the definition of the convolution product, the coefficient of $T_{\widetilde{w}_1\widetilde{w}_2}$ in $T_{\widetilde{w}_1}T_{\widetilde{w}_2}$ is non-zero.
Hence $H$ is a subgroup.

Let $\psi\colon \alggrp{Z}(F)\cap K\cap G'\to \C^\times$ be a character which is trivial on $H_0$.
We prove $\psi$ is trivial.
It follows that $H_0 = \alggrp{Z}(F)\cap K\cap G'$.
By the Bruhat decomposition, $H\cap (\alggrp{Z}(F)\cap K\cap G')\subset H_0$, on which $\psi$ is trivial.
Hence we can define a map $\psi'\colon G'(\alggrp{Z}(F)\cap I(1)) = H(\alggrp{Z}(F)\cap K\cap G')\to \C^\times$ by $\psi'(ht) = \psi'(t)$ for $h\in H$ and $t\in \alggrp{Z}(F)\cap K\cap G'$.

We prove $\psi'$ is a character.
Let $h_1,h_2\in \bigcup_{w\in \widetilde{W}_\aff} n_wH_0$ and $t_1,t_2\in \alggrp{Z}(F)\cap K\cap G'$.
Then $t_1$ normalizes $I(1)$ and $h_2$ normalizes $\alggrp{Z}(F)\cap K\cap G'$.
Hence
\[
	I(1)h_1I(1)t_1I(1)h_2I(1)t_2
	=
	I(1)h_1I(1)h_2I(1) (h_2^{-1}t_1h_2)t_2.
\]
Since $H$ is a subgroup, $I(1)h_1I(1)h_2I(1)\subset H$.
Therefore on the above set, the value of $\psi'$ is $\psi((h_2^{-1}t_1h_2)t_2)$.
We prove $\psi(h_2^{-1}t_1h_2) = \psi(t_1)$.
We may assume $h_2 = n_{s_i}$ where $i = 0$ or $1$.
We prove $t_1^{-1}h_2^{-1}t_1h_2\in Z_i$.
Put $\widetilde{\alpha} = (\alpha,-i)\in \Sigma_\aff$.
Let $\mathcal{F}$ be a facet which we used when we defined $Z_{\widetilde{\alpha},\kappa}$.
Then we have a finite group $G_{\mathcal{F},\kappa}$.
Let $U_{\widetilde{\alpha},\kappa}$ and $\overline{U}_{\widetilde{\alpha},\kappa}$ be as in the proof of Lemma~\ref{lem:Z_(a,k) is canonical} and let $G'_{s_i,\kappa}$ be the subgroup of $G_{\mathcal{F},\kappa}$ generated by these groups.
Let $t'_1$ (resp.\ $h'_2$) be the image of $t_1$ (resp.\ $h_2$) in $G_{\mathcal{F},\kappa}$.
Then $h'_2\in G'_{s_i,\kappa}$.
Since $t'_1$ normalizes this group, $(t'_1)^{-1}(h'_2)^{-1}t'_1h'_2\in G'_{s_i,\kappa}$.
It is also in $Z_\kappa$.
Hence $(t'_1)^{-1}(h'_2)^{-1}t'_1h'_2\in G'_{s_i,\kappa} \cap Z_\kappa = Z_{\widetilde{\alpha},\kappa}$.
Therefore $t_1^{-1}h_2t_1h_2\in Z_i$.
The map $\psi'$ is a character.

Since any character of $G'$ is trivial, $\psi'$ is a trivial character.
Hence $\psi$ is trivial.
\end{proof}

\section{Intertwining operators}\label{sec:Intertwining operators}
\subsection{Construction of intertwining operators}
We use the notation in the previous section.
In particular, $C$ is an algebraically closed filed of characteristic $p$, $\mathcal{H} = \mathcal{H}_\Z\otimes_{\Z}C$ and $C$ is a $C[q_s^{1/2}]$-algebra via $q_s^{1/2}\mapsto 0$.
In the rest of this paper, we investigate the representations of $\mathcal{H}\otimes C$.
All modules (or representations) in this paper are \emph{right} modules.
\begin{rem}
Since $\mathcal{H}\otimes C$ is finitely generated as a module of its center and the center is finitely generated over $C$ \cite[Theorem~1.2]{Vigneras-prop-II}, any irreducible $\mathcal{H}\otimes C$-module is finite-dimensional.
\end{rem}

Let $\lambda_1,\lambda_2\in \Lambda(1)$.
Then we have $E(\lambda_1)E(\lambda_2) = q_{\lambda_1\lambda_2}^{-1/2}q_{\lambda_1}^{1/2}q_{\lambda_2}^{1/2}E(\lambda_1\lambda_2)$ by Proposition~\ref{prop:multiplication formula}.
Hence $\mathcal{A} = \bigoplus_{\lambda\in \Lambda(1)}C[q^{1/2}_s]E(\lambda)$ is a subalgebra of $\mathcal{H}$.
Let $\Lambda^+(1)$ be the set of dominant elements in $\Lambda(1)$ and $C[\Lambda^+(1)] = \bigoplus_{\lambda\in\Lambda^+(1)}C \tau_\lambda$ the monoid algebra of $\Lambda^+(1)$.

In $\mathcal{A}\otimes C$, $E(\lambda_1)E(\lambda_2) = q_{\lambda_1\lambda_2}^{-1/2}q_{\lambda_1}^{1/2}q_{\lambda_2}^{1/2}E(\lambda_1\lambda_2)$ is $E(\lambda_1\lambda_2)$ or $0$ and it is not zero if and only if $\ell(\lambda_1\lambda_2) = \ell(\lambda_1) + \ell(\lambda_2)$.
By the length formula \eqref{eq:length formula}, we have $\ell(\lambda_1\lambda_2) = \ell(\lambda_1) + \ell(\lambda_2)$ if and only if $\nu(\lambda_1),\nu(\lambda_2)$ are in the same closed chamber.
Therefore we get
\begin{equation}\label{eq:multiplication in A}
E(\lambda_1)E(\lambda_2) =
\begin{cases}
E(\lambda_1\lambda_2) & \text{($\nu(\lambda_1)$ and $\nu(\lambda_2)$ are in the same closed chamber)},\\
0 & \text{(otherwise)}
\end{cases}
\end{equation}
in $\mathcal{A}\otimes C$.
Hence the $C$-linear map $\widetilde{\chi}\colon \mathcal{A}\otimes C\to C[\Lambda^+(1)]$ defined by
\[
	\widetilde{\chi}(E(\lambda)) =
	\begin{cases}
	\tau_\lambda & \text{($\lambda\in \Lambda^+(1)$)},\\
	0 & \text{(otherwise)}
	\end{cases}
\]
is an algebra homomorphism.
We also consider it as $\mathcal{A}\to C[\Lambda^+(1)]$.
For $w\in W$, we define $w\widetilde{\chi}\colon \mathcal{A}\to C[\Lambda^+(1)]$ by $(w\widetilde{\chi})(E(\lambda)) = \widetilde{\chi}(E(n^{-1}_{w}(\lambda)))$. (Note that this homomorphism depends on a lift $n_w$ of $w$.)
We will prove that, for an irreducible representation $X$ of $\mathcal{A}\otimes C$, the action of $\mathcal{A}\otimes C$ on $X$ factors through $w\widetilde{\chi}$ for some $w\in W$ (Proposition~\ref{prop:irred rep of A factors chi_Theta}).
This is a motivation to consider this homomorphism.
We construct intertwining operators between $\{w\widetilde{\chi}\otimes_\mathcal{A}\mathcal{H}\}_{w\in W}$.
When $\alggrp{G} = \GL_n$, they are constructed by Ollivier~\cite[5D1]{MR2728487}.
\begin{prop}\label{prop:construction of intertwining operators}
Let $w\in W$ and $s\in S$ such that $sw > w$.
Then
\[
	1\otimes 1\mapsto 1\otimes T_{n_s}^*
\]
gives a homomorphism $w\widetilde{\chi}\otimes_\mathcal{A}\mathcal{H}\to sw\widetilde{\chi}\otimes_\mathcal{A}\mathcal{H}$.
\end{prop}
\begin{proof}
We prove $1\otimes (T_{n_s} - c_s)E(\lambda) = (w\widetilde{\chi})(E(\lambda))\otimes (T_{n_s} - c_s)\in sw\widetilde{\chi}\otimes_\mathcal{A}\mathcal{H}$ for $\lambda\in\Lambda(1)$.
Take a simple root $\alpha$ such that $s = s_\alpha$.
If $\langle \nu(\lambda),\alpha\rangle = 0$, then by Lemma~\ref{lem:when <lambda,alpha>=0, relation with c'(-k)}, we have $1\otimes (T_{n_s} - c_s)E(\lambda) = 1\otimes E(n_s(\lambda))(T_{n_s} - c_s) = (sw\widetilde{\chi})(E(n_s(\lambda)))\otimes (T_{n_s} - c_s).$
Since $sw > w$, we have $n_{sw} = n_sn_w$.
Therefore we get $(sw\widetilde{\chi})(E(n_s(\lambda))) = \widetilde{\chi}(E(n_w^{-1}n_s^{-1}n_s(\lambda))) = \widetilde{\chi}(E(n_w^{-1}(\lambda))) = (w\widetilde{\chi})(E(\lambda))$.
We get the lemma in this case.

Assume that $\langle \alpha,\nu(\lambda)\rangle > 0$.
Then $\langle (sw)^{-1}(\alpha),(sw)^{-1}(\nu(\lambda))\rangle > 0$.
Therefore $n_{sw}^{-1}(\lambda)$ is not dominant since $(sw)^{-1}(\alpha) < 0$.
Hence $(sw\widetilde{\chi})(E(\lambda)) = 0$.
Therefore we have $1\otimes c_sE(\lambda) = 1\otimes E(\lambda)(\lambda^{-1}\cdot c_s) = (sw\widetilde{\chi})(E(\lambda))\otimes (\lambda^{-1}\cdot c_s) = 0$.
In $\mathcal{H}\otimes C$, we have $E(n_s(\lambda))(T_{n_s} - c_s) = T_{n_s}E(\lambda)$ by Lemma~\ref{lem:simple Bernstein relations}.
Hence
\begin{align*}
1\otimes (T_{n_s} - c_s)E(\lambda) & = 1\otimes T_{n_s}E(\lambda)\\ & = 1\otimes E(n_s(\lambda))(T_{n_s} - c_s)\\& = (sw\widetilde{\chi})(E(n_s(\lambda)))\otimes (T_{n_s} - c_s)\\ & = (w\widetilde{\chi})(E(\lambda))\otimes (T_{n_s} - c_s).
\end{align*}

Finally, assume that $\langle \alpha,\nu(\lambda)\rangle < 0$.
Then $\langle w^{-1}(\alpha),w^{-1}(\nu(\lambda))\rangle < 0$.
Since $w^{-1}(\alpha) > 0$, $n_w^{-1}(\lambda)$ is not dominant.
Hence $(w\widetilde{\chi})(E(\lambda)) = 0$.
Therefore it is sufficient to prove that $1\otimes(T_{n_s} - c_s)E(\lambda) = 0$.
By Lemma~\ref{lem:simple Bernstein relations}, we have
\begin{align*}
1\otimes (T_{n_s} - c_s)E(\lambda)& = 1\otimes E(n_s(\lambda))T_{n_s}\\
& = (sw\widetilde{\chi})(E(n_s(\lambda)))\otimes T_{n_s}\\
& = (w\widetilde{\chi})(E(\lambda))\otimes T_{n_s} = 0.\qedhere
\end{align*}
\end{proof}

The homomorphism obtained in the above lemma is denoted by $\Phi_{sw,w}\colon w\widetilde{\chi}\otimes_\mathcal{A}\mathcal{H}\to sw\widetilde{\chi}\otimes_\mathcal{A}\mathcal{H}$.
More generally, if $\ell(w_1) = \ell(w_1w_2^{-1}) + \ell(w_2)$, we have a homomorphism $\Phi_{w_1,w_2}\colon w_2\widetilde{\chi}\otimes_\mathcal{A}\mathcal{H}\to w_1\widetilde{\chi}\otimes_\mathcal{A}\mathcal{H}$ defined by $1\otimes 1\mapsto 1\otimes T_{n_{w_1w_2^{-1}}}^*$ by the above proposition and induction on $\ell(w_1w_2^{-1})$.

To construct the intertwining operator of the inverse direction, we need a notation.
For $\alpha\in \Delta$, $s = s_\alpha$ and $\lambda\in\Lambda(1)$ such that $\langle \nu(\lambda),\alpha\rangle > 0$, define
\[
	d(s,\lambda) =
	\begin{cases}
	1 & \text{($n_{w_1}(\lambda)$ is dominant and $w_1(\alpha)$ is simple for some $w_1\in W$)},\\
	0 & \text{(otherwise)}.
	\end{cases}
\]
\begin{lem}\label{lem:expression of d}
In $C$, we have $d(s,\lambda) = q_{\lambda}^{1/2}q_{s}^{-1}q_{\lambda\mu_{n_s^{-1}}(-1)}^{-1/2}$.
\end{lem}
\begin{proof}
By the length formula \eqref{eq:length formula} and the assumption $\langle\nu(\lambda),\alpha\rangle > 0$, we have
\begin{align*}
\ell(\lambda n_s^{-1}) & = \sum_{\beta\in\Sigma^+\setminus\{\alpha\}}\lvert\langle\nu(\lambda),\beta\rangle\rvert + \lvert\langle\nu(\lambda),\alpha\rangle - 1\rvert\\
& = \sum_{\beta\in\Sigma^+\setminus\{\alpha\}}\lvert\langle\nu(\lambda),\beta\rangle\rvert + \lvert\langle\nu(\lambda),\alpha\rangle\rvert - 1\\
& = \ell(\lambda) - 1.
\end{align*}
Hence $E(\lambda n_s^{-1}) = q_\lambda^{1/2}q_s^{-1}\theta(\lambda) T_{n_s^{-1}}$.
By the Bernstein relation (Lemma~\ref{lem:Bernstein relations}), we have
\begin{align*}
E(\lambda n_s^{-1}) & = q_\lambda^{1/2}q_s^{-1}\theta(\lambda)T_{n_s^{-1}}\\
& = q_\lambda q_s^{-1}T_{n_s^{-1}}\theta(n_s(\lambda)) - \sum_{k = 1}^{\langle\nu(\lambda),\alpha\rangle}q_\lambda^{1/2} q_s^{-1}q_{\lambda\mu_{n_s^{-1}}(-k)}^{-1/2}E(\lambda \mu_{n_s^{-1}}(-k))c_{s,-k}.
\end{align*}
Hence we have $q_\lambda^{1/2} q_s^{-1}q_{\lambda\mu_{n_s^{-1}}(-1)}^{-1/2}\in C[q_s]$.
The lemma follows from Lemma~\ref{lem:lemma on length function} and an argument in the proof Lemma~\ref{lem:simple Bernstein relations}.
\end{proof}

\begin{prop}\label{prop:hard intertwining operator}
Let $\alpha\in\Delta$, $s = s_\alpha$ and $w\in W$ such that $sw > w$.
Take $\lambda$ from the center of $\Lambda(1)$ such that $w^{-1}(\lambda)$ is dominant and $\langle \nu(\lambda),\alpha\rangle \ge 2$.
Then
\[
	1\otimes 1\mapsto 1\otimes (E(\lambda n_s^{-1}) + d(s,\lambda)E(\lambda \mu_{n_s^{-1}}(-1))c_{s,-1})
\]
gives a homomorphism $sw\widetilde{\chi}\otimes_\mathcal{A}\mathcal{H}\to w\widetilde{\chi}\otimes_\mathcal{A}\mathcal{H}$.
\end{prop}
\begin{proof}
Put $d = d(s,\lambda)$.
We have $E(\lambda n_s^{-1}) = q_\lambda^{1/2}q_s^{-1} \theta(\lambda)T_{n_s^{-1}}$ by the proof of Lemma~\ref{lem:expression of d}.
Let $\mu\in\Lambda(1)$.
We prove
\begin{equation}\label{eq:want to prove, constructing intertwining operators}
\begin{split}
&1\otimes (E(\lambda n_s^{-1}) + dE(\lambda \mu_{n_s^{-1}}(-1))c_{s,-1})E(\mu) \\&= (sw\widetilde{\chi})(E(\mu))\otimes (E(\lambda n_s^{-1}) + dE(\lambda \mu_{n_s^{-1}}(-1))c_{s,-1})
\end{split}
\end{equation}
in $w\widetilde{\chi}\otimes_\mathcal{A}\mathcal{H}$.
Assume that $\langle\nu(\mu),\alpha\rangle = 0$.
Since $\lambda$ is in the center of $\Lambda(1)$, by Lemma~\ref{lem:when <lambda,alpha>=0, relation with c'(-k)},  we have
\begin{align*}
\theta(\lambda\mu_{n_s^{-1}}(-1))c_{s,-1}\theta(\mu) & =
\theta(\lambda)\theta(\mu_{n_s^{-1}}(-1))c_{s,-1}\theta(\mu)\\
& = \theta(\lambda)\theta(n_s^{-1}(\mu))\theta(\mu_{n_s^{-1}}(-1))c_{s,-1}\\
& = \theta(n_s^{-1}(\mu))\theta(\lambda\mu_{n_s^{-1}}(-1))c_{s,-1}.
\end{align*}
Hence $E(\lambda\mu_{n_s^{-1}}(-1))c_{s,-1}E(\mu) = E(n_s^{-1}(\mu))E(\lambda\mu_{n_s^{-1}}(-1))c_{s,-1}$.
Since $T_{n_s^{-1}}\theta(\mu) = \theta(n_s^{-1}(\mu))T_{n_s^{-1}}$, we have
\[
	\theta(\lambda)T_{n_s^{-1}}\theta(\mu) = \theta(\lambda)\theta(n_s^{-1}(\mu))T_{n_s^{-1}} = \theta(n_s^{-1}(\mu))\theta(\lambda)T_{n_s^{-1}}.
\]
Hence $E(\lambda n_s^{-1})E(\mu) = E(n_s^{-1}(\mu))E(\lambda n_s^{-1})$.
Therefore,
\begin{align*}
&(E(\lambda n_s^{-1}) + dE(\lambda \mu_{n_s^{-1}}(-1))c_{s,-1})E(\mu)\\
&=
E(n_s^{-1}(\mu))(E(\lambda n_s^{-1}) + dE(\lambda \mu_{n_s^{-1}}(-1))c_{s,-1}).
\end{align*}
We get the lemma in this case.

Next assume that $\langle \nu(\mu),\alpha\rangle > 0$.
Then $\langle (sw)^{-1}(\alpha),(sw)^{-1}(\nu(\mu))\rangle > 0$.
Since $(sw)^{-1}(\alpha) = -w^{-1}(\alpha) < 0$, $n_{sw}^{-1}(\mu)$ is not dominant.
Hence we have  $(sw\widetilde{\chi})(E(\mu)) = 0$.
The right hand side of \eqref{eq:want to prove, constructing intertwining operators} is zero.

By the Bernstein relation, we have
\begin{equation}\label{eq:hard intertwining operator, expansion by Bernstein rel}
\begin{split}
E(\lambda n_s^{-1})E(\mu) & = q_\lambda^{1/2}q_\mu^{1/2}q_s^{-1}\theta(\lambda)T_{n_s^{-1}}\theta(\mu)\\
& = q_\lambda^{1/2}q_\mu^{1/2}q_s^{-1}\theta(\lambda n_s^{-1}(\mu))T_{n_s^{-1}}\\&\quad - \sum_{k = 1}^{\langle \nu(\mu),\alpha\rangle}q_\lambda^{1/2}q_\mu^{1/2}q_s^{-1}\theta(\lambda)c_{s,k}\theta(\mu_{n_s^{-1}}(-k)\mu).
\end{split}
\end{equation}
Since $\lambda$ is in the center of $\Lambda(1)$, we have
\begin{equation}\label{eq:cal E(n_s^-1mu)}
\begin{split}
q_\lambda^{1/2}q_\mu^{1/2}q_s^{-1}\theta(\lambda n_s^{-1}(\mu))T_{n_s^{-1}}
&=
q_\lambda^{1/2}q_\mu^{1/2}q_s^{-1} \theta(n_s^{-1}(\mu))\theta(\lambda)T_{n_s^{-1}}\\
&=
E(n_s^{-1}(\mu))E(\lambda n_s^{-1}).
\end{split}
\end{equation}
Hence
\begin{align*}
1\otimes q_\lambda^{1/2}q_\mu^{1/2}q_s^{-1}\theta(\lambda n_s^{-1}(\mu))T_{n_s^{-1}}
&= 1\otimes E(n_s^{-1}(\mu))E(\lambda n_s^{-1})\\
&= (w\widetilde{\chi})(E(n_s^{-1}(\mu)))\otimes E(\lambda n_s)\\
& = (sw\widetilde{\chi})(E(\mu))\otimes E(\lambda n_s) = 0.
\end{align*}
We calculate the second term of the right hand side of \eqref{eq:hard intertwining operator, expansion by Bernstein rel}.
We have
\begin{equation}\label{eq:term 1, constructing intertwining operators}
q_\lambda^{1/2}q_\mu^{1/2}q_s^{-1}\theta(\lambda)c_{s,k}\theta(\mu_{n_s^{-1}}(-k)\mu) =
q_\lambda^{1/2}q_\mu^{1/2}q_s^{-1}q_{\lambda\mu_{n_s^{-1}}(-k)\mu}^{-1/2}c_{s,k}E(\lambda\mu_{n_s^{-1}}(-k)\mu)
\end{equation}
and, by the argument in the proof of Lemma~\ref{lem:simple Bernstein relations}, it is zero in $\mathcal{H}\otimes C$ if $\ell(\lambda) + \ell(\mu) - \ell(\lambda \mu_{n_s^{-1}}(-k)\mu) - 2 > 0$.
By Lemma~\ref{lem:lemma on length function}, we have
\begin{equation}\label{eq:length, constructing intertwining operators}
\ell(\lambda) + \ell(\mu) - 2 \ge \ell(\lambda \mu) - 2 \ge \ell(\lambda \mu_{n_s^{-1}}(-k)\mu) + 2\min\{k,\langle \nu(\lambda\mu),\alpha\rangle - k\} - 2.
\end{equation}
Hence if $k > 1$ or $\langle \nu(\lambda\mu),\alpha\rangle - k > 1$, then \eqref{eq:term 1, constructing intertwining operators} is zero in $\mathcal{H}\otimes C$.
Since $k\le \langle \nu(\mu),\alpha\rangle$, $\langle\nu(\lambda\mu),\alpha\rangle - k \ge \langle\nu(\lambda),\alpha\rangle \ge 2 > 1$.
Hence \eqref{eq:term 1, constructing intertwining operators} is zero in $\mathcal{H}\otimes C$ if $k \ne 1$.
Since $\mu_{n_s}(-1)^{-1}\cdot c_{s,1} = c_{s,-1}$, by Lemma~\ref{lem:expression of d}, \eqref{eq:term 1, constructing intertwining operators} is
\begin{align*}
q_\lambda^{1/2}q_\mu^{1/2}q_s^{-1}\theta(\lambda)c_{s,1}\theta(\mu_{n_s^{-1}}(-1)\mu) & =
q_\lambda^{1/2}q_s^{-1}q_{\lambda\mu_{n_s^{-1}}(-1)}^{-1/2}E(\lambda\mu_{n_s^{-1}}(-1))c_{s,-1}E(\mu)\\
& = dE(\lambda\mu_{n_s^{-1}}(-1))c_{s,-1}E(\mu)
\end{align*}
Hence the left hand side of \eqref{eq:want to prove, constructing intertwining operators} is zero and we get \eqref{eq:want to prove, constructing intertwining operators}.

Finally assume that $\langle \nu(\mu),\alpha\rangle < 0$.
By the Bernstein relations, we have
\begin{align*}
E(\lambda n_s^{-1})E(\mu) & = q_{\lambda}^{1/2}q_s^{-1}q_\mu^{1/2}\theta(\lambda)T_{n_s^{-1}}\theta(\mu)\\
& = q_{\lambda}^{1/2}q_s^{-1}q_\mu^{1/2}\theta(\lambda n_s^{-1}(\mu))T_{n_s^{-1}}\\&\quad + \sum_{k = 1}^{-\langle \nu(\mu),\alpha\rangle}q_{\lambda}^{1/2}q_s^{-1}q_\mu^{1/2}\theta(\lambda n_s^{-1}(\mu)\mu_{n_s^{-1}}(-k))c_{s,-k}.
\end{align*}
By \eqref{eq:cal E(n_s^-1mu)},
\[
	q_{\lambda}^{1/2}q_s^{-1}q_\mu^{1/2}\theta(\lambda n_s^{-1}(\mu))T_{n_s^{-1}}
	=
	E(n_s^{-1}(\mu))E(\lambda n_s^{-1}).
\]
The term
\begin{equation}\label{eq:term 2, constructing intertwining operators}
\begin{split}
	&q_{\lambda}^{1/2}q_s^{-1}q_\mu^{1/2}\theta(\lambda n_s^{-1}(\mu)\mu_{n_s^{-1}}(-k))c_{s,-k}\\
	&=
	q_{\lambda}^{1/2}q_s^{-1}q_\mu^{1/2}q_{\lambda n_s^{-1}(\mu)\mu_{n_s^{-1}}(-k)}^{-1/2}E(\lambda n_s^{-1}(\mu)\mu_{n_s^{-1}}(-k))c_{s,-k}
\end{split}
\end{equation}
is zero if $\ell(\lambda) + \ell(\mu) + \ell(\lambda n_s^{-1}(\mu)\mu_{n_s^{-1}}(-k)) - 2 > 0$ by the argument in the proof of Lemma~\ref{lem:simple Bernstein relations}.
We have
\[
\begin{split}
&\ell(\lambda) + \ell(\mu) + \ell(\lambda n_s^{-1}(\mu)\mu_{n_s^{-1}}(-k)) - 2 \\
& = \ell(\lambda) + \ell(n_s^{-1}(\mu)) + \ell(\lambda n_s^{-1}(\mu)\mu_{n_s^{-1}}(-k)) - 2\\
& \ge \ell(\lambda n_s^{-1}(\mu)) + \ell(\lambda n_s^{-1}(\mu)\mu_{n_s^{-1}}(-k)) - 2\\
& \ge 2\min\{k,\langle \nu(\lambda n_s^{-1}(\mu)),\alpha\rangle - k\} - 2
\end{split}
\]
by Lemma~\ref{lem:lemma on length function}.
Since $k\le -\langle\nu(\mu),\alpha\rangle = \langle \nu(n_s^{-1}(\mu)),\alpha\rangle$, we have $\langle \nu(\lambda n_s^{-1}(\mu)),\alpha\rangle - k \ge \langle \nu(\lambda),\alpha\rangle > 1$.
Hence \eqref{eq:term 2, constructing intertwining operators} is zero if $k\ne 1$.
If $k = 1$, then \eqref{eq:term 2, constructing intertwining operators} is
\[
	q_{\lambda}^{1/2}q_s^{-1}q_{\lambda \mu_{n_s^{-1}}(-1)}^{-1/2}E(n_s^{-1}(\mu))E(\lambda\mu_{n_s^{-1}}(-1))c_{s,-1}
	=
	dE(n_s^{-1}(\mu))E(\lambda\mu_{n_s^{-1}}(-1))c_{s,-1}.
\]
by Lemma~\ref{lem:expression of d}.
Hence we get
\[
	E(\lambda n_s^{-1})E(\mu) = E(n_s^{-1}(\mu))(E(\lambda n_s^{-1}) + dE(\lambda \mu_{n_s^{-1}}(-1))c_{s,-1}).
\]
in $\mathcal{H}\otimes C$.

On the other hand, by \eqref{eq:E(lambda_1) and E(lambda_2)}, we have
\[
\begin{split}
&1\otimes E(\lambda \mu_{n_s^{-1}}(-1))c_{s,-1}E(\mu) \\& = 1\otimes E(\mu) E(\mu^{-1}\lambda \mu_{n_s^{-1}}(-1)\mu)(\mu^{-1}\cdot c_{s,-1})\\
& = (w\widetilde{\chi})(E(\mu))\otimes E(\mu^{-1}\lambda \mu_{n_s^{-1}}(-1)\mu)(\mu^{-1}\cdot c_{s,-1}).
\end{split}
\]
Since $\langle w^{-1}(\alpha),w^{-1}(\mu)\rangle < 0$ and $w^{-1}(\alpha) > 0$, $w^{-1}(\mu)$ is not dominant.
Hence $w\widetilde{\chi}(E(\mu)) = 0$.
Therefore, the left hand side of \eqref{eq:want to prove, constructing intertwining operators} is
\[
\begin{split}
&1\otimes E(n_s^{-1}(\mu))(E(\lambda n_s^{-1}) + dE(\lambda \mu_{n_s^{-1}}(-1))c_{s,-1})\\
&=(w\widetilde{\chi})(E(n_s^{-1}(\mu)))\otimes (E(\lambda n_s^{-1}) + dE(\lambda \mu_{n_s^{-1}}(-1))c_{s,-1})\\
&=(sw\widetilde{\chi})(E(\mu))\otimes (E(\lambda n_s^{-1}) + dE(\lambda \mu_{n_s^{-1}}(-1))c_{s,-1})\\
\end{split}
\]
We get the proposition.
\end{proof}

\begin{prop}\label{prop:compositions of intertwining operators}
Let $\alpha,s,w,\lambda$ be as in Proposition~\ref{prop:hard intertwining operator}.
Assume that there exist $\lambda_1,\lambda_2$ which satisfy the same conditions for $\lambda$ such that $\lambda = \lambda_1\lambda_2$.
Then the compositions of homomorphisms in Proposition~\ref{prop:construction of intertwining operators} and Proposition~\ref{prop:hard intertwining operator} is given by the multiplication of
\[
	(w\widetilde{\chi})(E(\lambda) - d(s,\lambda)E(\lambda \mu_{n_s^{-1}}(-1))c_{s,-1}c_s)\in C[\Lambda^+(1)].
\]
\end{prop}
\begin{proof}
Put $d = d(s,\lambda)$ and $\Phi = \Phi_{sw,w}$.
Let $\Psi\colon sw\widetilde{\chi}\otimes_\mathcal{A}\mathcal{H}\to w\widetilde{\chi}\otimes_\mathcal{A}\mathcal{H}$ be the homomorphisms defined in Proposition~\ref{prop:hard intertwining operator}.
Then we have
\[
	\Psi\circ\Phi(1\otimes 1) = 1\otimes(E(\lambda n_s^{-1}) + dE(\lambda \mu_{n_s^{-1}}(-1))c_{s,-1})(T_{n_s} - c_s).
\]
By the proof of Lemma~\ref{lem:expression of d}, we have $E(\lambda n_s^{-1}) = q_\lambda^{1/2} q_s^{-1}\theta(\lambda) T_{n_s^{-1}}$.
Therefore
\[
E(\lambda n_s^{-1})(T_{n_s} - c_s) = q_\lambda^{1/2}q_s^{-1}\theta(\lambda)T_{n_s^{-1}}(T_{n_s} - c_s) = q_\lambda^{1/2}\theta(\lambda) = E(\lambda).
\]
Since $\langle\nu(\lambda),\alpha\rangle = \langle\nu(\lambda_1),\alpha\rangle + \langle\nu(\lambda_2),\alpha\rangle \ge 4$, $\langle \nu(\lambda\mu_{n_s^{-1}}(-1)),\alpha\rangle = \langle\nu(\lambda),\alpha\rangle - 2 > 0$.
From this and by the length formula \eqref{eq:length formula}, we have $q_{\lambda\mu_{n_s^{-1}}(-1)n_s} = q_{\lambda \mu_{n_s}(-1)}q_s^{-1}$. (See the proof of Lemma~\ref{lem:expression of d}.)
Therefore, by Proposition~\ref{prop:multiplication formula}, we have $E(\lambda \mu_{n_s^{-1}}(-1))T_{n_s} = q_sE(\lambda \mu_{n_s^{-1}}(-1)n_s) = 0$ in $\mathcal{H}\otimes C$.
Hence
\begin{align*}
\Psi\circ\Phi(1\otimes 1) & = 1\otimes(E(\lambda n_s^{-1}) + dE(\lambda \mu_{n_s^{-1}}(-1))c_{s,-1})(T_{n_s} - c_s)\\
& = 1\otimes(E(\lambda) - dE(\lambda \mu_{n_s^{-1}}(-1))c_{s,-1}c_s)\\
& = (w\widetilde{\chi})(E(\lambda) - dE(\lambda \mu_{n_s^{-1}}(-1))c_{s,-1}c_s)\otimes 1.
\end{align*}

Next we calculate
\[
	\Psi\circ\Phi(1\otimes 1) = 1\otimes(T_{n_s} - c_s)(E(\lambda n_s^{-1}) + dE(\lambda \mu_{n_s^{-1}}(-1))c_{s,-1}).
\]
First, by the Bernstein relation, we have
\begin{align*}
T_{n_s}\theta(\lambda)T_{n_s^{-1}} & = \theta(n_s(\lambda))T_{n_s}T_{n_s^{-1}} - \sum_{k = 0}^{\langle \nu(\lambda),\alpha\rangle - 1}\theta(n_s(\lambda)\mu_{n_s}(k))c_{s,k}T_{n_s^{-1}}\\
& = \theta(n_s(\lambda))(q_s + c_sT_{n_s^{-1}}) - \sum_{k = 0}^{\langle \nu(\lambda),\alpha\rangle - 1}\theta(n_s(\lambda)\mu_{n_s}(k))c_{s,k}T_{n_s^{-1}}\\
& = q_s\theta(n_s(\lambda)) - \sum_{k = 1}^{\langle \nu(\lambda),\alpha\rangle - 1}\theta(n_s(\lambda)\mu_{n_s}(k))c_{s,k}T_{n_s^{-1}}.
\end{align*}
Hence, using \eqref{eq:cal E(n_s^-1mu)} with $\mu = 1$,
\begin{align*}
T_{n_s}E(\lambda n_s^{-1}) & = q_\lambda^{1/2}q_s^{-1}T_{n_s}\theta(\lambda)T_{n_s^{-1}}\\
& = E(n_s(\lambda)) - \sum_{k = 1}^{\langle \nu(\lambda),\alpha\rangle - 1}q_\lambda^{1/2}q_s^{-1}q_{n_s(\lambda)\mu_{n_s}(k)}^{-1/2}E(n_s(\lambda)\mu_{n_s}(k))c_{s,k}T_{n_s^{-1}}.
\end{align*}
The term
\begin{equation}\label{eq:term 1, calculating compositions of intertwining operators}
q_\lambda^{1/2}q_s^{-1}q_{n_s(\lambda)\mu_{n_s}(k)}^{-1/2}E(n_s(\lambda)\mu_{n_s}(k))c_{s,k}T_{n_s^{-1}}
\end{equation}
is zero if $\ell(\lambda) - 2 - \ell(n_s(\lambda)\mu_{n_s}(k)) > 0$ by an argument in the proof of Lemma~\ref{lem:simple Bernstein relations}.
Hence, by Lemma~\ref{lem:lemma on length function}, \eqref{eq:term 1, calculating compositions of intertwining operators} is zero if $k\ne 1,\langle \nu(\lambda),\alpha\rangle - 1$.
If $k = \langle \nu(\lambda),\alpha\rangle - 1$, then, since $\langle \nu(n_s(\lambda)\mu_{n_s}(k)),\alpha\rangle > 0$, we have $\ell(n_s(\lambda)\mu_{n_s}(k)n_s^{-1}) = \ell(n_s(\lambda)\mu_{n_s}(k)) - 1$. (See the proof of Lemma~\ref{lem:expression of d}.)
Hence $E(n_s(\lambda)\mu_{n_s}(k))T_{n_s^{-1}} = q_sE(n_s(\lambda)\mu_{n_s}(k)n_s^{-1})$ by Proposition~\ref{prop:multiplication formula}.
Therefore, \eqref{eq:term 1, calculating compositions of intertwining operators} is
\[
	q_\lambda^{1/2}q_{n_s(\lambda)\mu_{n_s}(k)}^{-1/2}E(n_s(\lambda)\mu_{n_s}(k)n_s^{-1})n_s\cdot c_{s,k}.
\]
It is zero since $\ell(\lambda) - \ell(n_s(\lambda)\mu_{n_s}(k)) > 0$ by Lemma~\ref{lem:lemma on length function}.
For $k = 1$, we have $q_{n_s(\lambda)\mu_{n_s}(1)} = q_{\lambda n_s^{-1}(\mu_{n_s}(1))} = q_{\lambda \mu_{n_s}(-1)}$.
Hence by Lemma~\ref{lem:expression of d}, we have that \eqref{eq:term 1, calculating compositions of intertwining operators} is $dE(n_s(\lambda)\mu_{n_s}(1))c_{s,1}T_{n_s^{-1}}$.
Hence, by Lemma~\ref{lem:Bernstein relations}, \ref{lem:Bernstein relation, n_s <-> n_s^-1}, we have
\begin{align*}
1\otimes T_{n_s}E(\lambda n_s^{-1}) & = 1\otimes (E(n_s(\lambda)) - dE(n_s(\lambda)\mu_{n_s}(1))c_{s,1}T_{n_s^{-1}})\\
& = 1\otimes (E(n_s(\lambda)) - dE(n_s(\lambda)\mu_{n_s}(1))T_{n_s^{-1}} (n_s\cdot c_{s,1}))\\
& = 1\otimes (E(n_s(\lambda)) - dE(n_s(\lambda)\mu_{n_s^{-1}}(1))T_{n_s} c_{s,-1})\\
& = 1\otimes (E(n_s(\lambda)) - dE(n_s(\lambda\mu_{n_s^{-1}}(-1)))T_{n_s} c_{s,-1})\\
& = (w\widetilde{\chi})(E(\lambda))\otimes 1 - d(w\widetilde{\chi})(E(\lambda \mu_{n_s^{-1}}(-1)))\otimes T_{n_s}c_{s,-1}.
\end{align*}
By the length formula \eqref{eq:length formula}, we have $\ell(\lambda n_s^{-1}) = \ell(\lambda_1\lambda_2n_s^{-1}) = \ell(\lambda_1) + \ell(\lambda_2 n_s^{-1})$.
Hence $E(\lambda n_s^{-1}) = E(\lambda_1)E(\lambda_2 n_s^{-1})$.
We have
\begin{align*}
1\otimes c_sE(\lambda n_s^{-1}) & = 1\otimes c_sE(\lambda_1)E(\lambda_2 n_s^{-1})\\
& = 1\otimes E(\lambda_1)(\lambda_1^{-1}\cdot c_s)E(\lambda_2 n_s^{-1})\\
& = (sw\widetilde{\chi})(E(\lambda_1))\otimes(\lambda_1^{-1}\cdot c_s)E(\lambda_2 n_s^{-1}).
\end{align*}
Since $\langle \alpha,\nu(\lambda_1)\rangle > 0$, we have $\langle (sw)^{-1}(\alpha),(sw)^{-1}(\nu(\lambda_1))\rangle > 0$.
By the assumption, $(sw)^{-1}(\alpha) < 0$.
Hence $n_{sw}^{-1}(\lambda_1)$ is not dominant.
We have $(sw\widetilde{\chi})(E(\lambda_1)) = 0$.
Hence
\[
	1\otimes c_sE(\lambda n_s^{-1}) = 0.
\]
By Proposition~\ref{prop:construction of intertwining operators}, we have
\[
1\otimes (T_{n_s} - c_s)E(\lambda \mu_{n_s^{-1}}(-1))c_{s,-1}
=
(w\widetilde{\chi})(E(\lambda\mu_{n_s^{-1}}(-1)))\otimes (T_{n_s} - c_s)c_{s,-1}.
\]
From these calculations, we have
\begin{align*}
\Psi\circ\Phi(1\otimes 1) & = 1\otimes(T_{n_s} - c_s)(E(\lambda n_s^{-1}) + dE(\lambda \mu_{n_s^{-1}}(-1))c_{s,-1})\\
& = (w\widetilde{\chi})(E(\lambda))\otimes 1 - d(w\widetilde{\chi})(E(\lambda \mu_{n_s^{-1}}(-1)))\otimes c_{s,-1}c_s\\
& = (w\widetilde{\chi})(E(\lambda) - dE(\lambda \mu_{n_s^{-1}}(-1))c_{s,-1}c_s)\otimes 1.\qedhere
\end{align*}
\end{proof}

\subsection{The modules $w\chi_\Theta\otimes_\mathcal{A}\mathcal{H}$}
For a subset $\Theta\subset \Delta$, set $\Lambda_\Theta(1) = \{\lambda\in\Lambda(1) \mid \langle\nu(\lambda),\alpha\rangle = 0\ (\alpha\in\Theta)\}$ and put $\Lambda_\Theta^+(1) = \Lambda_\Theta(1)\cap \Lambda^+(1)$.
Let $C[\Lambda_\Theta(1)] = \bigoplus_{\lambda\in \Lambda_\Theta(1)}C\tau_\lambda$ be the group algebra.
Define $\chi_\Theta\colon \mathcal{A}\to C[\Lambda_\Theta(1)]$ by
\[
	\chi_\Theta(E(\lambda)) =
	\begin{cases}
	\tau_\lambda & (\lambda\in\Lambda_\Theta^+(1)),\\
	0 & (\text{otherwise}).
	\end{cases}
\]
This homomorphism factors through $\widetilde{\chi}$.
By \eqref{eq:multiplication in A}, it is easy to see that this is an algebra homomorphism.
%Let $W_\Theta$ be the group generated by $\{s_\alpha\mid \alpha\in\Theta\}$.
%If $w\in W_\Theta$, then $(n_w\chi_\Theta)(E(\lambda)) = n_{w}^{-1}(\chi_\Theta(E(\lambda)))$ where the automorphism $n_{w}^{-1}\colon C[\Lambda_\Theta(1)]\to C[\Lambda_\Theta(1)]$ is defined by $\tau_\lambda\mapsto \tau_{n_{w}^{-1}(\lambda)}$.
%Hence every properties for $w\chi_\Theta$ can be obtained from that of $\chi_\Theta$.
%Therefore, we always assume $w(\Theta)\subset\Sigma^+$.

The motivation to consider $\chi_\Theta$ is the following lemma.
For an $\mathcal{A}\otimes C$-module $X$, set $\supp X = \{\lambda\in \Lambda(1)\mid \text{$E(\lambda)\ne 0$ on $X$}\}$.
We define a new module $wX$ by $(wX)(E(\lambda)) = X(E(n_w^{-1}(\lambda)))$ on the same space.
We have $\supp wX = n_w(\supp X)$.
\begin{prop}\label{prop:irred rep of A factors chi_Theta}
Let $X$ be an irreducible representation of $\mathcal{A}$ over $C$.
Then it factors through $w\chi_\Theta$ for some $w\in W$ and $\Theta\subset \Delta$.
\end{prop}
\begin{proof}
The statement is equivalent to
\begin{itemize}
\item There is a closure of a facet $\mathcal{F}\subset V$ with respect to $\Sigma$ such that $\nu(\supp X) = \nu(\Lambda(1))\cap \mathcal{F}$.
\item If $\lambda\in \supp X$, then $E(\lambda)$ is invertible on $X$.
\end{itemize}
Since $E(\mu)$ is invertible for any element in $\mu\in \Ker\nu$, we have $\nu(\lambda)\in \nu(\supp X)$ if and only if $\lambda\in \supp X$.

First we prove that $\nu(\supp X)$ is a subset of a closed Weyl chamber.
Take $\lambda\in \supp X$ and set $\mathcal{Y} = \{\mu \in \Lambda(1)\mid \text{$\nu(\mu)$ is in the same closed chamber of $\nu(\lambda)$}\}$.
Then $XE(\lambda) \ne 0$.
Hence $XE(\lambda)\mathcal{A} = X$.
If $\mu\notin \mathcal{Y}$, $E(\lambda)E(\mu) = 0$ and if $\mu\in \mathcal{Y}$ then $E(\lambda)E(\mu) = E(\lambda \mu)$ by \eqref{eq:multiplication in A}.
Hence $X = \sum_{\mu\in \mathcal{Y}}XE(\lambda)E(\mu) = \sum_{\mu\in \mathcal{Y}}XE(\lambda\mu)$.
Let $\lambda'\in \Lambda(1)$ such that $\lambda'\notin \mathcal{Y}$.
Then $E(\lambda\mu)E(\lambda') = 0$ for $\mu\in \mathcal{Y}$.
Hence $XE(\lambda') = \sum_{\mu\in \mathcal{Y}}XE(\lambda\mu)E(\lambda') = 0$.
Therefore $\lambda'\notin \supp X$.
Hence $\nu(\supp X)$ is contained in a closed chamber.

Hence there exists $w\in W$ such that $\supp X\subset n_w^{-1}(\Lambda^+(1))$.
Replacing $X$ with $wX$, we may assume that $\supp X\subset \Lambda^+(1)$.
For each $\alpha\in \Delta$, take $\lambda_\alpha$ from the center of $\Lambda(1)$ which satisfies $\langle \nu(\lambda_\alpha),\beta\rangle = 0$ for $\beta\in\Delta\setminus\{\alpha\}$ and $\langle\nu(\lambda_\alpha),\alpha\rangle > 0$.
Set $\Theta = \{\alpha\in\Delta\mid \lambda_\alpha\notin \supp X\}$.
We prove $\supp X = \Lambda_\Theta^+(1)$.

Take $\lambda\in \Lambda^+(1)$.
There exists $n,n_\alpha\in\Z_{>0}$ such that $\langle \Delta,\nu(\lambda^n\prod_{\alpha\in\Delta}\lambda_\alpha^{-n_\alpha})\rangle = 0$.
Put $\mu = \lambda^n \prod_{\alpha\in\Delta}\lambda_\alpha^{-n_\alpha}$.
Then $E(\mu)$ is invertible and we have $E(\lambda)^n = E(\mu)\prod_{\alpha\in\Delta}E(\lambda_\alpha)^{n_\alpha}$.
If $\lambda\not\in \Lambda_{\Theta}^+(1)$, namely $\langle\nu(\lambda),\alpha\rangle \ne 0$ for some $\alpha\in \Theta$, then $n_\alpha \ne 0$ for such $\alpha\in\Theta$.
Hence $E(\lambda)^n = 0$ on $X$ since $E(\lambda_\alpha) = 0$ on $X$ for $\alpha\in\Theta$.
Hence $\{x\in X\mid xE(\lambda) = 0\}$ is not zero.
Take $x\in X$ such that $xE(\lambda) = 0$.
Let $\mu\in \Lambda(1)$.
Then $xE(\mu)E(\lambda) = xE(\lambda)E(\lambda^{-1}\mu\lambda) = 0$ by \eqref{eq:E(lambda_1) and E(lambda_2)}.
Hence $\{x\in X\mid xE(\lambda) = 0\}$ is $\mathcal{A}$-stable and by the irreducibility of $X$, it is equal to $X$.
Therefore, $E(\lambda) = 0$ on $X$.
Hence $\supp X\subset \Lambda_\Theta^+(1)$.

If $\lambda\in \Lambda_\Theta^+(1)$ then $n_\alpha = 0$ for $\alpha\in\Theta$.
If $\alpha\notin \Theta$, then $E(\lambda_\alpha)$ is non-zero.
Since $\lambda_\alpha$ is in the center of $\Lambda(1)$, $E(\lambda_\alpha)$ is in the center of $\mathcal{A}$.
Hence $E(\lambda_\alpha)$ is scalar on $X$.
In particular, for $\alpha\notin \Theta$, $E(\lambda_\alpha)$ is invertible.
Since $E(\lambda)^n = E(\mu)\prod_{\alpha\notin \Theta}E(\lambda_\alpha)^{n_\alpha}$, $E(\lambda)$ is invertible on $X$.
\end{proof}

By the above lemma, $\supp X = n_w(\Lambda_\Theta^+(1))$ for some $w\in W$ and $\Theta\subset\Delta$.
Let $W_\Theta$ be the group generated by $\{s_\alpha \mid \alpha\in\Theta\}$.
Then for $w\in W_\Theta$, $n_w(\Lambda_\Theta^+(1)) = \Lambda_\Theta^+(1)$.
Since the multiplication $\{w\in W\mid w(\Theta)\subset\Sigma^+\}\times W_\Theta\to W$ is bijective, we can take the above $w$ such that $w(\Theta)\subset \Sigma^+$.
We have $\supp X = n_w(\Lambda_\Theta^+(1))$ if and only if the action of $\mathcal{A}$ on $X$ factors through $w\chi_\Theta$.
\begin{rem}
For an irreducible $\mathcal{A}$-module $X$, we can attache $\Theta\subset \Lambda(1)$, $w\in W/W_\Theta$ and an irreducible module $X_0$ of $C[\Lambda_\Theta(1)]$ such that $X = X_0\circ w\chi_\Theta$.
On the other hand, if $(\Theta,w,X_0)$ is given, then we get $X = X_0\circ w\chi_\Theta$.
Hence irreducible $\mathcal{A}$-modules are classified by such triples.
\end{rem}

The $\mathcal{H}$-module $w\chi_\Theta\otimes_\mathcal{A}\mathcal{H}$ has the right action of $C[\Lambda_\Theta(1)]$.
Hence it is a $(C[\Lambda_\Theta(1)],\mathcal{H})$-bimodule.
The first property of $w\chi_\Theta\otimes_\mathcal{A}\mathcal{H}$ is the following freeness as a $C[\Lambda_\Theta(1)]$-module.
\begin{prop}\label{prop:freeness}
The module $w\chi_\Theta\otimes_\mathcal{A}\mathcal{H}$ is free of finite rank as a $C[\Lambda_\Theta(1)]$-module.
\end{prop}
We use the following easy lemma.
\begin{lem}
Let $X$ be a $C[\Lambda_\Theta^+(1)]$-module.
Assume that $X$ has a decomposition $X = \bigoplus_{\lambda\in\Lambda_\Theta(1)}X_\lambda$ such that $\dim X_\lambda \le 1$ and $\tau_\mu(X_\lambda)\subset X_{\mu\lambda}$.
Then the module $C[\Lambda_\Theta(1)]\otimes_{C[\Lambda_\Theta^+(1)]}X$ is free of rank less than or equal to one.
\end{lem}
\begin{proof}
Notice that $C[\Lambda_\Theta(1)]$ is the ring of fractions of $C[\Lambda_\Theta^+(1)]$ with respect to $\{\tau_\lambda\mid \lambda\in\Lambda_\Theta^+(1)\}$.
We may assume $C[\Lambda_\Theta(1)]\otimes X_\lambda \ne 0$ for some $\lambda\in \Lambda_\Theta(1)$, since if there is no such $\lambda$, then $C[\Lambda_\Theta(1)]\otimes_{C[\Lambda_\Theta^+(1)]}X = 0$.
Then $\tau_\mu(X_\lambda) \ne 0$ for any $\mu\in \Lambda_\Theta^+(1)$.
Since $\dim X_\lambda,\dim X_{\mu\lambda}\le 1$, $\tau_\mu$ gives an isomorphism $X_\lambda\to X_{\mu\lambda}$.
Hence a homomorphism $C[\Lambda_\Theta^+(1)] \to X$ defined by $f\mapsto fm$ is injective for $m\in X_\lambda\setminus\{0\}$.
Fix $m\in X_\lambda\setminus\{0\}$ and consider the homomorphism $C[\Lambda_\Theta(1)]\to C[\Lambda_\Theta(1)]\otimes_{C[\Lambda_\Theta^+(1)]}X$ induced by the above homomorphism (namely $f\mapsto f\otimes m$).

We prove that it is an isomorphism.
It is injective since the above homomorphism is injective and $C[\Lambda_\Theta(1)]$ is a ring of fractions.
Take $\lambda'\in \Lambda_\Theta(1)$ such that $C[\Lambda_\Theta(1)]\otimes X_{\lambda'} \ne 0$.
Take $\mu_1,\mu_2\in \Lambda_\Theta^+(1)$ such that $\mu_1 \lambda = \mu_2\lambda'$.
Then we have $\tau_{\mu_1}(X_\lambda) = X_{\mu_1\lambda} = \tau_{\mu_2}(X_{\lambda'})$.
Hence $C[\Lambda_\Theta(1)]\otimes X_{\lambda'} = C[\Lambda_\Theta(1)]\tau_{\mu_2}^{-1}\tau_{\mu_1}\otimes X_{\lambda} = C[\Lambda_\Theta(1)]\otimes X_{\lambda}$.
Since $X = \bigoplus_{\lambda'\in \Lambda_\Theta(1)}X_{\lambda'}$, we have $C[\Lambda_\Theta(1)]\otimes X = \sum_{\lambda'\in \Lambda_\Theta(1)}C[\Lambda_\Theta(1)]\otimes X_{\lambda'}$.
Hence the homomorphism is surjective.
\end{proof}

\begin{proof}[Proof of Proposition~\ref{prop:freeness}]
Notice that the image of $w\chi_\Theta$ is $C[\Lambda_\Theta^+(1)]$ and the kernel of $w\chi_\Theta$ is $\sum_{\lambda\in\Lambda(1)\setminus n_w(\Lambda_\Theta^+(1))}CE(\lambda)$.
Hence
\[
	w\chi_\Theta\otimes_\mathcal{A}\mathcal{H}
	=
	C[\Lambda_\Theta(1)]\otimes_{C[\Lambda_\Theta^+(1)]}\left((\mathcal{H}\otimes C)/\left(\sum_{\lambda\in\Lambda(1)\setminus n_w(\Lambda_\Theta^+(1))}CE(\lambda)\mathcal{H}\right)\right).
\]
For $\lambda\in \Lambda(1)$ and $v\in W$, put
\[
	\sigma_{\lambda,v} = CE(n_w(\lambda) n_v)/\left(\sum_{\mu\in \Lambda(1)\setminus n_w(\Lambda_\Theta^+(1))}CE(\mu)E(\mu^{-1}n_w(\lambda) n_v)\right).
\]
Then $\dim \sigma_{\lambda,v}\le 1$.
Since $E(\lambda)E(\mu n_v)\in CE(\lambda \mu n_v)$ for $\lambda,\mu\in\Lambda(1)$ and $v\in W$, we have
\[
	(\mathcal{H}\otimes C)/\left(\sum_{\lambda\in\Lambda(1)\setminus n_w(\Lambda_\Theta^+(1))}CE(\lambda)\mathcal{H}\right)
	=
	\bigoplus_{\lambda\in \Lambda(1),v\in W}\sigma_{\lambda,v}.
\]
Moreover, for $\mu\in \Lambda^+_{\Theta}(1)$, $\tau_\mu(\sigma_\lambda)\subset \sigma_{\mu\lambda}$.
We have
\[
	w\chi_\Theta\otimes_\mathcal{A}\mathcal{H}
	=
	\bigoplus_{\lambda\in \Lambda_\Theta(1)\backslash\Lambda(1),v\in W}C[\Lambda_\Theta(1)]\otimes_{C[\Lambda_\Theta^+(1)]}\bigoplus_{\mu\in \Lambda_\Theta(1)}\sigma_{\mu\lambda,v}.
\]
Applying the above lemma to $\bigoplus_{\mu\in \Lambda_\Theta(1)}\sigma_{\mu\lambda,v}$, we get the freeness.
Since $\mathcal{H}$ is finitely generated as a left module over $\mathcal{A}$, $w\chi_\Theta\otimes_\mathcal{A}\mathcal{H}$ is finitely generated as a $C[\Lambda_\Theta(1)]$-module.
Hence the rank is finite.
\end{proof}
Next, we prove the injectivity of the intertwining operator obtained in Proposition~\ref{prop:construction of intertwining operators}.
We remark the following easy lemma.
\begin{lem}\label{lem:not a zero divisor}
Let $\alpha\in\Delta$ and $\lambda\in \Lambda_\Theta(1)$.
An element in $\tau_\lambda + \sum_{\langle\nu(\mu),\alpha\rangle < \langle\nu(\lambda),\alpha\rangle}C\tau_\mu$ is not a zero-divisor in $C[\Lambda_\Theta(1)]$.
\end{lem}
\begin{proof}
Let $E$ be such element and $F = \sum_{\mu\in \Lambda_\Theta(1)}c_\mu\tau_\mu\in C[\Lambda_\Theta(1)]$.
Assume that $F\ne 0$ and put $l =  \max\{\langle\nu(\mu),\alpha\rangle\mid c_\mu\ne 0\}$.
Then $EF \in \sum_{\langle\nu(\mu),\alpha\rangle = l}c_{\mu}\tau_{\lambda\mu} + \sum_{\langle\nu(\mu),\alpha\rangle < \langle\nu(\lambda),\alpha\rangle + l}C\tau_\mu$.
Hence $EF\ne 0$.
\end{proof}
We apply this lemma to the term appearing in Proposition~\ref{prop:compositions of intertwining operators}.
So if we can take a suitable $\lambda$, then it is proved that the intertwining operator obtained in Proposition~\ref{prop:construction of intertwining operators} is injective.
To take $\lambda$, we use the following lemma.

\begin{lem}\label{lem:can take lambda s.t. chi_Theta(lambda) ne 0}
Let $w\in W$ and $\alpha\in \Delta$ such that $s_\alpha w > w$.
If $s_\alpha w(\Theta)\subset \Sigma^+$, then $w^{-1}(\alpha)\not\in\Z\Theta$.
\end{lem}
\begin{proof}
By the condition, $w^{-1}(\alpha) > 0$.
If $w^{-1}(\alpha)\in \Z\Theta$, then $w^{-1}(\alpha)\in \Z\Theta \cap \Sigma^+$.
Hence $-\alpha = s_\alpha w(w^{-1}(\alpha))\in s_\alpha w(\Z\Theta\cap\Sigma^+)\subset\Sigma^+$.
This is a contradiction.
\end{proof}

\begin{prop}\label{prop:injectivity of intertwining operator}
Let $w\in W$ and $s\in S$ such that $s w(\Theta)\subset\Sigma^+$.
Then the homomorphism $\Phi_{s w, w}\colon w\chi_\Theta\otimes_\mathcal{A}\mathcal{H}\to s w\chi_\Theta \otimes_\mathcal{A}\mathcal{H}$ obtained in Proposition~\ref{prop:construction of intertwining operators} is injective.
\end{prop}
\begin{proof}
Take $\alpha\in\Delta$ such that $s = s_\alpha$.
By the above lemma, $w^{-1}(\alpha)\not\in\Z\Theta$.
Hence we can take $\lambda_1$ from the center of $\Lambda(1)$ such that $n_w^{-1}(\lambda_1)\in\Lambda^+_\Theta(1)$ and $\langle w^{-1}(\alpha),\nu(n_w^{-1}(\lambda_1))\rangle \ge 2$.
Then $\lambda = \lambda_1^2$ satisfies the condition of Proposition~\ref{prop:hard intertwining operator} and \ref{prop:compositions of intertwining operators}.
Hence we have a homomorphism $s w\chi_\Theta\otimes_\mathcal{A}\mathcal{H}\to w\chi_\Theta\otimes_\mathcal{A}\mathcal{H}$ and the composition with $\Phi_{s w,w}$ is given in Proposition~\ref{prop:compositions of intertwining operators}.
By Proposition~\ref{prop:freeness} and Lemma~\ref{lem:not a zero divisor}, the composition is injective.
Hence $\Phi_{sw,w}$ is injective.
\end{proof}
For $w\in W$, put $\Delta_w = \{\alpha\in\Delta\mid w(\alpha) > 0\}$.
For $\Theta'\subset\Delta$, let $w_{\Theta'}$ be the longest element in $W_{\Theta'}$.
Notice that $\Delta_{w_\Delta w_{\Theta'}} = \Theta'$.
We have $\Theta\subset\Delta_w$ if and only if $w(\Theta)\subset\Sigma^+$.
We prove the following theorem.
\begin{thm}\label{thm:std mod depends only on Delta_w}
Let $w,w'\in W$.
Assume that $\Delta_w = \Delta_{w'}$ and $\Theta\subset\Delta_w$.
Then $w\chi_\Theta\otimes_\mathcal{A}\mathcal{H}\simeq w'\chi_\Theta\otimes_\mathcal{A}\mathcal{H}$.
\end{thm}
We need lemmas to prove the theorem.
\begin{lem}\label{lem:certain element is integral}
Let $w\in W$ and $\lambda\in \Lambda$ such that for any $\alpha\in\Sigma^+$, if $w(\alpha) < 0$, then $\langle\nu(\lambda),\alpha\rangle > 0$.
Then we have $T_{n_{w}}^{-1}E(n_w(\lambda))\in \mathcal{H}$.
\end{lem}
\begin{proof}
We have $T_{n_w}^{-1} = q_{n_w}^{-1}T_{n_w^{-1}}^* = q_{n_w}^{-1}E_{w(-\Delta)}(n_w^{-1})$.
Hence $T_{n_w}^{-1}E(n_w(\lambda)) = q_{n_w}^{-1/2}q_{n_w(\lambda)}^{1/2}q_{\lambda n_w^{-1}}^{-1/2}E_{w(-\Delta)}(\lambda n_w^{-1})$ by Proposition~\ref{prop:multiplication formula}.
By \eqref{eq:length formula}, we have
\begin{align*}
\ell(\lambda n_w^{-1})
& = \sum_{\alpha\in\Sigma^+,w(\alpha) > 0}\lvert\langle\nu(\lambda),\alpha\rangle\rvert + \sum_{\alpha\in\Sigma^+,w(\alpha) < 0}\lvert\langle\nu(\lambda),\alpha\rangle - 1\rvert\\
& = \sum_{\alpha\in\Sigma^+,w(\alpha) > 0}\lvert\langle\nu(\lambda),\alpha\rangle\rvert + \sum_{\alpha\in\Sigma^+,w(\alpha) < 0}(\lvert\langle\nu(\lambda),\alpha\rangle\rvert - 1)\\
& = \ell(\lambda) - \ell(w).
\end{align*}
Hence we have $q_{n_w(\lambda)} = q_\lambda = q_{\lambda n_w^{-1}}q_{n_w}$.
Therefore we get $T_{n_w}^{-1}E(n_w(\lambda)) = E_{w(-\Delta)}(\lambda n_w^{-1})\in \mathcal{H}$.
\end{proof}

\begin{lem}[{\cite[Lemma~5.17, Lemma~5.18]{MR2728487}}]\label{lem:nice choice of s}
Let $w\in W$ such that $w\ne w_\Delta w_{\Delta_w}$.
Then there exists $\alpha\in\Delta$ such that $s_\alpha w > w$, $w^{-1}(\alpha)$ is not a simple root and $\Delta_w = \Delta_{s_\alpha w}$.
\end{lem}

\begin{proof}
There exists $\alpha\in\Delta$ such that $s_\alpha ww_{\Delta_w}^{-1} > ww_{\Delta_w}^{-1}$.
If $s_\alpha w < w$, then $w^{-1}(\alpha) < 0$.
Since $w_{\Delta_w}w^{-1}(\alpha) > 0$, $w^{-1}(\alpha)\in \Sigma^-\cap \Z\Delta_w$.
Hence $\alpha\in w(\Sigma^-\cap \Z\Delta_w)\subset\Sigma^-$, this is a contradiction.
Therefore $s_\alpha w > w$.
If $w^{-1}(\alpha)$ is a simple root, then $w^{-1}(\alpha)\in\Delta_w$.
Hence $w_{\Delta_w}w^{-1}(\alpha) < 0$.
This is a contradiction.
Hence $w^{-1}(\alpha)$ is not a simple root.
Therefore, for any $\beta\in\Delta$, $w(\beta)\ne \pm \alpha$.
Hence we have $w(\beta) > 0$ if and  only if $s_\alpha w(\beta) > 0$.
Therefore $\Delta_w = \Delta_{s_\alpha w}$.
\end{proof}
\begin{rem}\label{rem:on weyl elements}
From this lemma, $w_\Delta w_{\Theta'}$ is the maximal element in $\{w\in W\mid \Delta_w = \Theta'\}$.
If $\Theta''\supset \Theta$, then $w_\Delta w_{\Theta''} < w_\Delta w_{\Theta'}$.
Hence $w_\Delta w_{\Theta'}$ is also the maximal element in $\{w\in W\mid \Delta_w\supset \Theta'\}$.
\end{rem}

\begin{proof}[Proof of Theorem~\ref{thm:std mod depends only on Delta_w}]
We may assume $w' = w_\Delta w_{\Delta_w}$.
By Lemma~\ref{lem:nice choice of s}, there exists $\alpha_1,\dots,\alpha_l$ such that
\begin{itemize}
\item $s_{\alpha_i}\dotsm s_{\alpha_l}w > s_{\alpha_{i + 1}}\dotsm s_{\alpha_l}w$.
\item $(s_{\alpha_{i + 1}}\dotsm s_{\alpha_l}w)^{-1}(\alpha_i)$ is not simple.
\item $w_\Delta w_{\Delta_w} = s_{\alpha_1}\dotsm s_{\alpha_l}w$.
\end{itemize}
Set $s_i = s_{\alpha_i}$.
By Proposition~\ref{prop:construction of intertwining operators}, we have the homomorphism $\Phi_{w_\Delta w_{\Delta_w},w}\colon w\chi_\Theta\otimes_\mathcal{A}\mathcal{H}\to w_\Delta w_{\Delta_w}\chi_\Theta\otimes_\mathcal{A}\mathcal{H}$ defined by $\Phi_{w_\Delta w_{\Delta_w},w}(1\otimes 1) = 1\otimes T_{n_{s_1\dotsm s_l}}^*$.
This is injective by Proposition~\ref{prop:injectivity of intertwining operator}.
We prove that $\Phi_{w_\Delta w_{\Delta_w},w}$ is surjective.
\begin{claim}
We have $1\otimes T_{n_{s_1\dotsm s_{l - 1}}}^*T_{n_w} = 0$ in $w_\Delta w_{\Delta_w}\chi_\Theta\otimes_\mathcal{A}\mathcal{H}$.
\end{claim}
It is sufficient to prove that $\Phi_{w_\Delta ,w_\Delta w_{\Delta_w}}(1\otimes T_{n_{s_1\dotsm s_{l - 1}}}^*T_{n_{w}}) = 0$ by Proposition~\ref{prop:injectivity of intertwining operator}.
Put $s = s_l$.
Then $\Phi_{w_\Delta ,w_\Delta w_{\Delta_w}}(1\otimes T_{n_{s_1\dotsm s_{l - 1}}}^*T_{n_{w}}) = 1\otimes T_{n_{w_\Delta w^{-1}s}}^*T_{n_{w}}$.
We have $T_{n_{w_\Delta w^{-1}s}}^* = E_{w_\Delta w^{-1}s(-\Delta)}(n_{w_\Delta w^{-1}s})$ and $T_{n_{w}} = E_{-\Delta}(n_{w})$.
Therefore we have $T_{n_{w_\Delta w^{-1}s}}^*T_{n_{w}} = q_{n_{w_\Delta w^{-1}s}}^{1/2}q_{n_w}^{1/2}q_{n_{w_\Delta w^{-1}sw}}^{-1/2}E_{w^{-1}s(-\Delta)}(n_{w_\Delta w^{-1}s}n_{w})$.
It is sufficient to prove that $\ell(w_\Delta w^{-1}s) + \ell(w) > \ell(w_\Delta w^{-1}sw)$.
Since the positive root $w^{-1}(\alpha_l)$ is not simple, $\ell(s_{w^{-1}(\alpha_l)}) > 1$.
Hence
\begin{align*}
\ell(w_\Delta w^{-1}sw) & = \ell(w_\Delta s_{w^{-1}(\alpha_l)})\\
& = \ell(w_\Delta) - \ell(s_{w^{-1}(\alpha_l)})\\
& < \ell(w_\Delta) - 1\\
& = (\ell(w_\Delta w^{-1}) - 1) + \ell(w)\\
& = \ell(w_\Delta w^{-1}s) + \ell(w).
\end{align*}
We get the claim.

By the claim, $1\otimes T_{n_{s_1\dotsm s_{l - 1}}}^*c_{n_{s_l}}T_{n_w} = 1\otimes T_{n_{s_1\dotsm s_{l - 1}}}^*T_{n_w}(n_{w}^{-1}\cdot c_{n_{s_l}}) = 0$.
Notice that $s_lw > w$.
Hence
\begin{align*}
1\otimes T_{n_{s_1\dotsm s_l}}^*T_{n_{w}} & =
1\otimes T_{n_{s_1\dotsm s_{l - 1}}}^*(T_{n_{s_l}} - c_{n_{s_l}})T_{n_{w}}\\
& = 1\otimes T_{n_{s_1\dotsm s_{l - 1}}}^*T_{n_{s_lw}} - 1\otimes T_{n_{s_1\dotsm s_{l - 1}}}^*c_{n_{s_l}}T_{w}\\
& = 1\otimes T_{n_{s_1\dotsm s_{l - 1}}}^*T_{n_{s_lw}}.
\end{align*}
By induction on $l$, we get
\[
	1\otimes T_{n_{s_1\dotsm s_{l}}}^*T_{n_{w}} = 1\otimes T_{n_{s_1\dotsm s_lw}} = 1\otimes T_{n_{w_\Delta w_{\Delta_w}}}.
\]
Hence $1\otimes T_{n_{w_\Delta w_{\Delta_w}}}\in \Imm\Phi_{w_\Delta w_{\Delta_w},w}$.

Let $\lambda\in\Lambda(1)$ such that $\langle \nu(\lambda),\alpha\rangle = 0$ for all $\alpha\in\Delta_w$ and $\langle\nu(\lambda),\alpha\rangle > 0$ for all $\alpha\in\Delta\setminus\Delta_w$.
Since $\Theta\subset\Delta_w$, we have $\lambda\in\Lambda_\Theta^+(1)$.
Hence $\chi_\Theta(E(\lambda)) = \tau_\lambda$ is invertible.
Assume that $\alpha\in\Sigma^+$ satisfies $w_{\Delta}w_{\Delta_w}(\alpha) < 0$.
Then $\alpha\in\Sigma^+\setminus\Z_{\ge 0} \Delta_w$.
Therefore $\langle \nu(\lambda),\alpha\rangle > 0$.
Hence $\lambda$ satisfies the condition of Lemma~\ref{lem:certain element is integral}.
Therefore $T_{n_{w_\Delta w_{\Delta_w}}}^{-1}E(n_{w_\Delta w_{\Delta_w}}(\lambda))\in \mathcal{H}$.
Hence
\[
	\tau_\lambda^{-1}\otimes T_{n_{w_\Delta w_{\Delta_w}}}T_{n_{w_\Delta w_{\Delta_w}}}^{-1}E(n_{w_\Delta w_{\Delta_w}}(\lambda)) = \tau_\lambda^{-1}\otimes E(n_{w_\Delta w_{\Delta_w}}(\lambda)) = 1\otimes 1
\]
is in $\Imm\Phi_{w_\Delta w_{\Delta_w},w}$.
Since $1\otimes 1$ generates $w_\Delta w_{\Delta_w}\chi_\Theta\otimes_\mathcal{A}\mathcal{H}$, $\Phi_{w_\Delta w_{\Delta_w},w}$ is surjective.
\end{proof}
\begin{cor}\label{cor:isom between std mod for simple module}
Let $\Theta\subset\Delta$ be a subset and $X$ an irreducible $C[\Lambda_\Theta(1)]$-module.
We regard $X$ as an irreducible $\mathcal{A}$-module via $\chi_\Theta$.
Set
\[
	\Delta(X) = \{\alpha\in\Delta\mid \langle\Theta,\check{\alpha}\rangle = 0,\ \text{$\tau_\lambda\in C[\Lambda_\Theta(1)]$ is identity on $X$ for $\lambda\in\Lambda_{s_\alpha}'(1)$}\}\cup \Theta.
\]
Assume that $w,w'\in W$ satisfies $\Delta_w\cap \Delta(X) = \Delta_{w'}\cap \Delta(X)$.
Then $wX\otimes_\mathcal{A}\mathcal{H}\simeq w'X\otimes_\mathcal{A}\mathcal{H}$.
\end{cor}
\begin{proof}
First notice that if $\alpha\in\Delta$ and $\lambda\in \Lambda'_{s_\alpha}(1)$, then $\nu(\lambda)\in \R\check{\alpha}$.
Hence if $\langle \Theta,\check{\alpha}\rangle = 0$, $\langle\Theta,\nu(\lambda)\rangle = 0$.
Therefore $\lambda\in \Lambda_\Theta(1)$.
Hence $\Lambda'_{s_\alpha}(1)\subset \Lambda_\Theta(1)$.
So the definition makes sense.

We may assume that there exists $\alpha\in \Delta_w$ such that $\Delta_{w}= \Delta_{w'}\amalg\{\alpha\}$.
By Theorem~\ref{thm:std mod depends only on Delta_w}, we may assume $w = w_\Delta w_{\Delta_w}$ and $w' = ws_\alpha$.
Put $s = s_\alpha$, $\alpha' = w(\alpha)\in \Delta$ and $s' = s_{\alpha'}$.
By the assumption, we have $\alpha\notin\Theta$.
Hence we can take a dominant $\lambda_1\in \Lambda_\Theta(1)$ from the center of $\Lambda(1)$ such that $\langle\nu(\lambda_1),\alpha\rangle \ge 2$.
Put $\lambda = \lambda_1^2$.
Then $n_w(\lambda)$ satisfies the conditions of Proposition~\ref{prop:compositions of intertwining operators} for $w$ and $s'$.
Hence we have a homomorphism $s'w\widetilde{\chi}\otimes_\mathcal{A}\mathcal{H}\to w\widetilde{\chi}\otimes_\mathcal{A}\mathcal{H}$ and $w\widetilde{\chi}\otimes_\mathcal{A}\mathcal{H}\to s'w\widetilde{\chi}\otimes_\mathcal{A}\mathcal{H}$.
These homomorphisms induce $s'wX\otimes_\mathcal{A}\mathcal{H}\to wX\otimes_\mathcal{A}\mathcal{H}$ and $wX\otimes_\mathcal{A}\mathcal{H}\to s'wX\otimes_\mathcal{A}\mathcal{H}$, respectively.
The compositions are induced by $X\ni x\mapsto x\chi_\Theta(E(\lambda) - E(\lambda \mu_{n_s}(-1))c_sc_{s,-1})\in X$.
Let $\varphi\colon X\to X$ be this homomorphism.
We prove that it is an isomorphism.

By the condition, $\alpha\notin \Delta(X)$.
Hence $\langle\Theta,\check{\alpha}\rangle\ne 0$ or $\{\tau_\lambda \mid \lambda\in\Lambda_s'(1)\}$ is not trivial on $X$.
If $\langle\Theta,\check{\alpha}\rangle\ne 0$, then $\lambda \mu_{n_s^{-1}}(-1)\notin \Lambda_\Theta(1)$.
Hence $\chi_\Theta(E(\lambda \mu_{n_s^{-1}}(-1))) = 0$.
Therefore $\varphi(x) = x\chi_\Theta(E(\lambda)) = x\tau_\lambda$ on $X$, which is invertible.

Assume that $\langle\Theta,\check{\alpha}\rangle = 0$.
Then $\lambda \mu_{n_s}(-1)\in \Lambda^+_\Theta(1)$.
We have $\varphi(x) = x\tau_\lambda(1 - \tau_{\mu_{n_s^{-1}}(-1)}c_sc_{s,-1})$, here we regard $C[Z_\kappa]$ as a subalgebra of $C[\Lambda_\Theta(1)]$.

Assume that there exists $t\in \Lambda_s'(1)\cap Z_\kappa$ such that $\tau_t$ is not identity on $X$.
Namely, $Y = \{x\in X\mid \text{$x\tau_t = x$ for any $t\in \Lambda_s'(1)\cap Z_\kappa$}\}\ne X$.
Since $\Lambda_s'(1)\cap Z_\kappa\subset \Lambda(1)$ is a normal subgroup, the subspace $Y$ is $C[\Lambda_\Theta(1)]$-stable.
Hence, by the irreducibility of $X$, $Y = 0$.
Since $\#(\Lambda'_s(1)\cap Z_\kappa)$ is prime to $p$, we have a decomposition $X = \bigoplus_{\chi_\kappa} X_{\chi_\kappa}$ where $\chi_\kappa$ runs characters of $\Lambda_s'(1)\cap Z_\kappa$.
Then the trivial character does not appear in $X$.
Therefore, by Lemma~\ref{lem:c_s,c_{s,-1},mu_n_s and Lambda_s'(1) cap Z_k}, $c_sc_{s,-1}$ is zero on $X$.
Hence $\tau_\lambda - \tau_{\lambda \mu_{n_s}(-1)}c_sc_{s,-1} = \tau_\lambda$ on $X$, which is invertible.

Finally, assume that $\langle\Theta,\check{\alpha}\rangle = 0$ and $E(t)$ for $t\in \Lambda_s'(1)\cap Z_\kappa$ acts trivially on $X$.
Then $c_s = c_{s,-1} = -1$ on $X$.
Hence $\varphi(x) = x\tau_\lambda (1 - \tau_{\mu_{n_s}(-1)})$.
Since the group $\Lambda_s'(1)$ is generated by $\Lambda_s'(1)\cap Z_\kappa$ and $\mu_{n_s^{-1}}(-1)$, $\tau_{\mu_{n_s^{-1}}(-1)}$ is not identity.
Hence $\tau_\lambda(1 - \tau_{\mu_{n_s^{-1}}}(-1))\ne 0$ on $X$.
Since $\varphi$ is an homomorphism between irreducible modules, it is an isomorphism.
\end{proof}

\section{Classification theorem}\label{sec:Classification theorem}
\subsection{pro-$p$-Iwahori Hecke algebra of a Levi subgroup}
\label{subsec:pro-$p$-Iwahori Hecke algebra of a Levi subgroup}
Let $\alggrp{M}$ be a Levi subgroup of a standard parabolic subgroup $\alggrp{P}$ such that $\alggrp{Z}\subset\alggrp{M}$.
There is a pro-$p$-Iwahori Hecke algebra for $\alggrp{M}$.
This is an algebra over $C[q_{\alggrp{M},s}^{1/2}]$ and under the specialization $q_{\alggrp{M},s}\mapsto 0$, $\mathcal{H}_{\alggrp{M}}$ is isomorphic to the pro-$p$-Iwahori Hecke algebra of $\alggrp{M}$ with the pro-$p$-Iwahori subgroup $I_{\alggrp{M}}(1) = I(1)\cap \alggrp{M}(F)$.
Later we will construct an algebra homomorphism $C[q_{\alggrp{M},s}^{1/2}]\to C[q_s^{1/2}]$ and we denote the base change of the pro-$p$-Iwahori Hecke algebra of $\alggrp{M}$ to $C[q_s^{1/2}]$ by $\mathcal{H}_{\alggrp{M}}$.
Its Iwahori-Matsumoto basis is denoted by $T^{\alggrp{M}}_{\widetilde{w}}$ and the Bernstein basis is denoted by $E^{\alggrp{M}}(\widetilde{w})$.
The objects for $\alggrp{M}$ is denoted with the suffix `$\alggrp{M}$'.
The finite Weyl group of $\alggrp{M}$ is denoted by $W_{\alggrp{M}}$, etc.
We also write `$\alggrp{P}$' instead of $\alggrp{M}$, $W_{\alggrp{P}} = W_{\alggrp{M}}$, etc.

As in \cite[2C4]{MR2728487}, an element $\widetilde{w}\in \widetilde{W}_{\alggrp{M}}(1)$ is called $\alggrp{M}$-positive (resp.\ $\alggrp{M}$-negative) if $\widetilde{w}((\Sigma^+\setminus\Sigma_M^+)\times\{0\})\subset\Sigma_\aff^+$ (resp.\ $\widetilde{w}^{-1}((\Sigma^+\setminus\Sigma_M^+)\times\{0\})\subset\Sigma_\aff^+$).
It is easy to check that, for $\lambda\in\Lambda(1)$ and $w\in W_{\alggrp{M}}$, $\lambda n_w\in \widetilde{W}_{\alggrp{M}}(1)$ is $\alggrp{M}$-positive (resp.\ $\alggrp{M}$-negative) if and only if $\langle\nu(\lambda),\alpha\rangle\le 0$ (resp.\ $\langle\nu(\lambda),\alpha\rangle\ge 0$) for any $\alpha\in\Sigma^+\setminus\Sigma_{\alggrp{M}}^+$.
In particular, for $w_1,w_2\in W_{\alggrp{M}}$ and $\widetilde{w}\in W_{\alggrp{M}}(1)$, $\widetilde{w}$ is $\alggrp{M}$-positive (resp.\ $\alggrp{M}$-negative) if and only if $n_{w_1}\widetilde{w}n_{w_2}$ is $\alggrp{M}$-positive (resp.\ $\alggrp{M}$-negative).
The product of $\alggrp{M}$-positive (resp.\ $\alggrp{M}$-negative) elements is $\alggrp{M}$-positive (resp.\ $\alggrp{M}$-negative).

\begin{lem}\label{lem:Bruhat order, M-positive/negative}
If $\widetilde{w}\in W_{\alggrp{M}}(1)$ is $\alggrp{M}$-positive (resp.\ $\alggrp{M}$-negative) and $\widetilde{v} \le \widetilde{w}$ in $\widetilde{W}_{\alggrp{M}}(1)$, then $\widetilde{v}$ is $\alggrp{M}$-positive (resp.\ $\alggrp{M}$-negative).
\end{lem}
\begin{proof}
If $\widetilde{w}\in \Lambda(1)$, then it is \cite[Fact ii]{arXiv:1211.5366}.
In general, let $\lambda\in\Lambda(1)$ and $w\in W_{\alggrp{M}}$ such that $\widetilde{w} = \lambda n_{w}$.
We prove the lemma by induction on $\ell_{\alggrp{M}}(w)$.
Let $s\in S_{\alggrp{M}}$ such that $ws < w$.
Then $\widetilde{v}\le \lambda n_{w}$ implies $\widetilde{v}$ or $\widetilde{v}n_s$ is less than or equal to $\lambda n_{ws}$.
Hence $\widetilde{v}$ or $\widetilde{v}n_s$ is $\alggrp{M}$-positive.
Therefore $\widetilde{v}$ is $\alggrp{M}$-positive.
Since we have $\widetilde{v}\le \widetilde{w}$ if and only if $\widetilde{v}^{-1}\le \widetilde{w}^{-1}$, if $\widetilde{w}\in W_{\alggrp{M}}(1)$ is $\alggrp{M}$-negative, then $\widetilde{v}$ is $\alggrp{M}$-negative.
\end{proof}

Assume that an algebra homomorphism $C[q_{\alggrp{M},s}^{1/2}]\to C[q_s^{1/2}]$ is given.
Let $\mathcal{H}_{\alggrp{M}}^+$ (resp.\ $\mathcal{H}_{\alggrp{M}}^-$) be the sub $C[q_{s}^{1/2}]$-module of $\mathcal{H}_{\alggrp{M}}$ spanned by $\{T^{\alggrp{M}}_{\widetilde{w}}\}$ where $\widetilde{w}\in \widetilde{W}_{\alggrp{M}}(1)$ is $\alggrp{M}$-positive (resp.\ $\alggrp{M}$-negative).
By the above lemma, $\mathcal{H}_{\alggrp{M}}^{\pm}$ is a subalgebra.
\begin{lem}
For any set of simple roots $\Delta'$,
\[
	\{E^{\alggrp{M}}_{\Delta'}(\widetilde{w})\mid \text{\normalfont $\widetilde{w}\in \widetilde{W}_{\alggrp{M}}(1)$ is $\alggrp{M}$-positive (resp.\ $\alggrp{M}$-negative)}\}.
\]
is a basis of $\mathcal{H}_{\alggrp{M}}^+$ (resp.\ $\mathcal{H}_{\alggrp{M}}^-$).
\end{lem}
\begin{proof}
Since $E^{\alggrp{M}}_{\Delta'}(\widetilde{w}) \in T^{\alggrp{M}}_{\widetilde{w}} + \sum_{\widetilde{v} < \widetilde{w}}C[q_{s}^{1/2}]T^{\alggrp{M}}_{\widetilde{v}}$, this follows from Lemma~\ref{lem:Bruhat order, M-positive/negative}.
\end{proof}

We have the Bernstein basis $\{E(\widetilde{w})\mid \widetilde{w}\in \widetilde{W}(1)\}$.
We need another basis $\{E_-(\widetilde{w})\mid \widetilde{w}\in \widetilde{W}(1)\}$ defined by
\[
	E_-(n_w \lambda) = q_{n_w \lambda}^{1/2}q_{n_w}^{-1/2}T_{n_w}^*\theta(\lambda)
\]
for $w\in W$ and $\lambda\in\Lambda(1)$.

Recall that we have an anti-involution $\iota$ of $\mathcal{H}$ defined by $T_{\widetilde{w}}\mapsto T_{\widetilde{w}^{-1}}$.
To see that this is an anti-involution, we need to check that the braid relations and quadratic relations are preserved by $\iota$.
For the braid relations, let $w_1,w_2\in \widetilde{W}(1)$ such that $\ell(w_1w_2) = \ell(w_1) + \ell(w_2)$.
Then we have $\ell((w_1w_2)^{-1}) = \ell(w_2^{-1}) + \ell(w_1^{-1})$.
Hence $T_{(w_1w_2)^{-1}} = T_{w_2^{-1}}T_{w_1^{-1}}$.
Namely the braid relations are preserved by $\iota$.
For the quadratic relations, we have $T_{n_s^{-1}}^2 = q_sT_{n_s^{-1}} + T_{n_s^{-1}}c_s$ by the quadratic relation for $n_s^{-1}$ and $c_sT_{n_s^{-1}} = T_{n_s^{-1}}c_s$~\cite[Remark~4.9 (b)]{Vigneras-prop}.
This means that $\iota$ preserves the quadratic relations.
\begin{lem}
$E_-(\widetilde{w}) = \iota(E_{\Delta}(\widetilde{w}^{-1}))$ for $\widetilde{w}\in \widetilde{W}(1)$.
In particular, $\{E_-(\widetilde{w})\mid \widetilde{w}\in\widetilde{W}(1)\}$ is a $C[q_s^{1/2}]$-basis of $\mathcal{H}$.
\end{lem}
\begin{proof}
Let $\lambda \in \Lambda(1)$ and take anti-dominant $\lambda_1,\lambda_2\in \Lambda(1)$ such that $\lambda = \lambda_2\lambda_1^{-1}$.
Then we have
\[
E_\Delta(\lambda^{-1}) = E_\Delta(\lambda_1\lambda_2^{-1}) = q_{\lambda_1}^{-1/2}q_{\lambda_2}^{-1/2}q_{\lambda}^{1/2}E_\Delta(\lambda_1)E_\Delta(\lambda_2^{-1})
\]
by Proposition~\ref{prop:multiplication formula}.
By Proposition~\ref{prop:example of E}, we have $E_\Delta(\lambda_1) = T_{\lambda_1}^* = q_{\lambda_1}T_{\lambda_1^{-1}}^{-1}$ and $E_\Delta(\lambda_2^{-1}) = T_{\lambda_2^{-1}}$.
Hence we get
\[
	E_\Delta(\lambda^{-1}) = q_{\lambda_1}^{1/2}q_{\lambda_2}^{-1/2}q_{\lambda}^{1/2}T_{\lambda_1^{-1}}^{-1}T_{\lambda_2^{-1}}.
\]
Therefore
\[
	\iota(E_\Delta(\lambda^{-1})) = q_{\lambda_1}^{1/2}q_{\lambda_2}^{-1/2}q_\lambda^{1/2}T_{\lambda_2}T_{\lambda_1}^{-1}.
\]
By the definition of $\theta(\lambda)$ and Proposition~\ref{prop:example of E}, we have $T_{\lambda_2} = q_{\lambda_2}^{1/2}\theta(\lambda_2)$ and $T_{\lambda_1} = q^{1/2}_{\lambda_1}\theta(\lambda_1)$.
Hence
\[
	\iota(E_\Delta(\lambda^{-1})) = q_{\lambda}^{1/2}\theta(\lambda_2)\theta(\lambda_1)^{-1} = q_\lambda^{1/2}\theta(\lambda_2\lambda_1^{-1}) = q_\lambda^{1/2}\theta(\lambda).
\]
Now let $w\in W$.
Then we have $E_\Delta(\lambda^{-1} n_w^{-1}) = q_\lambda^{-1/2}q_{w}^{-1/2}q_{n_w \lambda}^{1/2}E_\Delta(\lambda^{-1}) E_\Delta(n_w^{-1})$ (Proposition~\ref{prop:multiplication formula}).
We have $E_\Delta(n_w^{-1}) = T_{n_w^{-1}}^* = q_{n_w}T_{n_w}^{-1}$ by Proposition~\ref{prop:example of E}.
Hence $\iota(E_\Delta(n_w^{-1})) = q_{n_w}T_{n_w^{-1}}^{-1} = T_{n_w}^*$.
Therefore we get
\begin{align*}
\iota(E_\Delta((n_w\lambda)^{-1})) & = q_{\lambda}^{-1/2}q_{n_w}^{-1/2}q_{n_w\lambda}^{1/2}\iota(E_\Delta(n_w^{-1}))\iota(E_\Delta(\lambda^{-1}))\\
& = q_{n_w}^{-1/2}q_{n_w\lambda}^{1/2}T_{n_w}^*\theta(\lambda) = E_-(n_w\lambda).\qedhere
\end{align*}
\end{proof}
Hence  $\{E_-(\widetilde{w})\mid \widetilde{w}\in \widetilde{W}(1)\}$ is a $C[q_s^{1/2}]$-basis of $\mathcal{H}$.
We have $E_{\Delta}(\widetilde{w}^{-1}) \in T_{\widetilde{w}^{-1}} + \sum_{\widetilde{v}  <\widetilde{w}}C[q_s^{1/2}]T_{\widetilde{v}^{-1}}$.
Applying $\iota$, we get
\[
	E_-(\widetilde{w})\in T_{\widetilde{w}} + \sum_{\widetilde{v} < \widetilde{w}}C[q_s^{1/2}]T_{\widetilde{v}} = E(\widetilde{w}) + \sum_{\widetilde{v} < \widetilde{w}}C[q_s^{1/2}]E(\widetilde{v}).
\]
By the definition, for $w\in W$ and $\lambda,\lambda_1\in\Lambda(1)$, we have
\[
	E_-(n_w\lambda)E(\lambda_1) = q_{n_w\lambda}^{1/2}q_{\lambda_1}^{1/2}q_{n_w \lambda \lambda_1}^{-1/2}E_-(n_w \lambda \lambda_1).
\]
In particular, in $\mathcal{H}\otimes C$, we get
\begin{equation}\label{eq:multiplication of E_-}
	E_-(n_w\lambda)E(\lambda_1)
	=
	\begin{cases}
	E_-(n_w \lambda \lambda_1) & (\ell(n_w \lambda) + \ell(\lambda_1) = \ell(n_w \lambda \lambda_1)),\\
	0 & (\text{otherwise}).
	\end{cases}
\end{equation}

We want to construct an embedding $\mathcal{H}_{\alggrp{M}}^{\pm}\hookrightarrow \mathcal{H}$.
To do this, we have to relate $\{q_s\mid s\in S_\aff/\mathord{\sim}\}$ with $\{q_{\alggrp{M},s}\mid s\in S_{\alggrp{M},\aff}/\mathord{\sim}\}$.
For $\widetilde{\alpha}\in \Sigma_\aff$, take a simple affine root $\widetilde{\beta}\in \widetilde{W}\widetilde{\alpha}$ from the $\widetilde{W}$-orbit through $\widetilde{\alpha}$ and put $q(\widetilde{\alpha}) = q_{s_{\widetilde{\beta}}}$.
Then it does not depend on $\widetilde{\beta}$.
For a finite subset $A\subset\Sigma_\aff$, put $q(A) = \prod_{\widetilde{\alpha}\in A}q(\widetilde{\alpha})$.
By induction on $\ell(\widetilde{w})$, we have $q_{\widetilde{w}} = q(\Sigma_\aff^+\cap \widetilde{w}(\Sigma_\aff^-))$.
Now we define a homomorphism $C[q_{\alggrp{M},s}^{1/2}]\to C[q_s^{1/2}]$ by $q_{\alggrp{M},\widetilde{w}}\mapsto q(\Sigma_{\alggrp{M},\aff}^+\cap \widetilde{w}(\Sigma^-_{\alggrp{M},\aff}))$ for $\widetilde{w}\in \widetilde{W}_{\alggrp{M}}(1)$.
The motivation of this definition is the following.
\begin{lem}
Under the specialization $q_s\mapsto \#(I(1)sI(1)/I(1))$, we have $q_{\alggrp{M},\widetilde{w}}\mapsto \#(\overline{I}_{\alggrp{M}}(1)\widetilde{w}\overline{I}_{\alggrp{M}}(1)/\overline{I}_{\alggrp{M}}(1))$.
\end{lem}
\begin{proof}
The proof of Corollary 3.31 and Proposition 3.28 in \cite{Vigneras-prop} is applicable.
\end{proof}

We have the following relation between $q_s$ and $q_{\alggrp{M},s}$.
\begin{lem}\label{lem:M-negativity and q}
If $\widetilde{w}_1,\widetilde{w}_2\in \widetilde{W}_{\alggrp{M}}(1)$ are $\alggrp{M}$-positive (resp.\ $\alggrp{M}$-negative), then we have $q_{\widetilde{w}_1\widetilde{w}_2}^{-1}q_{\widetilde{w}_1}q_{\widetilde{w}_2} = q_{\alggrp{M},\widetilde{w}_1\widetilde{w}_2}^{-1}q_{\alggrp{M},\widetilde{w}_1}q_{\alggrp{M},\widetilde{w}_2}$.
In particular, we have $\ell_{\alggrp{M}}(\widetilde{w}_1) + \ell_{\alggrp{M}}(\widetilde{w}_2) - \ell_{\alggrp{M}}(\widetilde{w}_1\widetilde{w}_2) = \ell(\widetilde{w}_1) + \ell(\widetilde{w}_2) - \ell(\widetilde{w}_1\widetilde{w}_2)$.
\end{lem}
\begin{proof}
Put $C^\pm = \Sigma_\aff^{\pm}\setminus\Sigma_{\alggrp{M},\aff}^{\pm}$.
By the definition, we have $q_{\widetilde{w}} = q_{\alggrp{M},\widetilde{w}}q(C^+\cap \widetilde{w}C^-)$ for $\widetilde{w}\in\widetilde{W}(1)$.
Hence it is sufficient to prove that
\[
	q(C^+\cap \widetilde{w}_1\widetilde{w}_2C^-) = q(C^+\cap \widetilde{w}_1C^-)q(C^+\cap \widetilde{w}_2C^-).
\]
Define $A^{\pm} = C^\pm\cap ((\Sigma^+\setminus\Sigma_{\alggrp{M}}^+)\times \Z)$ and $B^{\pm} = C^\pm\cap ((\Sigma^-\setminus\Sigma_{\alggrp{M}}^-)\times \Z)$.
Then we have $C^\pm = A^\pm \cup B^\pm$.
For $\lambda\in \Lambda(1)$ and $w\in W_{\alggrp{M}}$, we have $\lambda n_w(\alpha,k) = (w(\alpha),n - \langle w(\alpha),\nu(\lambda)\rangle)$.
If $w\in W_{\alggrp{M}}$ and $\alpha\notin\Sigma_{\alggrp{M}}$, we have $w(\alpha) > 0$ if and only if $\alpha > 0$.
Hence if $\widetilde{w}\in W_{\alggrp{M}}(1)$, we have $\widetilde{w}(A^+\cup A^-) = A^+\cup A^-$ and $\widetilde{w}(B^+\cup B^-) = B^+\cup B^-$.
Hence $C^+\cap \widetilde{w}C^- = (A^+\cap \widetilde{w}A^-)\cup(B^+\cap \widetilde{w}B^-)$.
If $\lambda n_w$ is $\alggrp{M}$-positive then for $\alpha \in \Sigma^-\setminus\Sigma_{\alggrp{M}}^-$, we have $\langle \alpha,\nu(\lambda)\rangle \ge 0$.
Hence $\widetilde{w} = \lambda n_w$ satisfies $\widetilde{w}(B^-)\subset B^-$.
Therefore $C^+\cap \widetilde{w}C^- = A^+\cap\widetilde{w}A^-$.
We also have $\widetilde{w}(A^+)\subset A^+$ and $\widetilde{w}(A^-)\supset A^-$ which will be used later.

Let $\widetilde{w}_1,\widetilde{w}_2\in \widetilde{W}_{\alggrp{M}}(1)$ be $\alggrp{M}$-positive elements.
Then we have
\[
	A^+\cap \widetilde{w}_1\widetilde{w}_2A^-
	=
	(A^+\cap \widetilde{w}_1\widetilde{w}_2A^-\cap \widetilde{w}_1A^+)
	\amalg
	(A^+\cap \widetilde{w}_1\widetilde{w}_2A^-\cap \widetilde{w}_1A^-).
\]
Since $\widetilde{w}_1A^+\subset A^+$, we have $A^+\cap \widetilde{w}_1\widetilde{w}_2A^-\cap \widetilde{w}_1A^+ = \widetilde{w}_1\widetilde{w}_2A^-\cap \widetilde{w}_1A^+ = \widetilde{w}_1(\widetilde{w}_2A^-\cap A^+)$.
Since $A^-\subset \widetilde{w}_2A^-$, we have $A^+\cap \widetilde{w}_1\widetilde{w}_2A^-\cap \widetilde{w}_1A^- = A^+\cap \widetilde{w}_1A^-$.
Hence
\[
	A^+\cap \widetilde{w}_1\widetilde{w}_2A^-
	=
	\widetilde{w}_1(A^+\cap \widetilde{w}_2A^-)
	\amalg
	(A^+\cap \widetilde{w}_1A^-).
\]
Therefore,
\[
	q(A^+\cap \widetilde{w}_1\widetilde{w}_2A^-)
	=
	q(\widetilde{w}_1(A^+\cap \widetilde{w}_2A^-))
	q(A^+\cap \widetilde{w}_1A^-).
\]
Since $q(\widetilde{w}_1(A^+\cap \widetilde{w}_2A^-)) = q(A^+\cap \widetilde{w}_2A^-)$, we get the lemma if $\widetilde{w}_1,\widetilde{w}_2$ are $\alggrp{M}$-positive.
Since $q_{\widetilde{w}^{-1}} = q_{\widetilde{w}}$ and $q_{\alggrp{M},\widetilde{w}^{-1}} = q_{\alggrp{M},\widetilde{w}}$. the lemma holds if $\widetilde{w}_1,\widetilde{w}_2$ are $\alggrp{M}$-negative.
\end{proof}

We define a $C[q_s^{1/2}]$-module homomorphism $j_{\alggrp{M}}^+\colon \mathcal{H}_{\alggrp{M}}^+\to \mathcal{H}$ (resp.\ $j_{\alggrp{M}}^-\colon \mathcal{H}_{\alggrp{M}}^-\to \mathcal{H}$) by
\[
	j_{\alggrp{M}}^+(T^{\alggrp{M}}_{\widetilde{w}}) = T_{\widetilde{w}}
	\quad \textrm{(resp.\ }
	j_{\alggrp{M}}^-(T^{\alggrp{M},*}_{\widetilde{w}}) = T_{\widetilde{w}}^*
	\textrm{)}
\]
for $\alggrp{M}$-positive (resp.\ $\alggrp{M}$-negative) $\widetilde{w}\in \widetilde{W}_{\alggrp{M}}(1)$.

\begin{lem}
The $C[q_s^{1/2}]$-module homomorphisms $j_{\alggrp{M}}^{\pm}$ are homomorphisms between $C[q_s^{1/2}]$-algebras.
We have $j^+_{\alggrp{M}}(E^{\alggrp{M}}(\widetilde{w})) = E(\widetilde{w})$, $j_{\alggrp{M}}^+(E_-^{\alggrp{M}}(\widetilde{w})) = E_-(\widetilde{w})$ for $\alggrp{M}$-positive $\widetilde{w}\in \widetilde{W}_{\alggrp{M}}(1)$ and $j_{\alggrp{M}}^-(E^{\alggrp{M}}(\widetilde{w})) = E(\widetilde{w})$, $j_{\alggrp{M}}^-(E^{\alggrp{M}}_-(\widetilde{w})) = E_-(\widetilde{w})$ for $\alggrp{M}$-negative $\widetilde{w}\in \widetilde{W}_{\alggrp{M}}(1)$.
\end{lem}
\begin{proof}
We prove $j^+_{\alggrp{M}}(T^{\alggrp{M}}_{\widetilde{w}_1}T^{\alggrp{M}}_{\widetilde{w}_2}) = T_{\widetilde{w}_1}T_{\widetilde{w}_2}$ for $\alggrp{M}$-positive elements $\widetilde{w}_1,\widetilde{w}_2\in \widetilde{W}_{\alggrp{M}}(1)$.
By induction on $\ell_{\alggrp{M}}(\widetilde{w}_2)$, we may assume $\ell_{\alggrp{M}}(\widetilde{w}_2) = 0$ or $\widetilde{w}_2 = n_s$ for $s\in S_{\alggrp{M},\aff}$.
If $\ell_{\alggrp{M}}(\widetilde{w}_2) = 0$, or $\widetilde{w}_2 = n_s$ and $\widetilde{w}_1 n_s > \widetilde{w}_1$, then $\ell_{\alggrp{M}}(\widetilde{w}_1) + \ell_{\alggrp{M}}(\widetilde{w}_2) - \ell_{\alggrp{M}}(\widetilde{w}_1\widetilde{w}_2) = 0$.
Hence $T^{\alggrp{M}}_{\widetilde{w}_1}T^{\alggrp{M}}_{\widetilde{w}_2} = T^{\alggrp{M}}_{\widetilde{w}_1\widetilde{w}_2}$.
We have $\ell(\widetilde{w}_1) + \ell(\widetilde{w}_2) - \ell(\widetilde{w}_1\widetilde{w}_2) = 0$ by Lemma~\ref{lem:M-negativity and q}.
Therefore $T_{\widetilde{w}_1}T_{\widetilde{w}_2} = T_{\widetilde{w}_1\widetilde{w}_2}$.
We have $j^+_{\alggrp{M}}(T^{\alggrp{M}}_{\widetilde{w}_1}T^{\alggrp{M}}_{\widetilde{w}_2}) = T_{\widetilde{w}_1}T_{\widetilde{w}_2}$.

If $\widetilde{w}_2 = n_s$ and $\widetilde{w}_1 n_s < \widetilde{w}_1$, then we have $T^{\alggrp{M}}_{\widetilde{w}_1}T^{\alggrp{M}}_{n_s} = T^{\alggrp{M}}_{\widetilde{w}_1n_s^{-1}}(q_{\alggrp{M},s}T_{n_s^2} - T_{n_s}c_s) = q_{\alggrp{M},s}T^{\alggrp{M}}_{\widetilde{w}_1n_s} - T^{\alggrp{M}}_{\widetilde{w}_1}c_s$.
Therefore $j^+_{\alggrp{M}}(T^{\alggrp{M}}_{\widetilde{w}_1}T^{\alggrp{M}}_{\widetilde{w}_2}) = q_{\alggrp{M},s}T_{\widetilde{w}_1n_s} - T_{\widetilde{w}_1}c_s$.
We have $\ell(n_s^2) = \ell_{\alggrp{M}}(n_s^2) = 0$.
Hence $2\ell(n_s) = \ell(n_s) + \ell(n_s) - \ell(n_s^2) = \ell_{\alggrp{M}}(n_s) + \ell_{\alggrp{M}}(n_s) - \ell_{\alggrp{M}}(n_s^2) = 2$ by Lemma~\ref{lem:M-negativity and q}.
Therefore, $\ell(n_s) = 1$.
Namely, $s\in S_\aff$.
Therefore, by the definition of $q_{\alggrp{M},s}$, we have $q_{\alggrp{M},s} = q_s$.
We also have $\ell(\widetilde{w}_1) + \ell(n_s) - \ell(\widetilde{w}_1n_s) = \ell_{\alggrp{M}}(\widetilde{w}_1) + \ell_{\alggrp{M}}(n_s) - \ell_{\alggrp{M}}(\widetilde{w}_1n_s) = 2$.
Hence by the same calculation implies $T_{\widetilde{w}_1}T_{n_s} = q_{s}T_{\widetilde{w}_1n_s} - T_{\widetilde{w}_1}c_s$.
Therefore $j_{\alggrp{M}}^+(T^{\alggrp{M}}_{\widetilde{w}_1}T^{\alggrp{M}}_{\widetilde{w}_2}) = T_{\widetilde{w}_1}T_{\widetilde{w}_2}$.
Hence $j_{\alggrp{M}}^+$ is an algebra homomorphism.
By the same argument, $j_{\alggrp{M}}^-$ is an algebra homomorphism.

By the argument in the proof of \cite[Proposition~4.7]{MR2728487}, we have $j_{\alggrp{M}}^+(E^{\alggrp{M}}(\widetilde{w})) = E(\widetilde{w})$ for $\alggrp{M}$-positive $\widetilde{w}\in \widetilde{W}_{\alggrp{M}}(1)$ and $j_{\alggrp{M}}^-(E^{\alggrp{M}}(\lambda)) = E(\lambda)$ for $\alggrp{M}$-negative $\lambda\in\Lambda(1)$.
If $s\in S_{\alggrp{M}}$, then $s\in S$.
Hence $j_{\alggrp{M}}^-(T^{\alggrp{M}}_{n_s}) = j_{\alggrp{M}}^-(T^{\alggrp{M},*}_{n_s} + c_s) = T_{n_s}^* + c_s = T_{n_s}$.
Therefore $j_{\alggrp{M}}^-(T^{\alggrp{M}}_{n_w}) = T_{n_w}$ for $w\in W_{\alggrp{M}}$.
In other words, we have $j_{\alggrp{M}}^-(E^{\alggrp{M}}(n_w)) = E(n_w)$.
Since $j_{\alggrp{M}}^-$ is an algebra homomorphism, for $\alggrp{M}$-negative $\lambda\in\Lambda(1)$ and $w\in W_{\alggrp{M}}$, we have
\begin{align*}
j_{\alggrp{M}}^-(E^{\alggrp{M}}(\lambda n_w)) & = q_{\alggrp{M},\lambda}^{-1/2}q_{\alggrp{M},n_w}^{-1/2}q_{\alggrp{M},\lambda n_w}^{1/2}j_{\alggrp{M}}^-(E^{\alggrp{M}}(\lambda)E^{\alggrp{M}}(n_w))\\
& = q_\lambda^{-1/2}q_{n_w}^{-1/2}q_{\lambda n_w}^{1/2}E(\lambda)E(n_w)\\
& = E(\lambda n_w)
\end{align*}
by Lemma~\ref{lem:M-negativity and q}.
A similar argument implies $j_{\alggrp{M}}^-(E^{\alggrp{M}}_-(\widetilde{w})) = E_-(\widetilde{w})$ for $\alggrp{M}$-negative $\widetilde{w}\in \widetilde{W}_{\alggrp{M}}(1)$ and $j_{\alggrp{M}}^+(E^{\alggrp{M}}_-(\widetilde{w})) = E_-(\widetilde{w})$ for $\alggrp{M}$-positive $\widetilde{w}\in \widetilde{W}_{\alggrp{M}}(1)$.
\end{proof}

\begin{rem}\label{rem:role of strongly positive/negative element}
Let $\lambda_0^+\in \Lambda(1)$ be an $\alggrp{M}$-positive element such that it is in the center of $\widetilde{W}_{\alggrp{M}}(1)$ and $\langle \alpha,\nu(\lambda_0^+)\rangle < 0$ for any $\alpha\in\Sigma^+\setminus\Sigma_{\alggrp{M}}^+$.
Then $\mathcal{H}_{\alggrp{M}}$ is a localization $E^{\alggrp{M}}(\lambda_0^+)^{-1}\mathcal{H}_{\alggrp{M}}^+$ of $\mathcal{H}_{\alggrp{M}}^+$ at $E^{\alggrp{M}}(\lambda_0^+)$.

We give a proof.
First we prove that $E^{\alggrp{M}}(\lambda_0^+)$ is in the center of $\mathcal{H}_{\alggrp{M}}$.
Since $\lambda_0^+$ is in the center of $\widetilde{W}_{\alggrp{M}}(1)$, we have $n_{s_\alpha}\lambda_0^+ n_{s_\alpha}^{-1} = \lambda_0^+$ for any $\alpha\in\Phi_{\alggrp{M}}$.
Hence $s_\alpha(\nu(\lambda_0^+)) = \lambda_0^+$, namely $\langle \alpha,\nu(\lambda_0^+)\rangle = 0$.
Therefore by the length formula \eqref{eq:length formula}, we have the length of $\lambda_0^+$ as an element of $\widetilde{W}_{\alggrp{M}}(1)$ is zero.
Hence for $\widetilde{w}\in \widetilde{W}_{\alggrp{M}}(1)$, we have $E^{\alggrp{M}}(\lambda_0^+)E^{\alggrp{M}}(\widetilde{w}) = E^{\alggrp{M}}(\lambda_0^+\widetilde{w}) = E^{\alggrp{M}}(\widetilde{w}\lambda_0^+) = E^{\alggrp{M}}(\widetilde{w})E^{\alggrp{M}}(\lambda_0^+)$.
Therefore $E^{\alggrp{M}}(\lambda_0^+)$ is in the center of $\mathcal{H}_{\alggrp{M}}$ and since the length of $\lambda_0^+$ is zero, it is invertible.

By the above argument, we have $E^{\alggrp{M}}(\lambda_0^+)E^{\alggrp{M}}(\widetilde{w}) = E^{\alggrp{M}}(\lambda_0^+\widetilde{w})$ for $\widetilde{w}\in \widetilde{W}_{\alggrp{M}}(1)$.
Moreover, for any $\widetilde{w}\in\widetilde{W}_{\alggrp{M}}(1)$, there exists $n\in\Z_{\ge 0}$ such that $(\lambda_0^+)^n\widetilde{w}$ is $\alggrp{M}$-positive.
Therefore $E^{\alggrp{M}}(\widetilde{w}) = E^{\alggrp{M}}(\lambda_0^+)^{-n}E^{\alggrp{M}}((\lambda_0^+)^n\widetilde{w})$ is in the image of the homomorphism $E^{\alggrp{M}}(\lambda_0^+)^{-1}\mathcal{H}_{\alggrp{M}}^+ \to \mathcal{H}_{\alggrp{M}}$.
Hence we have $E^{\alggrp{M}}(\lambda_0^+)^{-1}\mathcal{H}_{\alggrp{M}}^+ \simeq \mathcal{H}_{\alggrp{M}}$.

In particular, if $\sigma$ is an $\mathcal{H}_{\alggrp{M}}$-module, it is determined by its restriction to $\mathcal{H}_{\alggrp{M}}^+$.
Since $E^{\alggrp{M}}(\lambda_0^+)$ is in the center of $\mathcal{H}_{\alggrp{M}}$, if $\sigma$ is irreducible, $E^{\alggrp{M}}(\lambda_0^+)$ acts by a scalar on $\sigma$.
Hence if $\sigma$ is irreducible then its restriction to $\mathcal{H}_{\alggrp{M}}^+$ is also irreducible.

This is also true for $\mathcal{H}_{\alggrp{M}}^-$.
Namely, for $\lambda_0^-$ such that it is in the center of $\widetilde{W}_{\alggrp{M}}(1)$ and $\langle\nu(\lambda_0^-),\alpha\rangle > 0$ for any $\alpha\in\Sigma^+\setminus\Sigma_{\alggrp{M}}^+$, we have $E^{\alggrp{M}}(\lambda_0^-)^{-1}\mathcal{H}_{\alggrp{M}}^-\simeq \mathcal{H}_{\alggrp{M}}$.
\end{rem}

\subsection{Parabolic induction}
\begin{defn}\label{defn:parabolic induction}
Let $\alggrp{P} = \alggrp{M}\alggrp{N}$ be a parabolic subgroup of $\alggrp{G}$.
For a representation $\sigma$ of $\mathcal{H}_{\alggrp{M}}\otimes C$, define
\[
	I_{\alggrp{P}}(\sigma) = \Hom_{\mathcal{H}_{\alggrp{M}}^-}(\mathcal{H},\sigma).
\]
Here, $\mathcal{H}_{\alggrp{M}}^-$ acts on $\mathcal{H}$ by the multiplication from the right through $j_{\alggrp{M}}^-$.
\end{defn}

Set $W^{\alggrp{M}} = \{w\in W\mid w(\Delta_{\alggrp{M}})\subset\Sigma^+\}$.
This is a complete representative of $W/W_{\alggrp{M}}$ and for $w_1\in W^{\alggrp{M}}$ and $w_2\in W_{\alggrp{M}}$, we have $\ell(w_1w_2) = \ell(w_1) + \ell(w_2)$.
\begin{lem}\label{lem:on length, M-part and not M-part}
Let $\lambda\in \Lambda(1)$, $w_1\in W^{\alggrp{M}}$ and $w_2\in W_{\alggrp{M}}$.
Assume that $\lambda$ is $\alggrp{M}$-negative.
Then $\ell(n_{w_1w_2}\lambda) = \ell(n_{w_1}) + \ell(n_{w_2}\lambda)$.
In particular, $T_{n_{w_1}}^*E_-(n_{w_2}\lambda) = E_-(n_{w_1w_2}\lambda)$.
\end{lem}
\begin{proof}
By the length formula \eqref{eq:length formula}, we have
\[
	\ell(n_{w_1w_2}\lambda) = \ell(\lambda) + \ell(n_{w_1w_2}) - 2\#\{\alpha\in\Sigma^+\mid (w_1w_2)(\alpha) < 0,\langle \nu(\lambda),\alpha\rangle < 0\}.
\]
We have $\ell(w_1w_2) = \ell(w_1) + \ell(w_2)$.
Hence for $\alpha\in\Sigma^+$, $(w_1w_2)(\alpha) < 0$ implies $w_2(\alpha) < 0$ or $\alpha = w_2^{-1}(\beta)$ for some $\beta\in\Sigma^+$ such that $w_1(\beta) < 0$.
Assume that $\alpha = w_2^{-1}(\beta)$ for some $\beta\in\Sigma^+$ such that $w_1(\beta) < 0$.
Since $w_1(\beta) < 0$, $\beta\in\Sigma^+\setminus\Sigma^+_{\alggrp{M}}$ by the definition of $W^{\alggrp{M}}$.
Hence $\alpha = w_2^{-1}(\beta) \in \Sigma^+\setminus\Sigma^+_{\alggrp{M}}$.
Since $\lambda$ is $\alggrp{M}$-negative, $\langle\nu(\lambda),\alpha\rangle\ge 0$.
Therefore,
\[
	\{\alpha\in\Sigma^+\mid (w_1w_2)(\alpha) < 0,\langle \nu(\lambda),\alpha\rangle < 0\}
	=
	\{\alpha\in\Sigma^+\mid w_2(\alpha) < 0,\langle \nu(\lambda),\alpha\rangle < 0\}.
\]
Hence
\begin{align*}
\ell(n_{w_1w_2}\lambda) & = \ell(\lambda) + \ell(n_{w_1w_2}) - 2\#\{\alpha\in\Sigma^+\mid w_2(\alpha) < 0,\langle \nu(\lambda),\alpha\rangle < 0\}\\
& = \ell(\lambda) + \ell(n_{w_2}) + \ell(n_{w_1}) - 2\#\{\alpha\in\Sigma^+\mid w_2(\alpha) < 0,\langle \nu(\lambda),\alpha\rangle < 0\}\\
& = \ell(n_{w_2}\lambda) + \ell(n_{w_1}).\qedhere
\end{align*}
\end{proof}

\begin{lem}[{See \cite[Proposition~5.2]{MR2728487}}]\label{lem:parabolic induction as vector space}
The map $\varphi\mapsto (\varphi(T_{n_w}))_{w\in W^{\alggrp{M}}}$ gives an isomorphism $I_{\alggrp{P}}(\sigma)\simeq\bigoplus_{w\in W^{\alggrp{M}}}\sigma$ as vector spaces.
\end{lem}
\begin{proof}
We may replace $\varphi\mapsto \varphi(T_{n_w})$ with $\varphi\mapsto \varphi(T_{n_w}^*)$.
Take $\lambda_0^-\in \Lambda(1)$ as in Remark~\ref{rem:role of strongly positive/negative element}.
Then $\varphi$ is extended to $\mathcal{H}E(\lambda_0^-)^{-1}\to \sigma$.
Fix $v\in W^{\alggrp{M}}$.
Then by the Bernstein relation and the definition of $E_-(\widetilde{w})$, $H_v = \bigoplus_{w\in W_{\alggrp{M}},\lambda\in \Lambda(1)}CE_-(n_{vw}\lambda)\subset \mathcal{H}\otimes C$ is right $j_{\alggrp{M}}^-(\mathcal{H}_{\alggrp{M}}^-)$-stable.
We prove $\Phi_0\colon \mathcal{H}_{\alggrp{M}}^-\ni F\mapsto T_{n_v}^*j_{\alggrp{M}}^-(F)\in H_v$ induces an isomorphism $\Phi\colon\mathcal{H}_{\alggrp{M}}\otimes C\simeq H_vE(\lambda_0^-)^{-1}$.

Let $w\in W_{\alggrp{M}}$ and $\lambda\in \Lambda(1)$ such that $\lambda$ is $\alggrp{M}$-negative.
Then by Lemma~\ref{lem:on length, M-part and not M-part}, we have $T_{n_v}^*E_-(n_w\lambda) = E_-(n_{vw}\lambda)$.
Hence $\Phi_0$ is injective.
Therefore $\Phi$ is injective.

We prove surjectivity.
Let $w\in W_{\alggrp{M}}$ and $\lambda \in \Lambda(1)$.
Take $n\in\Z_{\ge 0}$ such that $\lambda (\lambda_0^-)^n$ is  $\alggrp{M}$-negative.
If $\ell(n_{vw} \lambda (\lambda_0^-)^n) = \ell(n_{vw} \lambda) + \ell((\lambda_0^-)^n)$, then $E_-(n_{vw} \lambda) = E_-(n_{vw} \lambda(\lambda_0^-)^n)E(\lambda_0^-)^{-n} = T_{n_v}^*E_-(n_{w} \lambda(\lambda_0^-)^n)E(\lambda_0^-)^{-n}\in \Imm\Phi$ by Lemma~\ref{lem:on length, M-part and not M-part}.
If $\ell(n_{vw} \lambda (\lambda_0^-)^n) > \ell(n_{vw} \lambda) + \ell((\lambda_0^-)^n)$, then $E_-(n_{vw}\lambda)E(\lambda_0^-)^n = 0$.
Hence, in $H_vE(\lambda_0^-)^{-1}$, $E(n_{vw} \lambda) = 0$.

Therefore $(\mathcal{H}\otimes C)E(\lambda_0^-)^{-1} = \bigoplus_{v\in W^{\alggrp{M}}}H_vE(\lambda_0^-)^{-1}\simeq \bigoplus_{v\in W^{\alggrp{M}}}(\mathcal{H}_{\alggrp{M}}\otimes C)$.
Since $\mathcal{H}_{\alggrp{M}} = \mathcal{H}_{\alggrp{M}}^-E^{\alggrp{M}}(\lambda_0^-)^{-1}$ by Remark~\ref{rem:role of strongly positive/negative element}, we have $I_P(\sigma) = \Hom_{\mathcal{H}_{\alggrp{M}}^-}(\mathcal{H},\sigma) = \Hom_{\mathcal{H}_{\alggrp{M}}}((\mathcal{H}\otimes C)E(\lambda_0^-)^{-1},\sigma)\simeq \bigoplus_{v\in W^{\alggrp{M}}}\Hom_{\mathcal{H}_{\alggrp{M}}}(\mathcal{H}_{\alggrp{M}}\otimes C,\sigma)\simeq \bigoplus_{v\in W^{\alggrp{M}}}\sigma$.
\end{proof}
For $w\in W^{\alggrp{M}}$, we denote the subspace $\{\varphi\in I_{\alggrp{P}}(\sigma)\mid \varphi(T_{n_v}) = 0\ (v\in W^{\alggrp{M}}\setminus\{w\})\}$ by $w\sigma$.
We have $I_{\alggrp{P}}(\sigma) = \bigoplus_{w\in W^{\alggrp{M}}}w\sigma$.
By the lemma below, $w\sigma$ is $\mathcal{A}$-stable.

\begin{rem}\label{rem:the condition to vanish at E_-(n_wlambda)}
Let $v\in W^{\alggrp{M}}$, $\varphi\in v\sigma\subset I_{\alggrp{P}}(\sigma)$ and $\lambda\in \Lambda(1)$.
From the proof of the above lemma, if $w\in \{v_0\in W^{\alggrp{M}}\mid v_0 < v\}W_{\alggrp{M}}$ then $\varphi(E_-(n_w\lambda)) = 0$.
We remark that if $w$ is in this subset and $w_1\le w$, then $w_1$ is also in this subset. (It follows from \cite[Lemma~4.20]{MR3143708}. See also the proof of Lemma~\ref{lem:w_Delta w_Delta_M is maximal}.)
\end{rem}

We can describe the $\mathcal{A}$-module structure of $I_{\alggrp{P}}(\sigma)$ using the above decomposition.
\begin{prop}
Let $\sigma$ be an $\mathcal{H}_{\alggrp{M}}\otimes C$-module.
Define an $\mathcal{A}$-module $\sigma_\mathcal{A}$ by
\begin{itemize}
\item $\sigma = \sigma_\mathcal{A}$ as vector spaces.
\item for $\alggrp{M}$-negative $\lambda\in\Lambda(1)$ we have $\sigma_\mathcal{A}(E(\lambda)) = \sigma(E^{\alggrp{M}}(\lambda))$.
\item for $\lambda\in \Lambda(1)$ which is not $\alggrp{M}$-negative, $\sigma_\mathcal{A}(E(\lambda)) = 0$.
\end{itemize}
Then $\varphi\mapsto (\varphi(T_{n_w}))_{w\in W^{\alggrp{M}}}$ gives an isomorphism $I_P(\sigma)|_{\mathcal{A}}\simeq \bigoplus_{w\in W^{\alggrp{M}}}w\sigma_\mathcal{A}$.
\end{prop}
\begin{proof}
We prove $(\varphi E(\lambda))(T_{n_w}) = \varphi(T_{n_w})\sigma_\mathcal{A}(E(n^{-1}_{w}(\lambda)))$ by induction on $\ell(w)$.

If $w = 1$, then $(\varphi E(\lambda))(1) = \varphi(E(\lambda))$.
If $\lambda$ is $\alggrp{M}$-negative, then $\varphi(E(\lambda)) = \varphi(1)E^{\alggrp{M}}(\lambda) = \varphi(1)\sigma_\mathcal{A}(E(\lambda))$.
Assume that $\lambda$ is not $\alggrp{M}$-negative.
Take $\lambda_0^-\in \Lambda(1)$ as in Remark~\ref{rem:role of strongly positive/negative element}.
Since $\lambda$ is not $\alggrp{M}$-negative, $\nu(\lambda)$ and $\nu(\lambda_0^-)$ do not belong to the same closed chamber.
Hence $E(\lambda)E(\lambda_0^-) = 0$ by \eqref{eq:multiplication in A}.
Therefore $\varphi(E(\lambda)) = \varphi(E(\lambda)E(\lambda_0^-))E^{\alggrp{M}}((\lambda_0^-)^{-1}) = 0 = \varphi(1)\sigma_\mathcal{A}(E(\lambda))$.

Assume that $w\ne 1$ and take $\alpha\in\Delta$ such that $s = s_\alpha$ satisfies $sw < w$.
Since $\{\beta\in\Sigma^+\mid sw(\beta) < 0\}\subset \{\beta\in \Sigma^+\mid w(\beta) < 0\}$, $sw\in W^{\alggrp{M}}$.
Since $w(-w^{-1}(\alpha)) = -\alpha < 0$ and $w\in W^{\alggrp{M}}$, we have $-w^{-1}(\alpha)\in \Sigma^+\setminus\Sigma^+_{\alggrp{M}}$.
Assume that $\langle\nu(\lambda),\alpha\rangle > 0$.
Then by Lemma~\ref{lem:simple Bernstein relations}, $E(\lambda)T_{n_s} = (T_{n_s} - c_s)E(n_s^{-1}(\lambda))$.
Hence
\begin{align*}
(\varphi E(\lambda))(T_{n_w}) & = \varphi(E(\lambda)T_{n_w})\\
& = \varphi(E(\lambda)T_{n_s}T_{n_{sw}})\\
& = \varphi((T_{n_s} - c_s)E(n_s^{-1}(\lambda))T_{n_{sw}})\\
& = (\varphi(T_{n_s} - c_s)E(n_s^{-1}(\lambda)))(T_{n_{sw}}).
\end{align*}
Applying the inductive hypothesis to $\varphi(T_{n_s} - c_s)$, $(\varphi(T_{n_s} - c_s)E(n_s^{-1}(\lambda)))(T_{n_{sw}}) = (\varphi(T_{n_s} - c_s))(T_{n_{sw}})\sigma_\mathcal{A}(E(n_w^{-1}(\lambda)))$.
Since $\langle\nu(n_w^{-1}(\lambda)),-w^{-1}(\alpha)\rangle = -\langle\nu(\lambda),\alpha\rangle < 0$, $n_w^{-1}(\lambda)$ is not $\alggrp{M}$-negative.
Hence $\sigma_\mathcal{A}(E(n_w^{-1}(\lambda))) = 0$.
Therefore $(\varphi(T_{n_s} - c_s))(T_{n_{sw}})\sigma_\mathcal{A}(E(n_w^{-1}(\lambda))) = 0 = \varphi (T_{n_{w}})\sigma_\mathcal{A}(E(n_w^{-1}(\lambda)))$.

If $\langle \nu(\lambda),\alpha\rangle = 0$, then $E(\lambda)T_{n_s} = T_{n_s}E(n_s^{-1}(\lambda))$ \cite[Lemma~5.34, 5.35]{Vigneras-prop}.
Hence
\begin{align*}
(\varphi E(\lambda))(T_{n_w}) & = \varphi(E(\lambda)T_{n_s}T_{n_{sw}})\\
& = \varphi(T_{n_s}E(n_s^{-1}(\lambda))T_{n_{sw}})\\
& = \varphi (T_{n_w})\sigma_\mathcal{A}(E(n_w^{-1}(\lambda))).
\end{align*}

Finally, assume that $\langle\nu(\lambda),\alpha\rangle < 0$.
Then $\langle \nu(n_{sw}^{-1}(\lambda)),-w^{-1}(\alpha)\rangle = \langle\nu(\lambda),\alpha\rangle < 0$.
Hence $n_{sw}^{-1}(\lambda)$ is not $\alggrp{M}$-negative.
Therefore, we have $\varphi(E(\lambda)c_sT_{n_{sw}}) = (\varphi(\lambda\cdot c_s)E(\lambda))(T_{n_{sw}}) = (\varphi(\lambda\cdot c_s))(T_{n_{sw}})\sigma_\mathcal{A}(E(n_{sw}^{-1}(\lambda))) = 0$.
By Lemma~\ref{lem:simple Bernstein relations}, $E(\lambda)(T_{n_s} - c_s) = T_{n_s}E(n_s^{-1}(\lambda))$.
Hence
\begin{align*}
(\varphi E(\lambda))(T_{n_w}) &= \varphi(E(\lambda)(T_{n_s} - c_s)T_{n_{sw}}) \\&= (\varphi T_{n_s})(E(n_s^{-1}(\lambda))T_{n_{sw}}) \\&= \varphi(T_{n_w})\sigma_\mathcal{A}(E(n_w^{-1}(\lambda))).\qedhere
\end{align*}
\end{proof}

\subsection{Another description of parabolic induction}
We give a realization of parabolic induction via a tensor product.
First we find a submodule of a pro-$p$-Iwahori Hecke algebra of a Levi subgroup (which is not $\alggrp{M}$ in general) in $I_{\alggrp{P}}(\sigma)$.

Set $\alggrp{P}' = n_{w_{\Delta}w_{\Delta_{\alggrp{M}}}}\alggrp{P}n_{w_{\Delta}w_{\Delta_{\alggrp{M}}}}^{-1}$.
This is a standard parabolic subgroup of $\alggrp{G}$ corresponding to $w_\Delta w_{\Delta_{\alggrp{M}}}(\Delta_{\alggrp{M}}) = -w_\Delta(\Delta_{\alggrp{M}})$.
Let $\alggrp{M}'$ be the Levi part of $\alggrp{P}$.
Then the map $m\mapsto n_{w_{\Delta}w_{\Delta_{\alggrp{M}}}} m n_{w_{\Delta}w_{\Delta_{\alggrp{M}}}}^{-1}$ is an isomorphism between $\alggrp{M}$ and $\alggrp{M}'$.
Moreover, this map preserves the maximal split torus, the pro-$p$-Iwahori subgroup, the minimal parabolic subgroup, the root system, the set of simple roots etc...
Since the pro-$p$-Iwahori Hecke algebra is defined by these data, this map induces an isomorphism $\mathcal{H}_{\alggrp{M}}\to \mathcal{H}_{\alggrp{M}'}$.
Explicitly it is given by $T^{\alggrp{M}}_{\widetilde{w}}\mapsto T^{\alggrp{M}'}_{n_{w_{\Delta}w_{\Delta_{\alggrp{M}}}}\widetilde{w}n_{w_{\Delta}w_{\Delta_{\alggrp{M}}}}^{-1}}$.
We also have $E_-^{\alggrp{M}}(\widetilde{w})\mapsto E_-^{\alggrp{M}'}(n_{w_{\Delta}w_{\Delta_{\alggrp{M}}}}\widetilde{w}n_{w_{\Delta}w_{\Delta_{\alggrp{M}}}}^{-1})$.
Notice that $w_\Delta w_{\Delta_{\alggrp{M}}}\in W^{\alggrp{M}}$.

\begin{lem}\label{lem:w_Delta w_Delta_M is maximal}
Let $v\in W$.
\begin{enumerate}
\item We have $W\setminus w_\Delta w_{\Delta_{\alggrp{M}}}W_{\alggrp{M}} = \{v\in W^{\alggrp{M}}\mid v < w_\Delta w_{\Delta_{\alggrp{M}}}\}W_{\alggrp{M}}$.
\item If $v < w_\Delta w_{\Delta_{\alggrp{M}}}$, then $v\notin w_\Delta w_{\Delta_{\alggrp{M}}}W_{\alggrp{M}}$.
\item If $v\notin w_\Delta w_{\Delta_{\alggrp{M}}}W_{\alggrp{M}}$ and $v'\le v$, then $v'\notin w_\Delta w_{\Delta_{\alggrp{M}}}W_{\alggrp{M}}$.
\end{enumerate}
\end{lem}
\begin{proof}
Let $w\in W$ and take $w_1\in W^{\alggrp{M}}$ and $w_2\in W_{\alggrp{M}}$ such that $w = w_1w_2$.
If $w_1 < w_\Delta w_{\Delta_{\alggrp{M}}}$, then, since $W/W_{\alggrp{M}}\simeq W^{\alggrp{M}}$, $w_1W_{\alggrp{M}}\cap w_\Delta w_{\Delta_{\alggrp{M}}}W_{\alggrp{M}} = \emptyset$.
Hence $w\notin w_{\Delta}w_{\Delta_{\alggrp{M}}}W_{\alggrp{M}}$.
Assume that $w\notin w_{\Delta}w_{\Delta_{\alggrp{M}}}W_{\alggrp{M}}$.
Since $w_1\le w_\Delta = w_{\Delta}w_{\Delta_{\alggrp{M}}}w_{\Delta_{\alggrp{M}}}$ and $w_{\Delta}w_{\Delta_{\alggrp{M}}}\in W^{\alggrp{M}}$, we have $w_1\le w_{\Delta}w_{\Delta_{\alggrp{M}}}$ by \cite[Lemma~4.20]{MR3143708}.
If $w_1 = w_{\Delta}w_{\Delta_{\alggrp{M}}}$, then $w\in w_{\Delta}w_{\Delta_{\alggrp{M}}}W_{\alggrp{M}}$.
Hence $w_1 < w_{\Delta}w_{\Delta_{\alggrp{M}}}$.
We get (1).

Assume that $v\in w_{\Delta}w_{\Delta_{\alggrp{M}}}W_{\alggrp{M}}$ and take $v'\in W_{\alggrp{M}}$ such that $v = w_\Delta w_{\Delta_{\alggrp{M}}}v'$.
Since $w_\Delta w_{\Delta_{\alggrp{M}}}\in W^{\alggrp{M}}$ and $v'\in W_{\alggrp{M}}$, we have $\ell(w_\Delta w_{\Delta_{\alggrp{M}}}v') = \ell(w_\Delta w_{\Delta_{\alggrp{M}}}) + \ell(v')$.
Hence $w_\Delta w_{\Delta_{\alggrp{M}}}\le w_\Delta w_{\Delta_{\alggrp{M}}}v' = v$.
We get (2).

We prove (3).
Take $v_1\in W^{\alggrp{M}}$ and $v_2\in W_{\alggrp{M}}$ such that $v = v_1v_2$.
By (1), $v_1 < w_\Delta w_{\Delta_{\alggrp{M}}}$.
Take $v_1'\in W^{\alggrp{M}}$ and $v_2'\in W_{\alggrp{M}}$ such that $v' = v_1'v_2'$.
Then $v_1'v_2'\le v_1v_2$.
Since $v_1'\le v_1'v_2'$, we have $v_1'\le v_1v_2$.
By \cite[Lemma~4.20]{MR3143708}, we have $v_1'\le v_1$.
Since $v_1 < w_\Delta w_{\Delta_{\alggrp{M}}}$, we have $v_1' < w_\Delta w_{\Delta_{\alggrp{M}}}$.
Hence we get (3) by (2).
\end{proof}

\begin{prop}\label{prop:M'-mod in parabolic induction}
Let $\sigma$ be an $\mathcal{H}_{\alggrp{M}}$-module and define an $\mathcal{H}_{\alggrp{M}'}$-module $\sigma'$ by pulling back $\sigma$ by the above isomorphism. (Namely it is given by $\sigma'(E^{\alggrp{M}'}_-(\widetilde{w})) = \sigma(E^{\alggrp{M}}_-(n_{w_\Delta w_{\Delta_{\alggrp{M}}}}^{-1}\widetilde{w}n_{w_\Delta w_{\Delta_{\alggrp{M}}}}))$ for $\widetilde{w}\in \widetilde{W}_{\alggrp{M}'}(1)$.
Then $w_\Delta w_{\Delta_{\alggrp{M}}}\sigma\subset I(\alggrp{P},\sigma,\alggrp{Q})$ is $\mathcal{H}_{\alggrp{M}'}^+$-stable and isomorphic to $\sigma'$.
In particular, if $\sigma$ is irreducible and an $\mathcal{H}$-submodule $\pi$ of $I_{\alggrp{P}}(\sigma)$ has a non-zero intersection with $w_\Delta w_{\Delta_{\alggrp{M}}}\sigma$, then $\pi$ contains $w_\Delta w_{\Delta_{\alggrp{M}}}\sigma$.
\end{prop}
\begin{proof}
Let $w\in W_{\alggrp{M}}$, $\lambda\in \Lambda(1)$.
Put $w_1 = w_{\Delta}w_{\Delta_{\alggrp{M}}}w(w_{\Delta}w_{\Delta_{\alggrp{M}}})^{-1}$ and $\lambda_1 = n_{w_{\Delta}w_{\Delta_{\alggrp{M}}}}(\lambda)$.
We prove $\varphi(E_-(n_{w_1}\lambda_1)T_{n_{w_\Delta w_{\Delta_{\alggrp{M}}}}}) = \varphi(T_{n_{w_\Delta w_{\Delta_{\alggrp{M}}}}}E_-(n_w \lambda))$ for $\varphi\in w_\Delta w_{\Delta_{\alggrp{M}}}\sigma$

Since $w\in W_{\alggrp{M}}$ and $w_\Delta w_{\Delta_{\alggrp{M}}}\in W^{\alggrp{M}}$, we have $\ell(w_\Delta w_{\Delta_{\alggrp{M}}}) + \ell(w) = \ell(w_\Delta w_{\Delta_{\alggrp{M}}} w)$.
Hence $n_{w_\Delta w_{\Delta_{\alggrp{M}}}}n_w = n_{w_\Delta w_{\Delta_{\alggrp{M}}}w}$.
We have $w_\Delta w_{\Delta_{\alggrp{M}}}w = w_1w_\Delta w_{\Delta_{\alggrp{M}}}$ and by the same argument implies $n_{w_1w_\Delta w_{\Delta_{\alggrp{M}}}} = n_{w_1}n_{w_\Delta w_{\Delta_{\alggrp{M}}}}$.
Hence $n_{w_\Delta w_{\Delta_{\alggrp{M}}}}n_w = n_{w_1}n_{w_\Delta w_{\Delta_{\alggrp{M}}}}$.
Therefore $n_{w_1}^{-1}n_{w_\Delta w_{\Delta_{\alggrp{M}}}} = n_{w_\Delta w_{\Delta_{\alggrp{M}}}}n_w^{-1}$.
Since $w^{-1}\in W_{\alggrp{M}}$ and $w_{\Delta}w_{\Delta_{\alggrp{M}}}\in W^{\alggrp{M}}$, we also have $\ell(n_{w_\Delta w_{\Delta_{\alggrp{M}}}}n_w^{-1}) = \ell(n_{w_\Delta w_{\Delta_{\alggrp{M}}}}) + \ell(n_w^{-1})$.
Hence $T_{n_{w_\Delta w_{\Delta_{\alggrp{M}}}}n_w^{-1}} = T_{n_{w_\Delta w_{\Delta_{\alggrp{M}}}}}T_{n_w^{-1}}$.
The same argument implies $T_{n_{w_1}^{-1}n_{w_\Delta w_{\Delta_{\alggrp{M}}}}} = T_{n_{w_1}^{-1}}T_{n_{w_\Delta w_{\Delta_{\alggrp{M}}}}}$.
Hence in $\mathcal{H}$, we have $T_{n_{w_\Delta w_{\Delta_{\alggrp{M}}}}}T_{n_w^{-1}} = T_{n_{w_1}^{-1}}T_{n_{w_\Delta w_{\Delta_{\alggrp{M}}}}}$.
Therefore we get $T_{n_{w_1}^{-1}}^{-1}T_{n_{w_\Delta w_{\Delta_{\alggrp{M}}}}} = T_{n_{w_\Delta w_{\Delta_{\alggrp{M}}}}}T_{n_w^{-1}}^{-1}$.
The definition of $T_{n_w}^*$ is $T_{n_w}^* = q_{n_w}T_{n_w^{-1}}^{-1}$.
Hence $T_{n_{w_1}}^*T_{n_{w_\Delta w_{\Delta_{\alggrp{M}}}}} = T_{n_{w_\Delta w_{\Delta_{\alggrp{M}}}}}T_{n_w}^*$.
Therefore,
\begin{align*}
&T_{n_{w_\Delta w_{\Delta_{\alggrp{M}}}}}E_-(n_w\lambda) - E_-(n_{w_1}\lambda_1)T_{n_{w_\Delta w_{\Delta_{\alggrp{M}}}}}\\
&\in C[q_s^{\pm 1/2}](T_{n_{w_\Delta w_{\Delta_{\alggrp{M}}}}}T_{n_w}^* \theta(\lambda) - T_{n_{w_1}}^*\theta(\lambda_1)T_{n_{w_\Delta w_{\Delta_{\alggrp{M}}}}})\\
& = C[q_s^{\pm 1/2}]T_{n_{w_1}}^* (T_{n_{w_\Delta w_{\Delta_{\alggrp{M}}}}}\theta(\lambda) - \theta(\lambda_1)T_{n_{w_\Delta w_{\Delta_{\alggrp{M}}}}})
\end{align*}
By the Bernstein relation, we have
\[
	T_{n_{w_\Delta w_{\Delta_{\alggrp{M}}}}}\theta(\lambda) - \theta(\lambda_1)T_{n_{w_\Delta w_{\Delta_{\alggrp{M}}}}}
	\in \sum_{v < w_\Delta w_{\Delta_{\alggrp{M}}},\mu\in \Lambda(1)}C[q_s^{\pm 1/2}]T_{n_v}\theta(\mu).
\]
We have
\[
	\sum_{v < w_\Delta w_{\Delta_{\alggrp{M}}}}C[q_s^{\pm 1/2}]T_{n_v}
	\subset
	\sum_{v < w_\Delta w_{\Delta_{\alggrp{M}}}}C[q_s^{\pm 1/2}][Z_\kappa]T_{n_v}^*
\]
and $T_{n_{w_1}}^*T_{n_v}^* \in \sum_{w_1'\le w_1,v'\le v,w' \le w_1'v'}C[q_s^{\pm 1/2}][Z_\kappa]T_{n_{w'}}^*$.

Let $w',w_1',v',v$ as above, namely $v<w_\Delta w_{\Delta_{\alggrp{M}}},w_1'\le w_1,v'\le v$ and $w' \le w_1'v'$.
Then $v' < w_\Delta w_{\Delta_{\alggrp{M}}}$.
Hence $v' \in W\setminus w_\Delta w_{\Delta_{\alggrp{M}}}W_{\alggrp{M}} = W\setminus W_{\alggrp{M}'}w_\Delta w_{\Delta_{\alggrp{M}}}$ by the above lemma (2).
Since $w_1\in W_{\alggrp{M}'}$, we have $w'_1\in W_{\alggrp{M}'}$.
Hence
\[
	w_1'v' \in W\setminus W_{\alggrp{M}'}w_\Delta w_{\Delta_{\alggrp{M}}} = W\setminus w_\Delta w_{\Delta_{\alggrp{M}}}W_{\alggrp{M}}.
\]
Therefore, if $w' \le w'_1v'$, then $w'\in W\setminus w_\Delta w_{\Delta_{\alggrp{M}}}W_{\alggrp{M}} = \{v\in W^{\alggrp{M}}\mid v < w_\Delta w_{\Delta_{\alggrp{M}}}\}W_{\alggrp{M}}$ by the above lemma (1) and (3).
Hence we get
\begin{align*}
&T_{n_{w_\Delta w_{\Delta_{\alggrp{M}}}}}E_-(n_w \lambda) - E_-(n_{w_1}\lambda_1)T_{n_{w_\Delta w_{\Delta_{\alggrp{M}}}}}\\
&\in \left(\sum_{a\in W^{\alggrp{M}},a < w_{\Delta} w_{\Delta_{\alggrp{M}}},b\in W_{\alggrp{M}},\mu\in \Lambda(1)}C[q_s^{\pm 1/2}]T^*_{n_{ab}}\theta(\mu)\right)\cap \mathcal{H}\\
& = \left(\sum_{a\in W^{\alggrp{M}},a < w_{\Delta} w_{\Delta_{\alggrp{M}}},b\in W_{\alggrp{M}},\mu\in \Lambda(1)}C[q_s^{\pm 1/2}]E_-(n_{ab}\mu)\right)\cap \mathcal{H}\\
& = \sum_{a\in W^{\alggrp{M}},a < w_{\Delta} w_{\Delta_{\alggrp{M}}},b\in W_{\alggrp{M}},\mu\in \Lambda(1)}C[q_s^{1/2}]E_-(n_{ab}\mu)
\end{align*}
By Remark~\ref{rem:the condition to vanish at E_-(n_wlambda)}, $\varphi$ is zero on this space.
\end{proof}

\begin{prop}\label{prop:tensor induction and parabolic induction}
Keep the notation in Proposition~\ref{prop:M'-mod in parabolic induction}.
Then we have $\sigma'\otimes_{\mathcal{H}_{\alggrp{M}'}^+}\mathcal{H}\simeq I_{\alggrp{P}}(\sigma)$.
\end{prop}
\begin{proof}
We get $\Phi\colon\sigma'\otimes_{\mathcal{H}_{\alggrp{M}'}^+}\mathcal{H}\to I_{\alggrp{P}}(\sigma)$ by Proposition~\ref{prop:M'-mod in parabolic induction}.
Put ${}^{\alggrp{M}'}\!W = \{w^{-1}\mid w\in W^{\alggrp{M}'}\}$.
Then for $w_1\in W_{\alggrp{M}'}$ and $w_2\in {}^{\alggrp{M}'}\!W$, we have $\ell(w_1w_2) = \ell(w_1) + \ell(w_2)$.
By a similar argument of the proof of Lemma~\ref{lem:parabolic induction as vector space}, the homomorphism $(x_w)_w\mapsto \sum_w x_w\otimes T_{n_w}^*$ gives an isomorphism $\bigoplus_{w\in {}^{\alggrp{M}'}\!W}\sigma'\simeq \sigma'\otimes_{\mathcal{H}_{\alggrp{M}'}^+}\mathcal{H}$.
By the construction, $\Phi$ induces an isomorphism $\sigma'\simeq \sigma'\otimes 1\to w_\Delta w_{\Delta_{\alggrp{M}}}\sigma\subset I_{\alggrp{P}}(\sigma)$.

First we prove that $v\mapsto w_\Delta w_{\Delta_{\alggrp{M}}} v^{-1}$ gives a bijection $W^{\alggrp{M}}\to {}^{\alggrp{M}'}\!W$.
We have $w_{\Delta_{\alggrp{M}}}w_\Delta(\Delta_{\alggrp{M}'}) = w_{\Delta_{\alggrp{M}}}(-\Delta_{\alggrp{M}}) = \Delta_{\alggrp{M}}$.
Hence $(w_\Delta w_{\Delta_{\alggrp{M}}} v^{-1})^{-1}(\Delta_{\alggrp{M}'})\subset v(\Delta_{\alggrp{M}})\subset\Sigma^+$.
Therefore the map is well-defined.
The inverse map is given by $w\mapsto w^{-1}w_\Delta w_{\Delta_{\alggrp{M}}}$.

We prove that the homomorphism $\Phi$ induces an isomorphism $ \sigma'\otimes T^*_{n_{w_\Delta w_{\Delta_{\alggrp{M}}}v^{-1}}}\simeq v\sigma$ for $v\in W^{\alggrp{M}}$.
To do it, it is sufficient to prove the following:
Let $w\in {}^{\alggrp{M}'}\!W$ and $v\in W^{\alggrp{M}}$.
Take $x\in \sigma'$ and put $\varphi = \Phi(x\otimes T_{n_w}^*)$.
Then we have $\varphi(T_{n_v}) = x$ if $wv = w_\Delta w_{\Delta_{\alggrp{M}}}$ and if not, $\varphi(T_{n_v}) = 0$.

Put $\varphi' = \Phi(x\otimes 1)$.
Then $\varphi(T_{n_v}) = \varphi'(T_{n_w}^*T_{n_v})$.
We have
\[
	T_{n_w}^*T_{n_v} = E_{w(-\Delta)}(n_w)E_{-\Delta}(n_v) = (q_{n_wn_v}^{-1}q_{n_w}q_{n_v})^{1/2}E_{w(-\Delta)}(n_wn_v).
\]
Hence it is zero if $\ell(w) + \ell(v) > \ell(wv)$.
Assume that $\ell(w) + \ell(v) = \ell(wv)$.
Then $T_{n_w}^*T_{n_v} = E_{w(-\Delta)}(n_wn_v) = E_{w(-\Delta)}(n_{wv}) \in T_{n_{wv}} + \sum_{v' < wv}C[Z_\kappa]T_{n_{v'}}$.

Assume that $wv\notin w_\Delta w_{\Delta_{\alggrp{M}}}W_{\alggrp{M}}$.
Then $v' < wv$ implies $v'\in W\setminus w_\Delta w_{\Delta_{\alggrp{M}}}W_{\alggrp{M}} = \{v_0\in W^{\alggrp{M}}\mid v_0 < w_\Delta w_{\Delta_{\alggrp{M}}}\}W_{\alggrp{M}}$ by Lemma~\ref{lem:w_Delta w_Delta_M is maximal} (1) and (3).
Hence by Remark~\ref{rem:the condition to vanish at E_-(n_wlambda)}, we have $\varphi'(C[Z_\kappa]T_{n_{v'}}) = 0$.
We also have $\varphi'(T_{n_{wv}}) = 0$.
Hence $\varphi'(T_{n_w}^*T_{n_v}) = 0$.

Assume that $wv\in w_\Delta w_{\Delta_{\alggrp{M}}}W_{\alggrp{M}}$.
Take $y\in W_{\alggrp{M}}$ such that $wv = w_\Delta w_{\Delta_{\alggrp{M}}} y$.
Since $w_\Delta w_{\Delta_{\alggrp{M}}}w_\Delta\in W_{\alggrp{M}'}$ and $w\in {}^{\alggrp{M}'}\!W$, we have
\begin{equation}\label{eq:length, to compare?}
\ell(w_\Delta w_{\Delta_{\alggrp{M}}}w_\Delta w) = \ell(w_\Delta w_{\Delta_{\alggrp{M}}}w_\Delta) + \ell(w).
\end{equation}
Since $v\in W^{\alggrp{M}}$ and $y\in W_{\alggrp{M}}$, we have $\ell(yv^{-1}) = \ell(y) + \ell(v)$.
Hence by $wv = w_\Delta w_{\Delta_{\alggrp{M}}} y$, the left hand side of \eqref{eq:length, to compare?} is
\[
	\ell(w_\Delta w_{\Delta_{\alggrp{M}}}w_\Delta w) = \ell(w_\Delta yv^{-1}) = \ell(w_\Delta) - \ell(yv^{-1}) = \ell(w_\Delta) - \ell(y) - \ell(v).
\]
Recall that we assumed $\ell(wv) = \ell(w) + \ell(v)$.
Similarly, the right hand side of \eqref{eq:length, to compare?} is
\begin{align*}
\ell(w_\Delta w_{\Delta_{\alggrp{M}}}w_\Delta) + \ell(w)
& = \ell(w_{\Delta_{\alggrp{M}}}) + \ell(w)\\
& = \ell(w_{\Delta_{\alggrp{M}}}) + \ell(wv) - \ell(v)\\
& = \ell(w_{\Delta_{\alggrp{M}}}) + \ell(w_\Delta w_{\Delta_{\alggrp{M}}}y) - \ell(v)\\
& = \ell(w_{\Delta_{\alggrp{M}}}) + \ell(w_\Delta) - \ell(w_{\Delta_{\alggrp{M}}}y) - \ell(v)\\
& = \ell(w_{\Delta_{\alggrp{M}}}) + \ell(w_\Delta) - \ell(w_{\Delta_{\alggrp{M}}}) + \ell(y) - \ell(v)\\
& = \ell(w_\Delta) + \ell(y) - \ell(v)
\end{align*}
We get $\ell(y) = 0$.
Hence $wv = w_\Delta w_{\Delta_{\alggrp{M}}}$.
Therefore we have $T_{n_w}^*T_{n_v}\in T_{n_{w_\Delta w_{\Delta_{\alggrp{M}}}}} + \sum_{w' < w_\Delta w_{\Delta_{\alggrp{M}}}}C[Z_\kappa]T_{n_{w'}}$.
We have $\varphi'(T_{n_{w_\Delta w_{\Delta_{\alggrp{M}}}}}) = x$.
Moreover, for $w'\in W$ such that $w' < wv = w_\Delta w_{\Delta_{\alggrp{M}}}$, $w'\notin w_\Delta w_{\Delta_{\alggrp{M}}}W_{\alggrp{M}}$ by Lemma~\ref{lem:w_Delta w_Delta_M is maximal}.
Hence by Remark~\ref{rem:the condition to vanish at E_-(n_wlambda)}, we get $\varphi'(C[Z_\kappa]T_{n_{w'}}) = 0$.
Therefore $\varphi'(T_{n_w}^*T_{n_v}) = x$.
\end{proof}

\subsection{Extension of a module}
In this subsection, assume that an orthogonal decomposition $\Delta = \Delta_1\amalg\Delta_2$ is given.
Then we have the decomposition $S = S_1\amalg S_2$ and $S_\aff = S_{\aff,1}\amalg S_{\aff,2}$.
Let $\alggrp{M}_1$ (resp.\ $\alggrp{M}_2$) be the Levi subgroup corresponding to $\Delta_1$ (resp.\ $\Delta_2$).

To formulate the classification theorem, we need the following proposition.
\begin{prop}\label{prop:extending}
Let $\sigma$ be an $\mathcal{H}_{\alggrp{M}_1}\otimes C$-module such that $E^{\alggrp{M}_1}(\lambda)$ acts trivially for $\lambda\in\Lambda_{\alggrp{M}_2}'(1)$.
Then there exists a unique $\mathcal{H}\otimes C$-module $e(\sigma)$ such that:
\begin{itemize}
\item as an $\mathcal{H}_{\alggrp{M}_1}^-$-module, $e(\sigma) = \sigma$.
\item for $s\in S_{\aff,2}$, $T_{n_s}$ is zero on $e(\sigma)$.
\item $T_t$ acts trivially for $t\in \Lambda'_{\alggrp{M}_2}(1)\cap Z_\kappa$.
\end{itemize}
\end{prop}
We call the module $e(\sigma)$ \emph{the extension of $\sigma$}.

We summarize some facts which come from the orthogonal decomposition.
We have a decomposition as Coxeter groups $\widetilde{W}_{\aff} = \widetilde{W}_{\alggrp{M}_1,\aff}\times\widetilde{W}_{\alggrp{M}_2,\aff}$.
Put $\ell_{\alggrp{M}_i}(\widetilde{w}) = \#(\Sigma_{\alggrp{M}_i,\aff}^+\cap \widetilde{w}(\Sigma_{\alggrp{M}_i,\aff}^-))$ for $i =1,2$.
Since $\Sigma_\aff^{\pm} = \Sigma_{\alggrp{M}_1,\aff}^{\pm}\amalg \Sigma_{\alggrp{M}_2,\aff}^{\pm}$ and the action of $\widetilde{W}$ preserves this decomposition, we have $\ell(\widetilde{w}) = \ell_{\alggrp{M}_1}(\widetilde{w}) + \ell_{\alggrp{M}_2}(\widetilde{w})$.
If $\widetilde{w}\in \alggrp{W}_{\alggrp{M}_i}(1)$, then $\ell_{\alggrp{M}_i}(\widetilde{w})$ is the usual length of $\widetilde{w}$.
It is easy to see that $\ell_{\alggrp{M}_i}(\widetilde{w}^{-1}) = \ell_{\alggrp{M}_i}(\widetilde{w})$ for $i =1,2$.
If $\ell_{\alggrp{M}_1}(\widetilde{w}) = 0$, then $\widetilde{w}(\Sigma^-_{\alggrp{M}_1,\aff}) \subset \Sigma^-_{\alggrp{M}_1,\aff}$.
By $\ell_{\alggrp{M}_1}(\widetilde{w}^{-1}) = 0$, we have $\widetilde{w}^{-1}(\Sigma^-_{\alggrp{M}_1,\aff}) \subset \Sigma^-_{\alggrp{M}_1,\aff}$, hence $\widetilde{w}(\Sigma^-_{\alggrp{M}_1,\aff}) \supset \Sigma^-_{\alggrp{M}_1,\aff}$.
Therefore $\widetilde{w}(\Sigma^-_{\alggrp{M}_1,\aff}) = \Sigma^-_{\alggrp{M}_1,\aff}$.
Hence for any $\widetilde{w}_1\in \widetilde{W}(1)$, we have $\widetilde{w}_1\widetilde{w}(\Sigma^-_{\alggrp{M}_1,\aff}) = \widetilde{w}_1(\Sigma^-_{\alggrp{M}_1,\aff})$.
We get $\ell_{\alggrp{M}_1}(\widetilde{w}_1\widetilde{w}) = \ell_{\alggrp{M}_1}(\widetilde{w}_1)$.
If $\widetilde{w}\in \widetilde{W}_{\alggrp{M}_1,\aff}(1)$, then its image in $\widetilde{W}$ is contained in $\widetilde{W}_{\alggrp{M}_1,\aff}$.
The group $\widetilde{W}_{\alggrp{M}_1,\aff}$ is generated by $S_{\aff,1}$ and, since the decomposition is orthogonal, $\widetilde{W}_{\alggrp{M}_1,\aff}$ acts trivially on $\Sigma_{\alggrp{M}_2,\aff}$.
Hence $\ell_{\alggrp{M}_2}(\widetilde{w}) = 0$ for $\widetilde{w}\in \widetilde{W}_{\alggrp{M}_1,\aff}(1)$.
Set $\Omega = \{\widetilde{w}\in \widetilde{W}\mid \ell(\widetilde{w}) = 0\}$.
If $\widetilde{w}\in \Omega$, then $\ell_{\alggrp{M}_1}(\widetilde{w}) + \ell_{\alggrp{M}_2}(\widetilde{w}) = \ell(\widetilde{w}) = 0$.
Hence $\ell_{\alggrp{M}_1}(\widetilde{w}) = \ell_{\alggrp{M}_2}(\widetilde{w}) = 0$.
We have $\widetilde{W} = \widetilde{W}_\aff\Omega = \widetilde{W}_{\alggrp{M}_1,\aff}\widetilde{W}_{\alggrp{M}_2,\aff}\Omega$.
Let $\Omega(1)$ be the inverse image of $\Omega$ by $\widetilde{W}(1)\to \widetilde{W}$.
Since $\widetilde{W}_{\alggrp{M}_i,\aff}(1)\to \widetilde{W}_{\alggrp{M}_i,\aff}$ is surjective for $i =1,2$, we get $\widetilde{W} = \widetilde{W}_{\alggrp{M}_1,\aff}(1)\widetilde{W}_{\alggrp{M}_2,\aff}(1)\Omega(1)$.
Namely, for any $\widetilde{w}\in\widetilde{W}(1)$, we can find $\widetilde{w}_1\in\widetilde{W}_{\alggrp{M}_1,\aff}(1)$, $\widetilde{w}_2\in\widetilde{W}_{\alggrp{M}_2,\aff}(1)$ and $u\in\widetilde{W}(1)$ such that $\ell(u) = 0$ and $\widetilde{w} = \widetilde{w}_1\widetilde{w}_2 u$.
Since $\ell_{\alggrp{M}_1}(\widetilde{w}_2) = \ell_{\alggrp{M_1}}(u) = 0$, we have $\ell_{\alggrp{M}_1}(\widetilde{w}) = \ell_{\alggrp{M}_1}(\widetilde{w}_1)$.
We have proved (1)--(3) of the following lemma.

\begin{lem}\label{lem:properties of decomposition}
Let $\widetilde{w},\widetilde{v}\in \widetilde{W}(1)$.
\begin{enumerate}
\item We have $\ell(\widetilde{w}) = \ell_{\alggrp{M}_1}(\widetilde{w}) + \ell_{\alggrp{M}_2}(\widetilde{w})$.
\item If $\ell_{\alggrp{M}_i}(\widetilde{w}) = 0$, then $\ell_{\alggrp{M}_i}(\widetilde{w}\widetilde{v}) = \ell_{\alggrp{M}_i}(\widetilde{v}\widetilde{w}) = \ell_{\alggrp{M}_i}(\widetilde{v})$ for $i = 1,2$.
\item There exists $\widetilde{w}_1\in \widetilde{W}_{\alggrp{M}_1,\aff}(1)$, $\widetilde{w}_2\in \widetilde{W}_{\alggrp{M}_2,\aff}(1)$ and $u\in \widetilde{W}(1)$ such that $\ell(u) = 0$ and $\widetilde{w} = \widetilde{w}_1\widetilde{w}_2u$.
We have $\ell_{\alggrp{M}_1}(\widetilde{w}) = \ell_{\alggrp{M}_1}(\widetilde{w}_1)$ and $\ell_{\alggrp{M}_2}(\widetilde{w}) = \ell_{\alggrp{M}_2}(\widetilde{w}_2)$.
In particular, there exists $\widetilde{w}_1\in \widetilde{W}_{\alggrp{M}_1,\aff}(1)$ and $u\in \widetilde{W}(1)$ such that $\ell_{\alggrp{M}_1}(u) = 0$ and $\widetilde{w} = \widetilde{w}_1u$.
\item\label{enum:on length, split to M_1,M_2} Let $\widetilde{w}_1,\widetilde{w}_2\in \widetilde{W}(1)$ such that $\ell_{\alggrp{M}_1}(\widetilde{w}_1) = 0$ and $\ell_{\alggrp{M}_2}(\widetilde{w}_2) = 0$.
Then $\ell(\widetilde{w}_1\widetilde{w}_2) = \ell(\widetilde{w}_1) + \ell(\widetilde{w}_2)$.
\item\label{enum:bruhat order and length_M_1}
 If $\widetilde{w}\le \widetilde{v}$, then $\ell_{\alggrp{M}_i}(\widetilde{w})\le \ell_{\alggrp{M}_i}(\widetilde{v})$ for $i = 1,2$
\item\label{enum:length formula for M_1,M_2}
For $w\in W$, $\lambda\in \Lambda(1)$ and $i = 1,2$, we have
\[
	\ell_{\alggrp{M}_i}(\lambda n_w) = \sum_{\alpha\in\Sigma_{\alggrp{M}_i}^+,w^{-1}(\alpha) > 0}\lvert\langle\nu(\lambda),\alpha\rangle\rvert + \sum_{\alpha\in\Sigma_{\alggrp{M}_i}^+,w^{-1}(\alpha) < 0}\lvert\langle\nu(\lambda),\alpha\rangle - 1\rvert.
\]
\item\label{enum:normality of Lambda dash}
For $i = 1,2$, $\Lambda'_{\alggrp{M}_i}(1)$ is normal in $\widetilde{W}(1)$.
\end{enumerate}
\end{lem}
\begin{proof}
We may assume $i = 1$ in \ref{enum:bruhat order and length_M_1}, \ref{enum:length formula for M_1,M_2} and \ref{enum:normality of Lambda dash}.

We prove \ref{enum:on length, split to M_1,M_2}.
Since $\ell_{\alggrp{M}_1}(\widetilde{w}_1) = 0$, $\ell_{\alggrp{M}_1}(\widetilde{w}_1\widetilde{w}_2) = \ell_{\alggrp{M}_1}(\widetilde{w}_2)$.
Hence
\begin{align*}
\ell(\widetilde{w}_1\widetilde{w}_2) & = \ell_{\alggrp{M}_1}(\widetilde{w}_1\widetilde{w}_2) + \ell_{\alggrp{M}_2}(\widetilde{w}_1\widetilde{w}_2)\\
& = \ell_{\alggrp{M}_1}(\widetilde{w}_2) + \ell_{\alggrp{M}_2}(\widetilde{w}_1)\\
& = (\ell_{\alggrp{M}_1}(\widetilde{w}_2) + \ell_{\alggrp{M}_2}(\widetilde{w}_2)) + (\ell_{\alggrp{M}_2}(\widetilde{w}_1) + \ell_{\alggrp{M}_2}(\widetilde{w}_1))\\
& = \ell(\widetilde{w}_1) + \ell(\widetilde{w}_2).
\end{align*}

For \ref{enum:bruhat order and length_M_1}, we prove the same statement in the group $\widetilde{W}$.
Let $\widetilde{w},\widetilde{v}\in \widetilde{W}$.
Take $\widetilde{w}_1,\widetilde{v}_1\in \widetilde{W}_{\alggrp{M}_1,\aff}$, $\widetilde{w}_2,\widetilde{v}_2\in \widetilde{W}_{\alggrp{M}_2,\aff}$ and $u,u'\in \Omega$ such that $\widetilde{w} = \widetilde{w}_1\widetilde{w}_2u$ and $\widetilde{v} =\widetilde{v}_1\widetilde{v}_2u'$.
Then $\widetilde{w}\le\widetilde{v}$ implies $\widetilde{w}_1\le \widetilde{v}_1$, $\widetilde{w}_2\le\widetilde{v}_2$ and $u = u'$.
Since $\widetilde{w}_1,\widetilde{v}_1\in \widetilde{W}_{\alggrp{M}_1,\aff}$, if $\widetilde{w}_1\le \widetilde{v}_1$ with respect to the Bruhat order of $\widetilde{W}$, then $\widetilde{w}_1\le \widetilde{v}_1$ with respect to the Bruhat order of $\widetilde{W}_{\alggrp{M}_1}$.
Hence $\ell_{\alggrp{M}_1}(\widetilde{w}_1)\le \ell_{\alggrp{M}_1}(\widetilde{v}_1)$.
On the other hand, we have $\ell_{\alggrp{M}_1}(\widetilde{w}) = \ell_{\alggrp{M}_1}(\widetilde{w}_1)$ and $\ell_{\alggrp{M}_1}(\widetilde{v}) = \ell_{\alggrp{M}_1}(\widetilde{v}_1)$.
Hence we get \ref{enum:bruhat order and length_M_1}.

We prove \ref{enum:length formula for M_1,M_2}.
If $w\in W_{\alggrp{M}_1}$, this is the length formula \eqref{eq:length formula}.
In general, take $w_1\in W_{\alggrp{M}_1}$ and $w_2\in W_{\alggrp{M}_2}$ such that $w = w_1w_2$.
Then we have $\ell_{\alggrp{M}_1}(w_2) = 0$.
Hence $\ell_{\alggrp{M}_1}(\lambda n_{w_1w_2}) = \ell_{\alggrp{M}_1}(\lambda n_{w_1})$.
Since $\lambda n_{w_1}\in \widetilde{W}_{\alggrp{M}_1}(1)$, applying the length formula \eqref{eq:length formula} we get
\[
	\ell_{\alggrp{M}_1}(\lambda n_w) = \sum_{\alpha\in\Sigma_{\alggrp{M}_1}^+,w_1^{-1}(\alpha) > 0}\lvert\langle\nu(\lambda),\alpha\rangle\rvert + \sum_{\alpha\in\Sigma_{\alggrp{M}_1}^+,w_1^{-1}(\alpha) < 0}\lvert\langle\nu(\lambda),\alpha\rangle - 1\rvert.
\]
For $\alpha\in \Sigma_{\alggrp{M}_1}$, $w_2^{-1}w_1^{-1}(\alpha) = w_1^{-1}w_2^{-1}(\alpha) = w_1^{-1}(\alpha)$ since $w_2\in W_{\alggrp{M}_2}$ and $w_1\in W_{\alggrp{M}_1}$.
Therefore $w^{-1}(\alpha) = (w_1w_2)^{-1}(\alpha) = w_1^{-1}(\alpha)$.
Since $M_1'$ is normal in $G$ \cite[II.7 Remark 4]{arXiv:1412.0737}, we have \ref{enum:normality of Lambda dash}.
\end{proof}

Now we give a proof of Proposition~\ref{prop:extending}.
For the proof, let $I$ be the two-sided ideal of $\mathcal{H}\otimes C$ generated by $\{T_{n_s}\mid s\in S_{\aff,2}\}\cup \{T_t - 1\mid t\in \Lambda'_{\alggrp{M}_2}(1)\cap Z_\kappa\}$ and $J$ the two-sided ideal of $\mathcal{H}_{\alggrp{M}_1}\otimes C$ generated by $\{E^{\alggrp{M}_1}(\lambda) - 1\mid \lambda\in\Lambda'_{\alggrp{M}_2}(1)\}$.
Then Proposition~\ref{prop:extending} follows from
\[
	(\mathcal{H}_{\alggrp{M}_1}^-\otimes C)/(J\cap (\mathcal{H}_{\alggrp{M}_1}^-\otimes C))
	\xrightarrow{\sim}
	(\mathcal{H}\otimes C)/I.
\]
We prove it.
We need a lemma.
\begin{lem}\label{lem:lemma for the proof of extending}
\begin{enumerate}
\item\label{enum:ideal I, as W_aff-orbit} We have
\begin{align*}
I & = \bigoplus_{\widetilde{w}\in \widetilde{W}_\aff(1)\backslash\widetilde{W}(1)}\left(I\cap \bigoplus_{\widetilde{v}\in \widetilde{W}_\aff(1)}CT_{\widetilde{v}\widetilde{w}}\right)\\
& = \bigoplus_{\widetilde{w}\in \widetilde{W}_\aff(1)\backslash\widetilde{W}(1)}\left(I\cap \bigoplus_{\widetilde{v}\in \widetilde{W}_\aff(1)}CE(\widetilde{v}\widetilde{w})\right).
\end{align*}
\item\label{enum:T_w^* in 1 + I} If $\widetilde{w}\in \widetilde{W}_{\alggrp{M}_2,\aff}(1)$, then $T_{\widetilde{w}}^* \in 1 + I$.
\item\label{enum:on ideal I, split to M_1,M_2} Let $\lambda\in \Lambda(1)$, $w\in W_{\alggrp{M}_1}$, $a\in \widetilde{W}_{\alggrp{M}_2,\aff}(1)$ and $b\in \widetilde{W}(1)$ such that $\lambda n_w$ is $\alggrp{M}_1$-negative, $\lambda n_w = ab$ and $\ell_{\alggrp{M}_2}(b) = 0$.
Then $E(\lambda n_w) \in T_b + \sum_{c < b}CT_c + I$.
\item\label{enum:on ideal I, ell_M_2 = 0 part} The subspace $(\bigoplus_{\widetilde{w}\in \widetilde{W}(1),\ell_{\alggrp{M}_2}(\widetilde{w}) = 0}CT_{\widetilde{w}}) \cap I$ is in the ideal of $\mathcal{H}$ generated by $\{T_t - 1\mid t\in \Lambda_{\alggrp{M}_2}'(1)\cap Z_\kappa\}$.
\end{enumerate}
\end{lem}
\begin{proof}
Since $\{n_s\mid s\in S_{\aff,2}\}$ and $\Lambda'_{\alggrp{M}_2}(1)$ are contained in $\widetilde{W}_\aff(1)$, the statement \ref{enum:ideal I, as W_aff-orbit} is easy.
For \ref{enum:T_w^* in 1 + I}, we may assume $\widetilde{w} = n_s$ for $s\in S_{\aff,2}$ or $\widetilde{w} = t$ for $t\in \Lambda_{\alggrp{M}_2}'(1)\cap Z_\kappa$.
If $\widetilde{w} = n_s$, then $T_{n_s}^* = T_{n_s} - c_s$.
By the description of $c_s$, we have $c_s \in -1 + I$.
Hence $T_{n_s}^* \in 1 + I$.
For $t\in \Lambda_{\alggrp{M}_2}'(1)\cap Z_\kappa$, $T_t^* = T_t \in 1 + I$ by the definition of $I$.

We prove \ref{enum:on ideal I, split to M_1,M_2}.
Take anti-dominant elements $\lambda_1,\lambda_2\in \Lambda(1)$ such that $\lambda = \lambda_1^{-1}\lambda_2$.
Since $\lambda$ is $\alggrp{M}_1$-negative, namely $\langle \nu(\lambda),\alpha\rangle\ge 0$ for any $\alpha\in\Sigma^+\setminus\Sigma^+_{\alggrp{M}_1} = \Sigma^+_{\alggrp{M}_2}$, we can take $\lambda_1,\lambda_2$ such that $\langle\nu(\lambda_2),\alpha\rangle = 0$ for all $\alpha\in\Sigma_{\alggrp{M}_2}^+$.
Then by the length formula \eqref{eq:length formula}, we have $\ell_{\alggrp{M}_2}(\lambda_2) = 0$.
Take $\widetilde{v}_1\in \widetilde{W}_{\alggrp{M}_2,\aff}(1)$ and $\widetilde{v}_2\in \widetilde{W}(1)$ such that $\lambda_1^{-1} = \widetilde{v}_1\widetilde{v}_2$ and $\ell_{\alggrp{M}_2}(\widetilde{v}_2) = 0$.
Then $\lambda n_w = \widetilde{v}_1 \widetilde{v}_2 \lambda_2 n_w$.
Put $a' = \widetilde{v}_1$ and $b' = \widetilde{v}_2 \lambda_2 n_w$.
Then $a'\in \widetilde{W}_{\alggrp{M}_2,\aff}(1)$ and $\ell_{\alggrp{M}_2}(b') = 0$.
Set $u = (a')^{-1}a = b'b^{-1}$.
Then $u\in \widetilde{W}_{\alggrp{M}_2,\aff}(1)$ and $\ell_{\alggrp{M}_2}(u) = 0$.
Hence the image of $u$ in $\widetilde{W}$ is in $\{w\in \widetilde{W}_{\alggrp{M}_2,\aff}\mid \ell_{\alggrp{M}_2}(w) = 0\} = \{1\}$.
Therefore $u\in Z_\kappa$.
Hence $u\in Z_\kappa\cap W_{\alggrp{M}_2,\aff}(1) = \Lambda'_{\alggrp{M}_2}(1)\cap Z_\kappa$.
Replacing $\widetilde{v}_1$ with $\widetilde{v}_1u$ and $\widetilde{v}_2$ with $u^{-1}\widetilde{v}_2$ (these satisfy the same conditions), we have $a = \widetilde{v}_1$ and $b = \widetilde{v}_2 \lambda_2 n_w$.

By \ref{enum:on length, split to M_1,M_2} of the above lemma, we have $\ell(\lambda_1^{-1}) = \ell(\widetilde{v}_1) + \ell(\widetilde{v}_2)$.
Hence we have
\begin{align*}
E(\lambda n_w) & \in C[q_s^{\pm 1/2}]E(\lambda_1^{-1})E(\lambda_2n_w)\\
& = C[q_s^{\pm 1/2}]T_{\lambda_1^{-1}}^*E(\lambda_2 n_w)\\
& = C[q_s^{\pm 1/2}]T_a^*T_{\widetilde{v}_2}^*E(\lambda_2 n_w)\\
& \subset T_a^*\left(\sum_{c_1 \le \lambda_2 n_w}C[q_s^{\pm 1/2}]T_{\widetilde{v}_2}^*T_{c_1}^*\right)\\
& \subset T_a^*\left(\sum_{c_1 \le \lambda_2 n_w,c_2\le \widetilde{v}_2,c_3\le c_1,c_4\le c_2c_3}C[q_s^{\pm 1/2}]T_{c_4}^*\right).
\end{align*}
Take $c_1,c_2,c_3,c_4$ as above.
Since $\ell_{\alggrp{M}_2}(\lambda_2n_w) = \ell_{\alggrp{M}_2}(\widetilde{v}_2) = 0$, we have $\ell_{\alggrp{M}_2}(c_1) = \ell_{\alggrp{M}_2}(c_2) = 0$.
Hence $\ell_{\alggrp{M}_2}(c_3) = \ell_{\alggrp{M}_2}(c_4) = 0$.
Since $\ell_{\alggrp{M}_1}(a) = 0$, we have $\ell(ac_4) = \ell(a) + \ell(c_4)$ by \ref{enum:on length, split to M_1,M_2} of the above lemma.
Therefore,
\[
	E(\lambda n_w) \in \left(\sum_{\ell_{\alggrp{M}_2}(c) = 0}C[q_s^{\pm 1/2}]T_{ac}^*\right) \cap \mathcal{H} = \sum_{\ell_{\alggrp{M}_2}(c) = 0}C[q_s^{1/2}]T_{ac}^*
\]
Take $d_c\in C[q_s^{1/2}]$ such that $E(\lambda n_w) = \sum_{\ell_{\alggrp{M}_2}(c) = 0}d_cT_{ac}^*$.
Since $E(\lambda n_w) = E(ab) \in T^*_{ab} + \sum_{\widetilde{w} < ab}C[q_s^{1/2}]T_{\widetilde{w}}^*$, $d_b = 1$ and if $d_c\ne 0$ and $c\ne b$, then $ac < ab$.
We have $\ell_{\alggrp{M}_1}(a) = 0$, hence $\ell(ac) = \ell(a) + \ell(c)$ and $\ell(ab) = \ell(a) + \ell(b)$ by \ref{enum:on length, split to M_1,M_2} of the above lemma.
Therefore $d_c\ne 0$ implies $c \le b$.
We have $E(\lambda n_w) \in T_a^*(T^*_b + \sum_{c < b}CT_c^*) = T_a^*(T_b + \sum_{c < b}CT_c)$ in $\mathcal{H}\otimes C$.
We get \ref{enum:on ideal I, split to M_1,M_2} from \ref{enum:T_w^* in 1 + I}.

We prove \ref{enum:on ideal I, ell_M_2 = 0 part}.
It is sufficient to prove that the ideal generated by $\{T_{n_s}\mid s\in S_{\aff,2}\}$ is equal to $\bigoplus_{\ell_{\alggrp{M}_2}(\widetilde{w}) > 0}CT_{\widetilde{w}}$.
Let $I'$ be a such ideal.
Then $I'\supset \bigoplus_{\ell_{\alggrp{M}_2}(\widetilde{w}) > 0}CT_{\widetilde{w}}$ is obvious.
For the reverse inclusion, we prove that $\bigoplus_{\ell_{\alggrp{M}_2}(\widetilde{w}) > 0}CT_{\widetilde{w}}$ is an ideal.

Let $\widetilde{w}\in \widetilde{W}(1)$ such that $\ell_{\alggrp{M}_2}(\widetilde{w})> 0$ and $\widetilde{v}\in\widetilde{W}(1)$.
We prove $T_{\widetilde{w}} T_{\widetilde{v}} \in \bigoplus_{\ell_{\alggrp{M}_2}(\widetilde{w}') > 0}CT_{\widetilde{w}'}$.
We may assume $\widetilde{v} = n_s$ for $s\in S_\aff$.
If $\widetilde{w}n_s > \widetilde{w}$, then $T_{\widetilde{w}} T_{n_s} = T_{\widetilde{w}n_s}$ and $\ell_{\alggrp{M}_2}(\widetilde{w}n_s) = \ell_{\alggrp{M}_2}(\widetilde{w}) + \ell_{\alggrp{M}_2}(n_s) > \ell_{\alggrp{M}_2}(\widetilde{w}) > 0$.
If $\widetilde{w}n_s < \widetilde{w}$, then $T_{\widetilde{w}} T_{n_s} = T_{\widetilde{w}}c_s$.
\end{proof}

\begin{proof}[Proof of Proposition~\ref{prop:extending}]
We prove
\[
	\mathcal{H}_{\alggrp{M}_1}^-\otimes C
	\to
	(\mathcal{H}\otimes C)/I
\]
is surjective and the kernel is $J\cap (\mathcal{H}_{\alggrp{M}_1}^-\otimes C)$.

First we prove that the homomorphism is surjective.
Let $\widetilde{w}\in \widetilde{W}(1)$.
We prove that $T_{\widetilde{w}}$ is in the image.
If $\ell_{\alggrp{M}_2}(\widetilde{w}) > 0$, then $T_{\widetilde{w}}\in I$.
Hence we have nothing to prove.
Assume that $\ell_{\alggrp{M}_2}(\widetilde{w}) = 0$.
Since $T_{\widetilde{w}} \in E(\widetilde{w}) + \sum_{\widetilde{v} < \widetilde{w}}CE(\widetilde{v})$ and $\widetilde{v} < \widetilde{w}$ implies $\ell_{\alggrp{M}_2}(\widetilde{v}) = 0$ by Lemma~\ref{lem:properties of decomposition} \ref{enum:bruhat order and length_M_1}, it is sufficient to prove that $E(\widetilde{w})$ is in the image.
Take $\lambda\in\Lambda(1)$, $w_1\in W_{\alggrp{M}_1}$ and $w_2\in W_{\alggrp{M}_2}$ such that $\widetilde{w} = \lambda n_{w_1}n_{w_2}$.
Then we have
\[
	E(\widetilde{w}) = q_{\lambda n_{w_1}n_{w_2}}^{1/2}q_{\lambda n_{w_1}}^{-1/2}q_{n_{w_2}}^{-1/2}E(\lambda n_{w_1})T_{n_{w_2}}.
\]
By the assumption, we also have $\ell_{\alggrp{M}_2}(\lambda n_{w_1}n_{w_2}) = 0$.
We have $\ell_{\alggrp{M}_1}(n_{w_2}) = 0$.
Hence by Lemma~\ref{lem:properties of decomposition} \ref{enum:on length, split to M_1,M_2}, we have $\ell(\lambda n_{w_1}) = \ell(\lambda n_{w_1}n_{w_2}) + \ell(n_{w_2}^{-1}) = \ell(\lambda n_{w_1}n_{w_2}) + \ell(n_{w_2})$.
Therefore we have $q_{\lambda n_{w_1}} = q_{\lambda n_{w_1}n_{w_2}}q_{n_{w_2}}$.
Hence we get
\[
	E(\widetilde{w}) = q_{n_{w_2}}^{-1}E(\lambda n_{w_1})T_{n_{w_2}}.
\]
Therefore $E(\widetilde{w})T_{n_{w_2}^{-1}}^* = E(\lambda n_{w_1})$.
Since $w_2\in W_{\alggrp{M}_2}$, $T_{n_{w_2}^{-1}}^*\in 1 + I$ by the above lemma \ref{enum:T_w^* in 1 + I}.
Hence $E(\widetilde{w}) \in E(\lambda n_{w_1}) + I$.
By $\ell_{\alggrp{M}_2}(\lambda n_{w_1}n_{w_2}) = 0$ and Lemma~\ref{lem:properties of decomposition} \ref{enum:length formula for M_1,M_2}, $\langle\nu(\lambda),\alpha\rangle = 0$ or $1$ for $\alpha\in\Sigma_{\alggrp{M}_2}^+ = \Sigma^+\setminus\Sigma_{\alggrp{M}_1}^+$.
In particular, $\lambda n_{w_1}$ is $\alggrp{M}_1$-negative.
Hence $E(\widetilde{w})$ is in the image of the homomorphism.

The two-sided ideal $J$ is generated by $\{E^{\alggrp{M}_1}(\lambda) - 1\mid \lambda\in\Lambda'_{\alggrp{M}_2}(1)\}$.
If $\lambda\in \Lambda'_{\alggrp{M}_2}(1)$, then $\ell_{\alggrp{M}_1}(\lambda) = 0$.
By Lemma~\ref{lem:properties of decomposition} \ref{enum:normality of Lambda dash}, $\widetilde{w}\lambda \widetilde{w}^{-1}\in \Lambda'_{\alggrp{M}_2}(1)$ for $\widetilde{w}\in \widetilde{W}_{\alggrp{M}_1}(1)$ and in particular, we have $\ell_{\alggrp{M}_1}(\widetilde{w}\lambda \widetilde{w}^{-1}) = 0$.
Hence $E^{\alggrp{M}_1}(\widetilde{w})E^{\alggrp{M}_1}(\mu) = E^{\alggrp{M}_1}(\widetilde{w}\mu) = E^{\alggrp{M}_1}(\widetilde{w}\mu\widetilde{w}^{-1})E^{\alggrp{M}_1}(\widetilde{w})$.
Therefore $\sum_{\widetilde{w}\in \widetilde{W}_{\alggrp{M}_1}(1),\mu\in\Lambda'_{\alggrp{M}_2}(1)}C(E^{\alggrp{M}_1}(\mu) - 1)E^{\alggrp{M}_1}(\widetilde{w}) = \sum_{\widetilde{w}\in \widetilde{W}_{\alggrp{M}_1}(1),\mu\in\Lambda'_{\alggrp{M}_2}(1)}C(E^{\alggrp{M}_1}(\mu\widetilde{w}) - E^{\alggrp{M}_1}(\widetilde{w}))$ is a two-sided ideal and it is $J$.
Hence if $\sum_{\widetilde{w}\in\widetilde{W}_{\alggrp{M}_1}(1)}c_{\widetilde{w}}E^{\alggrp{M}_1}(\widetilde{w})\in J$, then $\sum_{\mu\in\Lambda'_{\alggrp{M}_2}(1)}c_{\mu\widetilde{w}} = 0$ for any $\widetilde{w}\in\widetilde{W}_{\alggrp{M}_1}(1)$.

Assume that $\sum_{\widetilde{w}\in\widetilde{W}_{\alggrp{M}_1}(1)}c_{\widetilde{w}}E^{\alggrp{M}_1}(\widetilde{w})\in J$.
For each $x\in \Lambda'_{\alggrp{M}_2}(1)\backslash \widetilde{W}_{\alggrp{M}_1}(1)$, fix a representative $\widetilde{w}(x)\in \widetilde{W}_{\alggrp{M}_1}(1)$ which is $\alggrp{M}_1$-negative.
Then we have
\begin{align*}
&\sum_{\widetilde{w}\in\widetilde{W}_{\alggrp{M}_1}(1)}c_{\widetilde{w}}E^{\alggrp{M}_1}(\widetilde{w})\\
& = \sum_{x\in \Lambda'_{\alggrp{M}_2}(1)\backslash\widetilde{W}_{\alggrp{M}_1}(1)}\sum_{\mu\in \Lambda'_{\alggrp{M}_2}(1)}c_{\mu \widetilde{w}(x)}(E^{\alggrp{M}_1}(\mu\widetilde{w}(x)) - E^{\alggrp{M}_1}(\widetilde{w}(x))).
\end{align*}
Therefore $(\mathcal{H}_{\alggrp{M}}^-\otimes C)\cap J$ is contained in the subspace generated by $\{E^{\alggrp{M}_1}(\widetilde{w}_1) - E^{\alggrp{M}_1}(\widetilde{w}_2)\mid \widetilde{w}_1\widetilde{w}_2^{-1}\in \Lambda'_{\alggrp{M}_2}(1),\ \text{$\widetilde{w}_1,\widetilde{w}_2$ are $\alggrp{M}_1$-negative}\}$.

Let $J'$ be the kernel of $\mathcal{H}_{\alggrp{M}_1}^-\otimes C\to (\mathcal{H}\otimes C)/I$.
We prove $J\cap (\mathcal{H}_{\alggrp{M}_1}^-\otimes C) \subset J'$.
By the above argument, to prove it, it is sufficient to prove that $E(\lambda_1 n_w) - E(\lambda_2 n_w)\in I$ if $\lambda_1n_w,\lambda_2n_w$ are $\alggrp{M}_1$-negative and $\lambda_1\lambda_2^{-1}\in \Lambda'_{\alggrp{M}_2}(1)$.

Let $\mu \in \Lambda'_{\alggrp{M}_2}(1)$ be a dominant element.
By the length formula \eqref{eq:length formula}, we have
\begin{align*}
& \ell(\mu \lambda n_w)\\ & = \sum_{\alpha\in\Sigma^+, w^{-1}(\alpha) > 0}\lvert\langle\nu(\mu),\alpha\rangle + \langle\nu(\lambda),\alpha\rangle\rvert + \sum_{\alpha\in\Sigma^+, w^{-1}(\alpha) < 0}\lvert\langle\nu(\mu),\alpha\rangle + \langle\nu(\lambda),\alpha\rangle - 1\rvert
\end{align*}
If $w^{-1}(\alpha) < 0$, then since $w\in W_{\alggrp{M}_1}$, we have $\alpha\in\Sigma^+_{\alggrp{M}_1}$.
Since $\mu\in \Lambda'_{\alggrp{M}_2}(1)$, we have $\langle \nu(\mu),\alpha\rangle = 0$.
Therefore we have $\lvert\langle\nu(\mu),\alpha\rangle + \langle\nu(\lambda),\alpha\rangle - 1\rvert = \lvert\langle\nu(\lambda),\alpha\rangle - 1\rvert = \lvert\langle\nu(\mu),\alpha\rangle\rvert + \lvert\langle\nu(\lambda),\alpha\rangle - 1\rvert$.
If $w^{-1}(\alpha) > 0$ and $\alpha\in\Sigma^+_{\alggrp{M}_1}$, then again we have $\langle \nu(\mu),\alpha\rangle = 0$.
Hence $\lvert\langle\nu(\mu),\alpha\rangle + \langle\nu(\lambda),\alpha\rangle\rvert = \lvert\langle\nu(\lambda),\alpha\rangle\rvert = \lvert\langle\nu(\mu),\alpha\rangle\rvert + \lvert\langle\nu(\lambda),\alpha\rangle\rvert$.
If $\alpha\in\Sigma^+_{\alggrp{M}_2}$ (hence $w^{-1}(\alpha) > 0$), we have $\langle \nu(\lambda),\alpha\rangle \ge 0$ since $\lambda$ is $\alggrp{M}_1$-negative.
Since $\mu$ is dominant, we have $\langle \nu(\mu),\alpha\rangle \ge 0$.
Hence $\lvert\langle\nu(\mu),\alpha\rangle + \langle\nu(\lambda),\alpha\rangle\rvert = \lvert\langle\nu(\mu),\alpha\rangle\rvert + \lvert\langle\nu(\lambda),\alpha\rangle\rvert$.
Therefore we get
\begin{align*}
& \ell(\mu \lambda n_w)\\ & = \sum_{\alpha\in\Sigma^+}\lvert\langle\nu(\mu),\alpha\rangle\rvert + \sum_{\alpha\in\Sigma^+, w^{-1}(\alpha) > 0}\lvert\langle\nu(\lambda)\alpha\rangle\rvert + \sum_{\alpha\in\Sigma^+, w^{-1}(\alpha) < 0}\lvert\langle\nu(\lambda),\alpha\rangle - 1\rvert\\
& = \ell(\mu) + \ell(\lambda n_w).
\end{align*}
Hence $E(\mu \lambda n_w) = E(\mu)E(\lambda n_w)$.
Since $E(\mu) = T_\mu^* \in 1 + I$ by Proposition~\ref{prop:example of E} and the above lemma, we have $E(\mu \lambda n_w) - E(\lambda n_w)\in I$.
Let $\lambda_1,\lambda_2$ be $\alggrp{M}_1$-negative elements such that $\lambda_1^{-1}\lambda_2\in \Lambda'_{\alggrp{M}_2}(1)$.
Since $\Lambda'_{\alggrp{M}_2}(1)$ is normal in $\Lambda(1)$ by Lemma~\ref{lem:properties of decomposition} \ref{enum:normality of Lambda dash}, we also have $\lambda_2\lambda_1^{-1}\in \Lambda'_{\alggrp{M}_2}(1)$.
Take a dominant $\mu_0\in \Lambda'_{\alggrp{M}_2}(1)$ such that $\mu_0\lambda_2 \lambda_1^{-1}$ is dominant.
Then for $w\in W_{\alggrp{M}_1}$,
\begin{align*}
&E(\lambda_1n_w) - E(\lambda_2n_w)\\
& = (E(\lambda_1 n_w) - E((\mu_0 \lambda_2 \lambda_1^{-1})\lambda_1n_w)) + (E(\mu_0\lambda_2 n_w) - E(\lambda_2 n_w))\in I
\end{align*}

Finally, we prove $J'\subset J\cap (\mathcal{H}_{\alggrp{M}_1}^-\otimes C)$.
Take $\sum_{\widetilde{w}\in \widetilde{W}_{\alggrp{M}_1}(1)}c_{\widetilde{w}}E^{\alggrp{M}_1}(\widetilde{w})\in J'$.
Then for any $\widetilde{v}\in \widetilde{W}(1)$, we have $\sum_{\widetilde{w}\in \widetilde{W}_\aff(1)}c_{\widetilde{w}\widetilde{v}}E(\widetilde{w}\widetilde{v})\in I$ by the above lemma (1).
Hence we may assume $c_{\widetilde{w}} = 0$ if $\widetilde{w}\notin \widetilde{W}_\aff(1)\widetilde{v}$ for a fixed $\widetilde{v}\in\widetilde{W}(1)$.
Since $\widetilde{W}_\aff(1)\widetilde{W}_{\alggrp{M}_1}(1)\supset \widetilde{W}_\aff(1)\Lambda(1) = \widetilde{W}(1)$, we may assume $\widetilde{v}\in \widetilde{W}_{\alggrp{M}_1}(1)$.
Moreover, we may assume $\widetilde{v}$ is $\alggrp{M}_1$-negative.
Take $\lambda'\in \Lambda(1)$ and $w'\in W_{\alggrp{M}_1}$ such that $\widetilde{v} = \lambda' n_{w'}$.
Let $\lambda\in \Lambda(1)$, $w\in W$ such that $\lambda n_w\in \widetilde{W}_\aff(1)\widetilde{v}$.
Then $\lambda' n_{w'}n_w^{-1}\lambda^{-1} \in \widetilde{W}_\aff(1)$.
We have $\lambda' n_{w'}n_w^{-1}\lambda^{-1} = \lambda'\lambda^{-1}(\lambda n_{w'}n_w^{-1}\lambda^{-1})$ and $\lambda n_{w'}n_w^{-1}\lambda^{-1}\in \widetilde{W}_\aff(1)$ since $\widetilde{W}_\aff(1)$ is normal in $\widetilde{W}(1)$ (see subsection~\ref{subsec:Notation}).
Hence $\lambda'\lambda^{-1}\in \widetilde{W}_\aff(1)\cap \Lambda(1) = \Lambda'(1)$.
By Lemma~\ref{lem:generators of Lambda'}, we have $\Lambda'(1) = \Lambda'_{\alggrp{M}_1}(1)\Lambda'_{\alggrp{M}_2}(1)$.
Take $\lambda_1\in \Lambda'_{\alggrp{M}_1}(1)$ and $\lambda_2\in\Lambda'_{\alggrp{M}_2}(1)$ such that $\lambda'\lambda^{-1} = \lambda_1\lambda_2$.
Since $\lambda_2\in \Lambda'_{\alggrp{M}_2}(1)$, we have $\ell_{\alggrp{M}_1}(\lambda_2) = 0$.
Hence $E^{\alggrp{M}_1}(\lambda n_w) - E^{\alggrp{M}_1}(\lambda_2\lambda n_w) = (1 - E^{\alggrp{M}_1}(\lambda_2))E^{\alggrp{M}_1}(\lambda n_w)\in J$.
Moreover, since $\lambda_1\in \Lambda'_{\alggrp{M}_1}(1)$, we have $\langle \alpha,\nu(\lambda_1)\rangle = 0$ for any $\alpha\in\Sigma^+_{\alggrp{M}_2}$.
By the orthogonal decomposition $\Delta = \Delta_1\amalg \Delta_2$, we have $\Sigma^+\setminus\Sigma_{\alggrp{M}_1}^+ = \Sigma^+_{\alggrp{M}_2}$.
Hence for any $\alpha\in\Sigma^+\setminus\Sigma_{\alggrp{M}_1}^+$, $\langle \alpha,\nu(\lambda')\rangle = \langle \alpha,\nu(\lambda_1\lambda_2\lambda)\rangle = \langle \alpha,\nu(\lambda_2\lambda)\rangle$.
Therefore, since $\lambda'$ is $\alggrp{M}_1$-negative, $\lambda_2\lambda$ is also $\alggrp{M}_1$-negative.
Hence we have $E^{\alggrp{M}_1}(\lambda n_w) - E^{\alggrp{M}_1}(\lambda_2\lambda n_w)\in J\cap (\mathcal{H}_{\alggrp{M}_1}^-\otimes C)$.
Therefore, to prove $\sum_{\widetilde{w}\in \widetilde{W}_{\alggrp{M}_1}(1)}c_{\widetilde{w}}E^{\alggrp{M}_1}(\widetilde{w})\in J\cap (\mathcal{H}_{\alggrp{M}_1}^-\otimes C)$, it is sufficient to prove that $\sum_{\widetilde{w}\in \widetilde{W}_{\alggrp{M}_1}(1)}c_{\widetilde{w}}E^{\alggrp{M}_1}(\widetilde{w}) - c_{\lambda n_w}(E^{\alggrp{M}_1}(\lambda n_w) - E^{\alggrp{M}_1}(\lambda_2\lambda n_w))\in J\cap (\mathcal{H}_{\alggrp{M}_1}^-\otimes C)$.
Notice that $\lambda'(\lambda_2 \lambda)^{-1} = \lambda_1\in\Lambda'_{\alggrp{M}_1}(1)$.
Therefore, we may assume that $c_{\lambda n_w}\ne 0$ implies $\lambda'\lambda^{-1}\in \Lambda'_{\alggrp{M}_1}(1)$.
Or, equivalently, since $\Lambda'_{\alggrp{M}_1}(1)$ is normal in $\Lambda(1)$, we may assume $\lambda^{-1}\lambda'\in \Lambda'_{\alggrp{M}_1}(1)$.
Hence $(\lambda n_w)^{-1}\lambda' n_{w'}\in \widetilde{W}_{\alggrp{M}_1,\aff}(1)$.
Take $a_0\in \widetilde{W}_{\alggrp{M}_2,\aff}(1)$ and $b_0\in \widetilde{W}(1)$ such that $\lambda' n_{w'} = a_0b_0$ and $\ell_{\alggrp{M}_2}(b_0) = 0$.
If $c_{\lambda n_w} \ne 0$, then $\ell_{\alggrp{M}_2}((\lambda' n_{w'})^{-1}\lambda n_{w}) = 0$.
Namely $\ell_{\alggrp{M}_2}(b_0^{-1}a_0^{-1}\lambda n_w) = 0$.
Since $\ell_{\alggrp{M}_2}(b_0) = 0$, $\ell_{\alggrp{M}_2}(a_0^{-1}\lambda n_w) = 0$.

Put $l = \max\{\ell(a_0^{-1}\lambda n_w)\mid c_{\lambda n_w}\ne 0\}$.
We prove $\sum c_{\lambda n_w}E^{\alggrp{M}_1}(\lambda n_w)\in J\cap (\mathcal{H}_{\alggrp{M}_1}^-\otimes C)$ by induction on $l$.
We have $\sum c_{\lambda n_w}E(\lambda n_w)\in I$.
By (3) of the above lemma, we have $E(\lambda n_w)\in T_{a_0^{-1}\lambda n_w} + \sum_{b < a_0^{-1}\lambda n_w}CT_b + I$.
Hence
\[
	\sum_{\ell(a_0^{-1}\lambda n_w) = l}c_{\lambda n_w}T_{a_0^{-1}\lambda n_w} - \sum_{\ell(b) < l,\ell_{\alggrp{M}_2}(b) = 0}c_bT_b \in I
\]
for some $c_b\in C$.
The left hand side is in $\bigoplus_{\ell_{\alggrp{M}_2}(\widetilde{w}) = 0}CT_{\widetilde{w}}$.
Hence it is contained in the ideal $I_0$ generated by $\{T_t - 1\mid t\in \Lambda_{\alggrp{M}_2}'(1)\cap Z_\kappa\}$ by the above lemma \ref{enum:on ideal I, ell_M_2 = 0 part}.
The ideal $I_0$ satisfies $I_0 = \bigoplus_{n\ge 0}\left(\bigoplus_{\ell(w) = n}CT_w\cap I_0\right)$.
Hence we have $\sum_{\ell(a_0^{-1}\lambda n_w) = l}c_{\lambda n_w}T_{a_0^{-1}\lambda n_w}\in I_0$.
Therefore, if $\ell(a_0^{-1}\lambda n_w) = l$, then we have $\sum_{t\in \Lambda'_{\alggrp{M}_2}(1)\cap Z_\kappa}c_{t\lambda n_w} = 0$.
Set $X = \{\lambda n_w\in \widetilde{W}_{\alggrp{M}_1}(1)\mid  \ell(a_0^{-1}\lambda n_w) = l\}$ and $Y = (\Lambda'_{\alggrp{M}_2}(1)\cap Z_\kappa)\backslash X$.
For each $y\in Y$ fix a representative $x(y)\in X$.
Then we have
\begin{align*}
&\sum_{\ell(a_0^{-1}\lambda n_w) = l}c_{\lambda n_w}E^{\alggrp{M}_1}(\lambda n_w)\\
& = \sum_{y\in Y}\sum_{t\in \Lambda'_{\alggrp{M}_2}(1)\cap Z_\kappa}c_{t x(y)}(E^{\alggrp{M}_1}(t x(y)) - E^{\alggrp{M}_1}(x(y)))\\
& = \sum_{y\in Y}\sum_{t\in \Lambda'_{\alggrp{M}_2}(1)\cap Z_\kappa}c_{t x(y)}(T^{\alggrp{M}_1}_t - 1)E^{\alggrp{M}_1}(x(y))
\in I_0\subset J.
\end{align*}
Since we have proved $J\cap \mathcal{H}_{\alggrp{M}_1}^-\otimes C\subset J'$, we have $\sum_{\ell(a_0^{-1}\lambda n_w) = l}c_{\lambda n_w}E^{\alggrp{M}_1}(\lambda n_w)\in J'$.
Therefore $\sum_{\ell(a_0^{-1}\lambda n_w) < l}c_{\lambda n_w}E^{\alggrp{M}_1}(\lambda n_w)\in J'$.
By inductive hypothesis, we have $\sum_{\ell(a_0^{-1}\lambda n_w) < l}c_{\lambda n_w}E^{\alggrp{M}_1}(\lambda n_w)\in J\cap \mathcal{H}_{\alggrp{M}_1}^-\otimes C$.
Hence we have $\sum c_{\lambda n_w}E^{\alggrp{M}_1}(\lambda n_w)\in J\cap \mathcal{H}_{\alggrp{M}_1}^-\otimes C$.
\end{proof}

\begin{prop}\label{prop:compatibility of extensions}
Let $\alggrp{P} = \alggrp{M}\alggrp{N}$ be a parabolic subgroup such that $\Delta_{\alggrp{M}}\supset\Delta_1$ and $\sigma$ as in Proposition~\ref{prop:extending}.
Denote the extension of $\sigma$ to $\mathcal{H}$ (resp.\ $\mathcal{H}_{\alggrp{M}}$) by $e(\sigma)$ (resp.\ $e_{\alggrp{M}}(\sigma)$).
Then $e(\sigma)|_{\mathcal{H}_{\alggrp{M}}^-}\simeq e_{\alggrp{M}}(\sigma)|_{\mathcal{H}_{\alggrp{M}}^-}$.
\end{prop}
\begin{proof}
Let $\widetilde{w}\in \widetilde{W}_{\alggrp{M}}(1)$ be a $\alggrp{M}$-negative element and we prove $e_{\alggrp{M}}(\sigma)(T_{\widetilde{w}}^{\alggrp{M},*}) = e(\sigma)(T_{\widetilde{w}}^*)$.
Let $\lambda\in \Lambda(1)$, $w_1\in W_{\alggrp{M}_1}$ and $w_2\in W_{\alggrp{M}_2}\cap W_{\alggrp{M}}$ such that $\widetilde{w} = \lambda n_{w_1}n_{w_2}$.
First we prove that $e(\sigma)(T_{\widetilde{w}}^*) = e(\sigma)(T_{\lambda n_{w_1}}^*)$ by induction on $\ell(w_2)$.
Take $s\in S_2$ such that $w_2 s < w_2$.
If $\lambda n_{w_1}n_{w_2}n_s > \lambda n_{w_1}n_{w_2}$, then $T_{\lambda n_{w_1}n_{w_2}}^*T_{n_s^{-1}}^* = T_{\lambda n_{w_1}n_{w_2s}}^*$.
If $\lambda n_{w_1}n_{w_2}n_s < \lambda n_{w_1}n_{w_2}$, then $T_{\lambda n_{w_1}n_{w_2}}^* = T_{\lambda n_{w_1}n_{w_2s}}^*T_{n_s}^*$.
By Lemma~\ref{lem:lemma for the proof of extending} \ref{enum:T_w^* in 1 + I}, we have $e(\sigma)(T_{n_s}^*) = e(\sigma)(T_{n_s^{-1}}^*) = 1$.
Hence in any case, we have $e(\sigma)(T^*_{\lambda n_{w_1}n_{w_2}}) = e(\sigma)(T^*_{\lambda n_{w_1}n_{w_2s}})$.
It completes the induction.
From the same argument, we have $e_{\alggrp{M}}(T^{\alggrp{M},*}_{\widetilde{w}}) = e_{\alggrp{M}}(T^{\alggrp{M},*}_{\lambda n_{w_1}})$.

Let $\alggrp{M}'_2$ be the Levi part of the parabolic subgroup corresponding to $\Delta_{\alggrp{M}}\cap \Delta_2$.
Take $\mu\in \Lambda_{\alggrp{M}'_2}'(1)$ such that $\langle \alpha,\nu(\mu)\rangle$ is sufficiently large for $\alpha\in\Sigma_{\alggrp{M}'_2}^+$.
Since $\widetilde{w}$ is $\alggrp{M}$-negative, $\langle \alpha,\nu(\lambda)\rangle\ge 0$ for $\alpha\in\Sigma^+\setminus\Sigma_{\alggrp{M}}^+ = \Sigma^+_{\alggrp{M}_2}\setminus\Sigma^+_{\alggrp{M}'_2}$.
Hence we can take $\mu$ such that $\langle \alpha,\nu(\mu \lambda)\rangle\ge 0$ for any $\alpha\in\Sigma^+_{\alggrp{M}_2} = \Sigma^+\setminus\Sigma_{\alggrp{M}_1}^+$, namely $\mu \lambda n_{w_1}$ is $\alggrp{M}_1$-negative.
The element $\mu$ is in $W_{\alggrp{M}'_2,\aff}(1)$.
Hence using the same argument in the above, we have $e_{\alggrp{M}}(\sigma)(T^{\alggrp{M},*}_{\lambda n_{w_1}}) = e_{\alggrp{M}}(\sigma)(T^{\alggrp{M},*}_{\mu\lambda n_{w_1}})$.
Since $\mu$ is also in $W_{\alggrp{M}_2,\aff}(1)$, we have $e(\sigma)(T^{*}_{\lambda n_{w_1}}) = e(\sigma)(T^{*}_{\mu\lambda n_{w_1}})$.
Finally, since $\mu \lambda n_{w_1}\in W_{\alggrp{M}_1}(1)$ is $\alggrp{M}_1$-negative, we have $e_{\alggrp{M}}(\sigma)(T^{\alggrp{M},*}_{\mu \lambda n_{w_1}}) = \sigma(T^{\alggrp{M}_1,*}_{\mu \lambda n_{w_1}}) = e(\sigma)(T^*_{\mu \lambda n_{w_1}})$.
\end{proof}

\subsection{Statement of the classification theorem}
We begin with the definition of supersingular representations.
\begin{defn}[{\cite[Proposition-Definition 5.9]{arXiv:1211.5366}, \cite[Definition 6.10]{Vigneras-prop-III}}]\label{defn:supersingular}
For each adjoint $\widetilde{W}(1)$-orbit $\mathcal{O}$ in $\Lambda(1)$, put $z_\mathcal{O} = \sum_{\lambda\in \mathcal{O}}E(\lambda)$.
Let $\mathcal{J}$ be the ideal generated by $\{z_\mathcal{O}\mid \ell(\lambda) > 0\ (\lambda\in\mathcal{O})\}$.
An $\mathcal{H}$-module $\pi$ is called \emph{supersingular} if there exists $n\in\Z_{>0}$ such that $\mathcal{J}^n\pi = 0$.
\end{defn}
Notice that $\ell(\lambda)$ is independent of $\lambda\in \mathcal{O}$.
By~\cite[Theorem~1.2]{Vigneras-prop-II}, $z_\mathcal{O}$ is in the center of $\mathcal{H}$.
We mainly use the characterization \ref{enum:supersingularity, by any X} and \ref{enum:supersingularity, by some X} in the following lemma.
\begin{lem}\label{lem:supersingurality, in terms of A-moudle}
Let $\pi$ be an irreducible $\mathcal{H}\otimes C$-module and $\mathcal{J}$ as in Definition~\ref{defn:supersingular}.
Then the following are equivalent.
\begin{enumerate}
\item\label{enum:supersingular} The module $\pi$ is supersingular.
\item\label{enum:supersingularity, Jpi = 0} We have $\mathcal{J}\pi = 0$.
\item\label{enum:supersingularity, by any X} For any irreducible $\mathcal{A}$-submodule $X$ of $\pi$, $\supp X = \Lambda_\Delta(1)$.
\item\label{enum:supersingularity, by some X} There exists an irreducible $\mathcal{A}$-submodule $X$ of $\pi$ such that $\supp X = \Lambda_\Delta(1)$.
\end{enumerate}
\end{lem}
\begin{proof}
Since $\mathcal{J}\pi$ is $\mathcal{H}$-stable, the equivalence between \ref{enum:supersingular} and \ref{enum:supersingularity, Jpi = 0} is clear.
Obviously, \ref{enum:supersingularity, by any X} implies \ref{enum:supersingularity, by some X}.

Notice that $\Lambda_\Delta(1) = \Lambda_\Delta^+(1)$.
Recall that $\supp X = \Lambda_\Delta(1)$ if and only if the action of $\mathcal{A}$ on $X$ factors through $\chi_\Delta$.
Let $X$ be an irreducible $\mathcal{A}$-submodule of $\pi$.
We prove $\mathcal{J}X = 0$ if and only if $\supp X = \Lambda_\Delta(1)$.
Let $\Theta\subset \Delta$ and $w\in W$ such that the action of $\mathcal{A}$ on $X$ factors through $w\chi_\Theta$.
By the length formula \eqref{eq:length formula}, $\ell(\lambda) = 0$ if and only if $\langle \nu(\lambda),\alpha\rangle = 0$ for any $\alpha\in\Sigma^+$, namely $\lambda\in \Lambda_\Delta(1)$.
Hence if $\ell(\lambda) > 0$, then $\chi_\Delta(E(\lambda)) = 0$.
Therefore, if $\Theta = \Delta$, then $\mathcal{J}X = 0$.

Assume that $\Theta \ne \Delta$.
We take $\lambda\in \Lambda(1)$ as follows.
Fix a uniformizer $\varpi$ of $F$ and consider the embedding $X_*(\alggrp{S})\hookrightarrow \alggrp{S}(F)$ defined by $\lambda\mapsto \lambda(\varpi)^{-1}$.
Then composition $X_*(\alggrp{S})\hookrightarrow \alggrp{S}(F)\hookrightarrow \alggrp{Z}(F)\to \Lambda(1)\xrightarrow{\nu}V$ is equal to the embedding $X_*(\alggrp{S})\hookrightarrow V$ and in particular it is injective.
Take $\lambda_0\in X_*(\alggrp{S})$ such that its image $\lambda$ in $\Lambda(1)$ satisfies $w\chi_\Theta(E(\lambda)) \ne 0$ and $\ell(\lambda) > 0$.
Since $\lambda_0\in \alggrp{S}(F)$ and $\alggrp{S}(F)$ is contained in the center of $\alggrp{Z}(F)$, $\lambda$ commutes with an element of $\Lambda(1)$.
Hence the orbit $\mathcal{O}$ through $\lambda$ is $\{n_w(\lambda)\mid w\in W\}$.
Moreover, if $w\in W$ satisfies $\nu(n_w(\lambda)) = n_w(\lambda)$, then we have $w(\lambda_0) = \lambda_0$.
Hence $n_w(\lambda) = \lambda$.
Therefore $\mathcal{O} = \{n_w(\lambda)\mid w\in W/\Stab_W(\nu(\lambda))\}$.
Since $w\chi_\Theta(E(\lambda))\ne 0$, $w^{-1}(\nu(\lambda))$ is dominant.
Hence, if $v\in W$ does not stabilize $\nu(\lambda)$, then $w^{-1}v^{-1}(\nu(\lambda))$ is not dominant.
Therefore we have $w\chi_\Theta(E(n_v(\lambda))) = 0$.
Hence $w\chi_\Theta(z_\mathcal{O}) = \sum_{v\in W/\Stab(\nu(\lambda))}w\chi_\Theta(E(n_v(\lambda))) = w\chi_\Theta(E(\lambda))$ and it is invertible.
Therefore $\mathcal{J}X\ne 0$.

We assume that $\pi$ is supersingular.
Let $X$ be an irreducible $\mathcal{A}$-submodule.
Then $\mathcal{J}X \subset \mathcal{J}\pi = 0$.
Hence we have $\supp X = \Lambda_\Delta(1)$.
Thereofre \ref{enum:supersingularity, Jpi = 0} implies \ref{enum:supersingularity, by any X}.
Assume that there exists an irreducible $\mathcal{A}$-submodule $X$ such that $\supp X = \Lambda_\Delta(1)$.
Then $\mathcal{J}X = 0$.
Hence $\{x\in \pi\mid \mathcal{J}x = 0\}\ne 0$.
Since the left hand side is $\mathcal{H}$-stable, $\{x\in \pi\mid \mathcal{J}x = 0\}= \pi$.
Therefore $\mathcal{J}\pi = 0$.
Hence \ref{enum:supersingularity, by some X} implies \ref{enum:supersingularity, Jpi = 0}.
\end{proof}

Let $\alggrp{P} = \alggrp{M}\alggrp{N}$ be a standard parabolic subgroup and $\sigma$ an irreducible supersingular representation of $\mathcal{H}_{\alggrp{M}}$.
Define
\[
	\Delta(\sigma) = \{\alpha\in\Delta\mid \langle\Delta_{\alggrp{M}},\check{\alpha}\rangle = 0, \text{$\sigma(E^{\alggrp{M}}(\lambda)) = 1$ for any $\lambda \in \Lambda_{s_\alpha}'(1)$}\}\cup \Delta_{\alggrp{M}}.
\]
Notice that if $\langle \Delta_{\alggrp{M}},\check{\alpha}\rangle = 0$ and $\lambda\in \Lambda_{s_\alpha}'(1)$, then $\langle \Delta_{\alggrp{M}},\nu(\lambda)\rangle = 0$ by Lemma~\ref{lem:structure of Lambda'_s}.
Hence $\ell_{\alggrp{M}}(\lambda) = 0$.
Therefore we have $E^{\alggrp{M}}(\lambda) = T^{\alggrp{M}}_\lambda$.

Let $\alggrp{P}(\sigma) = \alggrp{M}(\sigma)\alggrp{N}(\sigma)$ be the parabolic subgroup corresponding to $\Delta(\sigma)$.
Let $\alggrp{Q}$ be a parabolic subgroup such that $\alggrp{P}(\sigma)\supset \alggrp{Q}\supset \alggrp{P}$.
By Proposition~\ref{prop:extending}, we have the extension $e_{\alggrp{Q}}(\sigma)$ of $\sigma$ to $\mathcal{H}_{\alggrp{Q}}$.
By Proposition~\ref{prop:compatibility of extensions}, if $\alggrp{Q}'\supset \alggrp{Q}$, then we have
\begin{align*}
I_{\alggrp{Q}'}(e_{\alggrp{Q}'}(\sigma)) & = \Hom_{\mathcal{H}_{\alggrp{Q}'}^-}(\mathcal{H},e_{\alggrp{Q}'}(\sigma)|_{\mathcal{H}_{\alggrp{Q}'}^-})\\
& \subset \Hom_{\mathcal{H}_{\alggrp{Q}}^-}(\mathcal{H},e_{\alggrp{Q}'}(\sigma)|_{\mathcal{H}_{\alggrp{Q}}^-})\\
& = \Hom_{\mathcal{H}_{\alggrp{Q}}^-}(\mathcal{H},e_{\alggrp{Q}}(\sigma)|_{\mathcal{H}_{\alggrp{Q}}^-}) = I_{\alggrp{Q}}(e_{\alggrp{Q}}(\sigma)).
\end{align*}
Define an $\mathcal{H}$-module $I(\alggrp{P},\sigma,\alggrp{Q})$ by
\[
	I(\alggrp{P},\sigma,\alggrp{Q}) = \left.I_{\alggrp{Q}}(e_{\alggrp{Q}}(\sigma))\;\middle/\;\left(\sum_{\alggrp{P}(\sigma)\supset\alggrp{Q}'\supsetneq\alggrp{Q}}I_{\alggrp{Q}'}(e_{\alggrp{Q}'}(\sigma))\right)\right.
\]
The main theorem of this paper is the following.
\begin{thm}\label{thm:main theorem}
For any triple $(\alggrp{P},\sigma,\alggrp{Q})$, $I(\alggrp{P},\sigma,\alggrp{Q})$ is irreducible and any irreducible $\mathcal{H}$-module have this form.
Moreover, $I(\alggrp{P},\sigma,\alggrp{Q}) \simeq I(\alggrp{P}',\sigma',\alggrp{Q}')$ implies $\alggrp{P} = \alggrp{P}'$, $\sigma\simeq\sigma'$ and $\alggrp{Q} = \alggrp{Q}'$.
\end{thm}

\subsection{Irreducibility}
We prove that $I(\alggrp{P},\sigma,\alggrp{Q})$ is irreducible.
\begin{lem}\label{lem:I(P,sigma,Q) as A-mod}
The isomorphism $I_{\alggrp{Q}}(e_{\alggrp{Q}}(\sigma))\simeq \bigoplus_{w\in W^{\alggrp{Q}}}w\sigma_\mathcal{A}$ induces $I(\alggrp{P},\sigma,\alggrp{Q})\simeq \bigoplus_{\Delta_w\cap \Delta(\sigma) = \Delta_{\alggrp{Q}}}w\sigma_\mathcal{A}$.
\end{lem}
\begin{proof}
We remark that $\sigma_\mathcal{A} = e_{\alggrp{Q}}(\sigma)_\mathcal{A}$.
Indeed, if $\lambda\in \Lambda(1)$ is $\alggrp{M}$-negative, then $\sigma_\mathcal{A}(E(\lambda)) = \sigma(E^{\alggrp{M}}(\lambda)) = e_{\alggrp{Q}}(\sigma)(E^{\alggrp{Q}}(\lambda))$ by the definition of $\sigma_\mathcal{A}$ and the extension.
Since $\lambda$ is also $\alggrp{Q}$-negative, we have $e_{\alggrp{Q}}(\sigma)(E^{\alggrp{Q}}(\lambda)) = e_{\alggrp{Q}}(\sigma)_\mathcal{A}(E(\lambda))$.
If $\lambda$ is not $\alggrp{M}$-negative, take $\lambda_0^-\in \Lambda(1)$ as in Remark~\ref{rem:role of strongly positive/negative element}.
Then the above calculation shows $e_{\alggrp{Q}}(\sigma)_\mathcal{A}(E(\lambda_0^-)) = \sigma(E^{\alggrp{M}}(\lambda_0^-))$ and it is invertible.
On the other hand, since $\lambda$ is not $\alggrp{M}$-negative, $\nu(\lambda)$ and $\nu(\lambda_0^-)$ do not belong to the same closed chamber.
Hence $E(\lambda)E(\lambda_0^-) = 0$ by \eqref{eq:multiplication in A}.
Therefore we have $e_{\alggrp{Q}}(\sigma)_\mathcal{A}(E(\lambda)) = e_{\alggrp{Q}}(\sigma)_\mathcal{A}(E(\lambda)E(\lambda_0^-))e_{\alggrp{Q}}(\sigma)_\mathcal{A}(E(\lambda_0^-))^{-1} = 0$.
By the definition of $\sigma_\mathcal{A}$, we have $\sigma_\mathcal{A}(E(\lambda)) = 0$.
Hence we get $\sigma_\mathcal{A} = e_{\alggrp{Q}}(\sigma)_\mathcal{A}$.

For $\alggrp{P}(\sigma)\supset\alggrp{Q}'\supset\alggrp{Q}$, we have
\[
\begin{tikzcd}
I_{\alggrp{Q}'}(e_{\alggrp{Q}'}(\sigma))\arrow[hookrightarrow]{r}\arrow[-]{d}[above,rotate=90]{\sim} & I_{\alggrp{Q}}(e_{\alggrp{Q}}(\sigma))\arrow[-]{d}[above,rotate=90]{\sim}\\
\displaystyle\bigoplus_{w\in W^{\alggrp{Q}'}}w\sigma_\mathcal{A} \arrow[hookrightarrow]{r} & \displaystyle\bigoplus_{w\in W^{\alggrp{Q}}}w\sigma_\mathcal{A}.
\end{tikzcd}
\]
We write the low horizontal map explicitly.
If $\varphi\in I_{\alggrp{Q}'}(e_{\alggrp{Q}'}(\sigma))$, then its image in $\bigoplus_{w\in W^{\alggrp{Q}}}w\sigma_\mathcal{A}\simeq I_{\alggrp{Q}}(e_{\alggrp{Q}}(\sigma))$ is $(\varphi(T_{n_w}))_{w\in W^{\alggrp{Q}}}$.
Let $w\in W$ and decompose $w = w_1w_2$ such that $w_1\in W^{\alggrp{Q}'}$ and $w_2\in W_{\alggrp{Q}'}$.
Then we have $(\varphi(T_{n_w})) = (\varphi(T_{n_{w_1}}T_{n_{w_2}})) = (\varphi(T_{n_{w_1}})T_{n_{w_2}})$.
Hence the low horizontal map of the above diagram is given by $(x_w)_{w\in W^{\alggrp{Q}'}}\mapsto (x_{w_1}T_{n_{w_2}})_{w_1w_2}$.
For $\alpha\in \Sigma^+_{\alggrp{Q}}\subset \Sigma^+_{\alggrp{Q}'}$, $w_2(\alpha)\in \Sigma_{\alggrp{Q}'}$.
Hence $w_1w_2(\alpha) > 0$ if and only if $w_2(\alpha) > 0$.
Therefore, $w_1w_2\in W^{\alggrp{Q}}$ if and only if $w_2\in W^{\alggrp{Q}}$.

We have $w_2\in W_{\alggrp{Q}'}\cap W^{\alggrp{Q}}\subset W_{\alggrp{Q}'}\cap W^{\alggrp{P}}$.
Since we have the orthogonal decomposition $\Delta_{\alggrp{Q}'} = \Delta_{\alggrp{P}}\cup (\Delta_{\alggrp{Q}'}\setminus\Delta_{\alggrp{P}})$, the group $W_{\alggrp{Q}'}\cap W^{\alggrp{P}}$ is generated by $\{s_\alpha\mid \alpha\in \Delta_{\alggrp{Q}'}\setminus\Delta_{\alggrp{P}}\}$.
For such $\alpha$, $T_{n_{s_\alpha}} = 0$ on $e_{\alggrp{Q}'}(\sigma)$ by the definition of the extension.
Therefore, $x_{w_1}T_{n_{w_2}} = 0$ if $w_2\ne 1$.
Hence the low horizontal map of the above diagram is the inclusion map induced by $W^{\alggrp{Q}'}\hookrightarrow W^{\alggrp{Q}}$.
Therefore we have
\[
	I(\alggrp{P},\sigma,\alggrp{Q}) = \bigoplus_w w\sigma_\mathcal{A}
\]
where $w$ runs in $W^{\alggrp{Q}}$ such that $w\notin W^{\alggrp{Q}'}$ for any $\alggrp{P}(\sigma)\supset\alggrp{Q}'\supsetneq \alggrp{Q}$.
Notice that $w\in W^{\alggrp{Q}}$ if and only if $\Delta_w\supset \Delta_{\alggrp{Q}}$ by the definition.
Hence $w\in W^{\alggrp{Q}}$ and $w\notin W^{\alggrp{Q}'}$ for any $\alggrp{P}(\sigma)\supset\alggrp{Q}'\supsetneq \alggrp{Q}$ is equivalent  to $\Delta_w\supset \Delta_{\alggrp{Q}}$ and $\alpha\not\in \Delta_w$ for any $\alpha\in\Delta(\sigma)\setminus\Delta_{\alggrp{Q}}$ which is equivalent to $\Delta_w\cap \Delta(\sigma) = \Delta_{\alggrp{Q}}$.
\end{proof}

We use the notation $\Delta(X)$ from Corollary~\ref{cor:isom between std mod for simple module}.

\begin{lem}\label{lem:on A-submodule of sigma_A}
Let $X$ be an irreducible $\mathcal{A}$-submodule of $\sigma_\mathcal{A}$.
Then $\supp X = \Lambda_{\Delta_{\alggrp{M}}}^+(1)$ and $\Delta(X) = \Delta(\sigma)$.
\end{lem}
\begin{proof}
Take $\lambda_0^-\in \Lambda(1)$ as in Remark~\ref{rem:role of strongly positive/negative element}.
Then $\mathcal{A}_{\alggrp{M}} = (\mathcal{A}_{\alggrp{M}}\cap \mathcal{H}_{\alggrp{M}}^-)E^{\alggrp{M}}(\lambda_0^-)^{-1}$ by the argument in Remark~\ref{rem:role of strongly positive/negative element}.
Put $X' = X\sigma(\mathcal{A}_{\alggrp{M}}) = \bigcup_{n\in\Z_{\ge 0}}X\sigma(E^{\alggrp{M}}(\lambda_0^-)^{-n})$.
We prove that $X'$ is an irreducible $\mathcal{A}_{\alggrp{M}}$-module.
Take $x\in X'\setminus\{0\}$.
Then by multiplying the power of $E^{\alggrp{M}}(\lambda_0^-)$, we have $X\cap x\mathcal{A}_{\alggrp{M}}\ne 0$.
Since $X$ is an irreducible $\mathcal{A}$-module, we have $X \subset (X\cap x\mathcal{A}_{\alggrp{M}})\sigma_{\mathcal{A}}(\mathcal{A})$.
The definition of $\sigma_\mathcal{A}$ tells that $\sigma_\mathcal{A}(\mathcal{A}) = \sigma_\mathcal{A}(\mathcal{A}\cap \mathcal{H}_{\alggrp{M}}^-) = \sigma(\mathcal{A}_{\alggrp{M}}\cap \mathcal{H}_{\alggrp{M}}^-)$.
Hence $X \subset (X\cap x\mathcal{A}_{\alggrp{M}})\sigma(\mathcal{A}_{\alggrp{M}}\cap \mathcal{H}_{\alggrp{M}}^-)\subset x\sigma(\mathcal{A}_{\alggrp{M}})$.
Hence $X' = x\sigma(\mathcal{A}_{\alggrp{M}})$.
We get the irreducibility of $X'$.

Since $\sigma$ is supersingular, $\supp X' = \Lambda_{\Delta_{\alggrp{M}}}(1)$.
Recall that if $\lambda$ is $\alggrp{M}$-negative, then the action of $E(\lambda)$ on $\sigma_\mathcal{A}$ is the action of $E^{\alggrp{M}}(\lambda)$ on $\sigma$.
Hence for $\alggrp{M}$-negative $\lambda$, we have $\lambda\in \supp X$ if and only if $\lambda\in\Lambda_{\Delta_{\alggrp{M}}}(1)$.
By the definition of $\sigma_\mathcal{A}$, if $\lambda$ is not $\alggrp{M}$-negative then $\lambda\notin \supp X$.
Hence $\lambda$ is in $\supp X$ if and only if $\lambda$ is $\alggrp{M}$-negative and $\lambda\in \Lambda_{\Delta_{\alggrp{M}}}(1)$.
Namely we have $\supp X = \Lambda_{\Delta_{\alggrp{M}}}^+(1)$.

Hence $X$ is a $C[\Lambda_{\Delta_{\alggrp{M}}}(1)]$-module.
Assume that $\langle \alpha,\Delta_{\alggrp{M}}\rangle = 0$.
Then for $\lambda\in \Lambda_{s_\alpha}'(1)$, we have $\langle\nu(\lambda),\Delta_{\alggrp{M}}\rangle = 0$.
Namely, $\lambda\in \Lambda_{\Delta_{\alggrp{M}}}(1)$.
Take $\lambda_1,\lambda_2\in \Lambda_{\Delta_{\alggrp{M}}}^+(1)$ such that $\lambda = \lambda_1\lambda_2^{-1}$.
Then the action of $\tau_{\lambda_1}$ is equal to that of $E(\lambda_1)$ and, it is also equal to the action of $E^{\alggrp{M}}(\lambda_1)$ on $X\subset \sigma$.
It is also true if we replace $\lambda_1$ with $\lambda_2$.
Since $\tau_\lambda = \tau_{\lambda_1}\tau_{\lambda_2}^{-1}$ and $E^{\alggrp{M}}(\lambda) = E^{\alggrp{M}}(\lambda_1)E^{\alggrp{M}}(\lambda_2)^{-1}$, the action of $\tau_\lambda$ on $X$ and $E^{\alggrp{M}}(\lambda)$ on $X\subset \sigma$ coincide with each other.
If $\alpha\in \Delta(\sigma)$, then $E^{\alggrp{M}}(\lambda)$ is trivial on $\sigma$ for any $\lambda\in \Lambda'_{s_\alpha}(1)$.
Hence $\tau_\lambda$ is trivial on $X$, namely $\alpha\in \Delta(X)$.
Conversely, assume that $\alpha\in\Delta(X)$.
Then $\{x\in \sigma\mid xE^{\alggrp{M}}(\lambda) = x\ (\lambda\in \Lambda'_{s_\alpha}(1))\}\supset X$.
In particular it is not zero.
On the other hand, this subspace is stable under the action of $\mathcal{H}_{\alggrp{M}}$.
Hence this space is $\sigma$.
Therefore $\alpha\in\Delta(\sigma)$.
\end{proof}

\begin{lem}\label{lem:irreducible A-representation of parabolic induction}
There exists $\lambda\in n_{w_{\Delta}w_{\Delta_{\alggrp{Q}}}}(\Lambda_{\Delta_{\alggrp{Q}}}^+(1))$ such that $w\sigma_\mathcal{A}(E(\lambda)) = 0$ for any $w\in W^{\alggrp{Q}}\setminus \{w_\Delta w_{\Delta_{\alggrp{Q}}}\}$.
In particular, if $X$ is an irreducible $\mathcal{A}$-submodule of $I(\alggrp{P},\sigma,\alggrp{Q})$ such that $\supp X = n_{w_\Delta w_{\Delta_{\alggrp{Q}}}}(\Lambda_{\Delta_{\alggrp{Q}}}^+(1))$, then we have $X\subset w_{\Delta}w_{\Delta_{\alggrp{Q}}}\sigma$.
\end{lem}
\begin{proof}
First we prove the following claim.
\begin{claim}
If $w\in W^{\alggrp{Q}}\setminus \{w_\Delta w_{\Delta_{\alggrp{Q}}}\}$, then $(\Sigma^-\setminus \Sigma_{\alggrp{Q}}^-)\cap (w_\Delta w_{\Delta_{\alggrp{Q}}})^{-1} w(\Sigma^+\setminus \Sigma_{\alggrp{Q}}^+) \ne \emptyset$.
\end{claim}
Set $\alggrp{Q}' = n_{w_{\Delta} w_{\Delta_{\alggrp{Q}}}}\alggrp{Q}(n_{w_{\Delta}w_{\Delta_{\alggrp{Q}}}})^{-1}$ and $v = w(w_\Delta w_{\Delta_{\alggrp{Q}}})^{-1}$.
Then we have $\Delta_{\alggrp{Q}'} = -w_{\Delta}(\Delta_{\alggrp{Q}}) = w_{\Delta}w_{\Delta_{\alggrp{Q}}}(\Delta_{\alggrp{Q}})$.
Hence $v(\Delta_{\alggrp{Q}'}) = w(\Delta_{\alggrp{Q}}) \subset \Sigma^+$.
Therefore $v\in W^{\alggrp{Q}'}$.
We also have $v\ne 1$, hence $v\notin W_{\alggrp{Q}'}$.
Therefore $v^{-1}\notin W_{\alggrp{Q}'}$.
Hence there exists $\alpha\in\Sigma^+\setminus\Sigma^+_{\alggrp{Q}'}$ such that $v^{-1}(\alpha) < 0$.
If $v^{-1}(\alpha) \in \Sigma^-_{\alggrp{Q}'}$, then $\alpha = v(v^{-1}(\alpha)) \in v(\Sigma^-_{\alggrp{Q}'})\subset\Sigma^-$, this is a contradiction.
Hence $v^{-1}(\alpha)\in \Sigma^-\setminus\Sigma_{\alggrp{Q}'}^-$.
Namely $\alpha\in v(\Sigma^-\setminus\Sigma_{\alggrp{Q}'}^-)\cap (\Sigma^+\setminus\Sigma^+_{\alggrp{Q}'})$.
By the definition of $v$ and $\Sigma^{\mp}\setminus\Sigma^{\mp}_{\alggrp{Q}'} = w_\Delta (\Sigma^{\pm}\setminus\Sigma^{\pm}_{\alggrp{Q}}) = w_\Delta w_{\Delta_{\alggrp{Q}}}(\Sigma^{\pm}\setminus\Sigma^{\pm}_{\alggrp{Q}})$, we get $\alpha\in w(\Sigma^+\setminus\Sigma_{\alggrp{Q}}^+)\cap w_\Delta w_{\Delta_{\alggrp{Q}}}(\Sigma^-\setminus\Sigma^-_{\alggrp{Q}})$.
Hence $(w_\Delta w_{\Delta_{\alggrp{Q}}})^{-1}(\alpha)$ gives an element in the intersection of the claim.

Take $\lambda_0^-\in \Lambda_{\Delta_{\alggrp{Q}}}^+(1)$ as in Remark~\ref{rem:role of strongly positive/negative element} and set $\lambda = n_{w_\Delta w_{\Delta_{\alggrp{Q}}}}(\lambda_0^-)$.
We have $\lambda\in n_{w_{\Delta}w_{\Delta_{\alggrp{Q}}}}(\Lambda_{\Delta_{\alggrp{Q}}}^+(1))$.
Let $w\in W^{\alggrp{Q}}\setminus \{w_\Delta w_{\Delta_{\alggrp{Q}}}\}$.
Take $\alpha$ from the intersection in the claim.
Then since $\alpha\in \Sigma^-\setminus\Sigma^-_{\alggrp{Q}}$, we have $\langle w^{-1}w_\Delta w_{\Delta_{\alggrp{Q}}}(\alpha),\nu(n_w^{-1}(\lambda))\rangle = \langle \alpha,\nu(\lambda_0^-)\rangle < 0$.
Therefore $n_w^{-1}(\lambda)$ is not $\alggrp{Q}$-negative.
Hence $(w\sigma_\mathcal{A})(E(\lambda)) = 0$.

Assume that $X$ is an irreducible $\mathcal{A}$-submodule of $I(\alggrp{P},\sigma,\alggrp{Q})$ such that $\supp X = n_{w_\Delta w_{\Delta_{\alggrp{Q}}}}(\Lambda^+_{\Delta_{\alggrp{Q}}}(1))$.
Take $\lambda$ as in the above and $x\in X\setminus\{0\}$.
Let $x_w\in w\sigma$ such that $x = \sum x_w$.
Since $\lambda\in \supp X$, $E(\lambda)$ is invertible on $X$.
Hence $0\ne xE(\lambda) = \sum x_w E(\lambda) = x_{w_\Delta w_{\Delta_{\alggrp{Q}}}}E(\lambda)$.
Therefore $w_{\Delta}w_{\Delta_{\alggrp{Q}}}\sigma\cap X$ is a non-zero $\mathcal{A}$-submodule of $X$.
Since $X$ is irreducible, we have $X\subset w_{\Delta}w_{\Delta_{\alggrp{Q}}}\sigma$.
\end{proof}

We prove that $I(\alggrp{P},\sigma,\alggrp{Q})$ is irreducible.
Let $\pi\subset I(\alggrp{P},\sigma,\alggrp{Q})$ be a non-zero $\mathcal{H}$-submodule and take an irreducible $\mathcal{A}$-submodule $X'$ of $\pi$.
By Lemma~\ref{lem:I(P,sigma,Q) as A-mod}, we have the decomposition $I(\alggrp{P},\sigma,\alggrp{Q}) = \bigoplus_{\Delta_w\cap \Delta(\sigma) = \Delta_{\alggrp{Q}}} w\sigma_\mathcal{A}$.
Take $w$ such that the composition $X'\hookrightarrow I(\alggrp{P},\sigma,\alggrp{Q}) = \bigoplus_{\Delta_w\cap \Delta(\sigma) = \Delta_{\alggrp{Q}}} w\sigma_\mathcal{A}\twoheadrightarrow w\sigma_{\mathcal{A}}$ is non-zero.
Then we have $X'\hookrightarrow w\sigma_{\mathcal{A}}$.
Such irreducible representation has a form $wX$ for an irreducible submodule $X$ of $\sigma_\mathcal{A}$.
Hence we get a non-zero homomorphism $wX\otimes_\mathcal{A}\mathcal{H}\to \pi$.
By Lemma~\ref{lem:on A-submodule of sigma_A} and Corollary~\ref{cor:isom between std mod for simple module}, if $\Delta_{w'}\cap \Delta(\sigma) = \Delta_{\alggrp{Q}}$, then $wX\otimes_\mathcal{A}\mathcal{H}\simeq w'X\otimes_\mathcal{A}\mathcal{H}$.
Hence we have a non-zero homomorphism $w'X\otimes_\mathcal{A}\mathcal{H}\to \pi$, which gives a non-zero $\mathcal{A}$-homomorphism $w'X\to \pi$.
In particular we have $w_{\Delta}w_{\Delta_{\alggrp{Q}}} X\hookrightarrow \pi$.
By Lemma~\ref{lem:irreducible A-representation of parabolic induction}, the image of $w_{\Delta}w_{\Delta_{\alggrp{Q}}} X\hookrightarrow \pi\hookrightarrow I(\alggrp{P},\sigma,\alggrp{Q})$ is contained in $w_\Delta w_{\Delta_{\alggrp{Q}}}\sigma$.
Therefore we have $\pi\cap w_\Delta w_{\Delta_{\alggrp{Q}}}\sigma \ne 0$.
By Proposition~\ref{prop:M'-mod in parabolic induction}, we have $w_\Delta w_{\Delta_{\alggrp{Q}}}\sigma\subset \pi$.
Put $\sigma' = w_\Delta w_{\Delta_{\alggrp{Q}}}\sigma$.
Let $\alggrp{Q}'$ be a parabolic subgroup corresponding to $w_\Delta w_{\Delta_{\alggrp{Q}}}(\Delta_{\alggrp{Q}})$.
By Proposition~\ref{prop:M'-mod in parabolic induction}, $\sigma'$ is $\mathcal{H}_{\alggrp{Q}'}^+$-stable and we have a homomorphism $\sigma'\otimes_{\mathcal{H}_{\alggrp{Q}'}^+}\mathcal{H}\to \pi$.
By the construction, we have the following commutative diagram:
\[
\begin{tikzcd}
\sigma'\otimes_{\mathcal{H}_{\alggrp{Q}'}^+}\mathcal{H} \arrow{r}\arrow{d} &  I_{\alggrp{Q}}(e_{\alggrp{Q}}(\sigma))\arrow[twoheadrightarrow]{d}\\
\pi\arrow[hookrightarrow]{r} & I(\alggrp{P},\sigma,\alggrp{Q}).
\end{tikzcd}
\]
Here $\sigma'\otimes_{\mathcal{H}_{\alggrp{Q}'}^+}\mathcal{H} \to I_{\alggrp{Q}}(e_{\alggrp{Q}}(\sigma))$ is a homomorphism in Proposition~\ref{prop:tensor induction and parabolic induction}.
Hence it is an isomorphism.
Therefore $\pi = I(\alggrp{P},\sigma,\alggrp{Q})$.

\begin{cor}\label{cor:composition factors of parabolic induction}
Let $\alggrp{P} = \alggrp{M}\alggrp{N}$ be a parabolic subgroup and $\sigma$ an irreducible supersingular $\mathcal{H}_{\alggrp{M}}$-module.
Then the composition factors of $I_{\alggrp{P}}(\sigma)$ is given by $\{I(\alggrp{P},\sigma,\alggrp{Q})\mid \alggrp{P}(\sigma)\supset \alggrp{Q}\supset \alggrp{P}\}$.
\end{cor}
\begin{rem}
After finishing the proof of Theorem~\ref{thm:main theorem}, we know that $I_{\alggrp{P}}(\sigma)$ is multiplicity free.
\end{rem}

\subsection{Completion of the proof}
We finish the proof of Theorem~\ref{thm:main theorem}.
First we prove that any irreducible $\mathcal{H}$-module $\pi$ is isomorphic to $I(\alggrp{P},\sigma,\alggrp{Q})$ by induction on $\dim \alggrp{G}$.
By Corollary~\ref{cor:composition factors of parabolic induction}, it is sufficient to prove that $\pi$ is a subquotient of $I_{\alggrp{P}}(\sigma)$ for a parabolic subgroup $\alggrp{P} = \alggrp{M}\alggrp{N}$ and a supsersingular irreducible representation $\sigma$ of $\mathcal{H}_{\alggrp{M}}$.
We use the following lemma.
This is a generalization of a conjecture of Ollivier~\cite[Conjecture~5.20]{MR2728487}.

\begin{lem}
Let $X$ be an irreducible $\mathcal{A}$-module and $\Theta\subset\Delta$ such that $\supp X = \Lambda_\Theta^+(1)$.
Let $w,w'\in W$ such that $w(\Theta),w'(\Theta)\subset\Sigma^+$.
Then $wX\otimes_\mathcal{A}\mathcal{H}$ and $w'X\otimes_\mathcal{A}\mathcal{H}$ have the same composition factors with multiplicities.
\end{lem}
\begin{proof}
By Theorem~\ref{thm:std mod depends only on Delta_w}, we may assume $w = w_{\Delta}w_{\Theta_0}$ and $w' = ws_\alpha$ for $\Theta_0\subset\Delta$ and $\alpha\in\Theta_0$ such that $\Theta_0\subset \Theta$.
By Corollary~\ref{cor:isom between std mod for simple module}, we may assume $\alpha\in \Delta(X)$.
Take $x\in V^*$ such that $x(\nu(\Lambda(1)))\subset\Z$ and $\langle x,\check{\alpha}\rangle \ne 0$.
Consider an algebra homomorphism $C[\Lambda_\Theta(1)]\to C[t^{\pm 1}]$ defined by $\tau_\lambda\mapsto t^{\langle x,\nu(\lambda)\rangle}$.
Both $X$ and $C[t^{\pm 1}]$ are representations of $\Lambda_\Theta(1)$.
Hence we have a representation $X[t^{\pm 1}] = X\otimes C[t^{\pm 1}]$ of $\Lambda_\Theta(1)$, equivalently $C[\Lambda_\Theta(1)]$-module.
Via $\chi_\Theta$, we get an $\mathcal{A}$-module $X[t^{\pm 1}]$.
Consider the $\mathcal{H}$-module $wX[t^{\pm 1}]\otimes_\mathcal{A}\mathcal{H}$ and $w'X[t^{\pm 1}]\otimes_\mathcal{A}\mathcal{H}$.
By Proposition~\ref{prop:freeness}, these modules are free $C[t^{\pm 1}]$-modules.
From the proof of Corollary~\ref{cor:isom between std mod for simple module}, we have injective homomorphisms $wX[t^{\pm 1}]\otimes_\mathcal{A}\mathcal{H}\to w'X[t^{\pm 1}]\otimes_\mathcal{A}\mathcal{H}$ and $w'X[t^{\pm 1}]\otimes_\mathcal{A}\mathcal{H}\to wX[t^{\pm 1}]\otimes_\mathcal{A}\mathcal{H}$.
The compositions are given by $1 - t^{\langle x,-\check{\alpha}\rangle}$.
By \cite[Lemma~31]{MR1361556}, we get the lemma.
\end{proof}

Let $\pi$ be an irreducible representation of $\mathcal{H}$ and take an irreducible $\mathcal{A}$-submodule $X'$.
Then there exists $\Theta\subset\Delta$ and $w\in W$ such that $\supp X' = n_w(\Lambda_\Theta^+(1))$ and $w(\Theta) > 0$.
Take an irreducible $\mathcal{A}$-module $X$ such that $X' = wX$.
Then $\supp X = \Lambda_\Theta^+(1)$.
Since we have $wX = X'\hookrightarrow \pi$, we have a non-zero homomorphism $wX\otimes_\mathcal{A}\mathcal{H}\to \pi$.
Therefore $\pi$ is a quotient of $wX\otimes_\mathcal{A}\mathcal{H}$.
Taking $w' = 1$ in the above lemma, the composition factors of $wX\otimes_\mathcal{A}\mathcal{H}$ and $X\otimes_\mathcal{A}\mathcal{H}$ is the same.
Hence $\pi$ is a subquotient of $X\otimes_\mathcal{A}\mathcal{H}$.

If $\Theta = \Delta$, then $\pi$ is supersingular.
So we have nothing to prove.
We assume that $\Theta\ne\Delta$.
Let $\alggrp{P}' = \alggrp{M}'\alggrp{N}'$ (resp.\ $\alggrp{P} = \alggrp{M}\alggrp{N}$) be the parabolic subgroup corresponding to $\Theta$ (resp. $-w_\Delta(\Theta)$).
\begin{lem}
Put $\mathcal{A}' = \mathcal{A}\cap j_{\alggrp{M}'}^+(\mathcal{H}_{\alggrp{M}'}^+)$.
Then $X\otimes_\mathcal{A}\mathcal{H}\simeq X\otimes_{\mathcal{A}'}\mathcal{H}$ and $X\otimes_{\mathcal{A}'}\mathcal{H}_{\alggrp{M}'}^+$ is the restriction of an $\mathcal{H}_{\alggrp{M}'}$-module.
\end{lem}
\begin{proof}
We have a natural homomorphism $X\otimes_{\mathcal{A}'}\mathcal{H}\to X\otimes_{\mathcal{A}}\mathcal{H}$.
Take $\lambda_0^+\in \Lambda(1)$ as in Remark~\ref{rem:role of strongly positive/negative element} for $\alggrp{M}'$.
Then we have $\lambda_0^+\in \Lambda_\Theta(1)$.
Hence $E(\lambda_0^+)$ is invertible on $X$.
If $\lambda\in \Lambda(1)$ is not $\alggrp{M}'$-positive, then $\nu(\lambda)$ and $\nu(\lambda_0^+)$ does not belong to the same closed Weyl chamber, hence $E(\lambda_0^+)E(\lambda) = 0$.
Hence $x\otimes E(\lambda)F = xE(\lambda_0^+)^{-1}\otimes E(\lambda_0^+)E(\lambda)F = 0$ in $X\otimes_{\mathcal{A}'}\mathcal{H}$ for $x\in X$ and $F\in \mathcal{H}$.
Therefore $x\otimes E(\lambda)F = 0 = xE(\lambda)\otimes F$.
Hence $X\otimes_{\mathcal{A}'}\mathcal{H}\to X\otimes_{\mathcal{A}}\mathcal{H}$ is an isomorphism.

The element $E^{\alggrp{M}'}(\lambda_0^+)$ is in the center of $\mathcal{H}_{\alggrp{M}'}^+$ and $j_{\alggrp{M}'}^+(E^{\alggrp{M}'}(\lambda_0^+)) = E(\lambda_0^+)\in \mathcal{A}'$ is invertible on $X$.
Hence $E^{\alggrp{M}'}(\lambda_0^+)$ is invertible on $X\otimes_{\mathcal{A}'}\mathcal{H}_{\alggrp{M}'}^+$.
Since $\mathcal{H}_{\alggrp{M}'} = \mathcal{H}_{\alggrp{M}'}^+E^{\alggrp{M}'}(\lambda_0^+)^{-1}$ by Remark~\ref{rem:role of strongly positive/negative element}, we have $X\otimes_{\mathcal{A}'}\mathcal{H}_{\alggrp{M}'}^+\simeq (X\otimes_{\mathcal{A}'}\mathcal{H}_{\alggrp{M}'}^+)E^{\alggrp{M}}(\lambda_0^+)^{-1}\simeq X\otimes_{\mathcal{A}'}\mathcal{H}_{\alggrp{M}'}$.
\end{proof}
Hence we have $X\otimes_\mathcal{A}\mathcal{H} \simeq (X\otimes_{\mathcal{A}'}\mathcal{H}^+_{\alggrp{M}'})\otimes_{\mathcal{H}^+_{\alggrp{M}'}}\mathcal{H}$.
By Proposition~\ref{prop:tensor induction and parabolic induction}, $\pi$ is an subquotient of $I_{\alggrp{P}}(\sigma)$  for an $\mathcal{H}_{\alggrp{M}}$-module $\sigma$. (Explicitly, $\sigma$ is given by $X\otimes_{\mathcal{A}'}\mathcal{H}_{\alggrp{M}'}^+$ twisting by $n_{w_{\Delta}w_{\Delta_{\alggrp{M}}}}$.)
Hence for some irreducible subquotient $\sigma'$ of $\sigma$, $\pi$ is a subquotient of $I_{\alggrp{P}}(\sigma')$.
By inductive hypothesis, $\sigma'$ is a subquotient of $I_{\alggrp{P}_0\cap\alggrp{M}}(\sigma_0)$ where $\alggrp{P}_0 = \alggrp{M}_0\alggrp{N}_0\subset \alggrp{P}$ is a parabolic subgroup of $\alggrp{G}$ and $\sigma_0$ is a supersingular $\mathcal{H}_{\alggrp{M}_0}$-module.
Hence $\pi$ is a subquotient of $I_{\alggrp{P}_0}(\sigma_0)$.

Finally, we prove that $I(\alggrp{P},\sigma,\alggrp{Q}) \simeq I(\alggrp{P}',\sigma,\alggrp{Q}')$ implies $\alggrp{P} = \alggrp{P}'$, $\alggrp{Q} = \alggrp{Q}'$  and $\sigma\simeq\sigma'$.
Let $\alggrp{M}$ (resp.\ $\alggrp{M}'$) be the Levi subgroup of $\alggrp{P}$ (resp.\ $\alggrp{P}'$).
Let $X$ be an irreducible $\mathcal{A}$-submodule of $I(\alggrp{P},\sigma,\alggrp{Q})$.
By Lemma~\ref{lem:I(P,sigma,Q) as A-mod} and Lemma~\ref{lem:on A-submodule of sigma_A}, $\supp X = n_{w}(\Lambda_{\Delta_{\alggrp{M}}}^+(1))$ for $w\in W$ such that $\Delta_w\cap \Delta(\sigma) = \Delta_{\alggrp{Q}}$.
Hence
\[
	\{n_{w}(\Lambda_{\Delta_{\alggrp{M}}}^+(1))\mid \Delta_w\cap \Delta(\sigma) = \Delta_{\alggrp{Q}}\}
	=
	\{n_{w}(\Lambda_{\Delta_{\alggrp{M}'}}^+(1))\mid \Delta_w\cap \Delta(\sigma') = \Delta_{\alggrp{Q}'}\}.
\]
Notice that $w\Lambda_{\Theta}^+(1) = w'\Lambda_{\Theta'}^+(1)$ implies $\Theta = \Theta'$ and $w\in w'W_{\Theta'}$.
Hence $\Delta_{\alggrp{M}} = \Delta_{\alggrp{M}'}$.
Therefore $\alggrp{P} = \alggrp{P}'$.
If $w\in W$ satisfies $\Delta_w\cap \Delta(\sigma) = \Delta_{\alggrp{Q}}$, then $\Delta_w\supset \Delta_{\alggrp{Q}}\supset \Delta_{\alggrp{P}}$.
Hence $w(\Delta_{\alggrp{P}})\subset \Sigma^+$.
Namely $w\in W^{\alggrp{M}}$.
Therefore, if $w,w'\in W$ satisfies $\Delta_w\cap \Delta(\sigma) = \Delta_{\alggrp{Q}}$, $\Delta_{w'}\cap \Delta(\sigma') = \Delta_{\alggrp{Q}'}$ and $w \in w'W_{\alggrp{M}}$, then $w = w'$.
Hence
\[
	\{w\mid \Delta_w\cap \Delta(\sigma) = \Delta_{\alggrp{Q}}\} =
	\{w\mid \Delta_w\cap \Delta(\sigma') = \Delta_{\alggrp{Q}'}\}.
\]
Therefore $\Delta(\sigma) = \Delta(\sigma')$ and $\alggrp{Q} = \alggrp{Q}'$.

We have $w_\Delta w_{\Delta_{\alggrp{Q}}}\sigma\subset I(\alggrp{P},\sigma,\alggrp{Q})$.
Take $\lambda\in \Lambda(1)$ as in Lemma~\ref{lem:irreducible A-representation of parabolic induction}.
Then the subspace $w_\Delta w_{\Delta_{\alggrp{Q}}}\sigma\subset I(\alggrp{P},\sigma,\alggrp{Q})$ is characterized by $\{x\in I(\alggrp{P},\sigma,\alggrp{Q})\mid xE(\lambda) \ne 0\}$.
Hence an isomorphism $I(\alggrp{P},\sigma,\alggrp{Q})\to I(\alggrp{P},\sigma',\alggrp{Q}')$ gives a morphism $w_\Delta w_{\Delta_{\alggrp{Q}}}\sigma\to w_\Delta w_{\Delta_{\alggrp{Q}}}\sigma'$.
By Proposition~\ref{prop:M'-mod in parabolic induction}, we have a non-zero homomorphism $\sigma\to \sigma'$ as $\mathcal{H}_{\alggrp{M}}^-$-modules.
By Remark~\ref{rem:role of strongly positive/negative element}, $\sigma$ and $\sigma'$ are irreducible $\mathcal{H}_{\alggrp{M}}^-$-modules.
Hence $\sigma\simeq\sigma'$ as $\mathcal{H}_{\alggrp{M}}^-$-modules and by Remark~\ref{rem:role of strongly positive/negative element} again, $\sigma\simeq\sigma'$ as $\mathcal{H}_{\alggrp{M}}$-modules.

Finally we introduce the notion of supercuspidality and compare it with supersingularlity.

\begin{defn}
An irreducible representation $\pi$ of $\mathcal{H}$ is called \emph{supercuspidal} if it is not isomorphic to a subquotient of $I_{\alggrp{P}}(\sigma)$ where $\alggrp{P} = \alggrp{M}\alggrp{N}$ is a proper parabolic subgroup of $\alggrp{G}$ and $\sigma$ an irreducible representation of $\mathcal{H}_{\alggrp{M}}$.
\end{defn}
\begin{cor}
Let $\pi = I(\alggrp{P},\sigma,\alggrp{Q})$ is an irreducible representation of $\mathcal{H}$.
Then the following is equivalent.
\begin{enumerate}
\item $\pi$ is supersingular.
\item $\pi$ is supercuspidal.
\item $\alggrp{P} = \alggrp{G}$.
\end{enumerate}
\end{cor}
\begin{proof}
Obviously (3) implies (1).
Since $I(\alggrp{P},\sigma,\alggrp{Q})$ is the subquotient of $I_{\alggrp{P}}(\sigma)$, (2) implies (3).
Assume that $\pi$ is not supercuspidal.
Take $(\alggrp{P}',\sigma')$ such that $\pi$ is a subquotient of $I_{\alggrp{P}'}(\sigma')$ and $\alggrp{P}'$ is minimal subject to this condition.
Then $\sigma'$ is supercuspidal.
We have already proved that irreducible supercuspidal representations are supersingular.
Hence $\sigma'$ is supersingular.
By Corollary~\ref{cor:composition factors of parabolic induction}, it is of a form $I(\alggrp{P}',\sigma',\alggrp{Q}')$ for some $\alggrp{Q}'$.
Hence $\alggrp{P} = \alggrp{P}' \ne\alggrp{G}$.
Let $X\subset I(\alggrp{P},\sigma,\alggrp{Q})$ be an irreducible $\mathcal{A}$-submodule.
By Lemma~\ref{lem:I(P,sigma,Q) as A-mod}, $X\subset w\sigma_\mathcal{A}$ for some $w\in W$ and by \ref{lem:on A-submodule of sigma_A}, we have $\supp X = n_w(\Lambda_{\Delta_{\alggrp{P}}}(1))$.
Hence by Lemma~\ref{lem:supersingurality, in terms of A-moudle} \ref{enum:supersingularity, by any X}, $I(\alggrp{P},\sigma,\alggrp{Q})$ is not supersingular.
\end{proof}

\def\cprime{$'$} \def\dbar{\leavevmode\hbox to 0pt{\hskip.2ex \accent"16\hss}d}
  \def\Dbar{\leavevmode\lower.6ex\hbox to 0pt{\hskip-.23ex\accent"16\hss}D}
  \def\cftil#1{\ifmmode\setbox7\hbox{$\accent"5E#1$}\else
  \setbox7\hbox{\accent"5E#1}\penalty 10000\relax\fi\raise 1\ht7
  \hbox{\lower1.15ex\hbox to 1\wd7{\hss\accent"7E\hss}}\penalty 10000
  \hskip-1\wd7\penalty 10000\box7}
  \def\cfudot#1{\ifmmode\setbox7\hbox{$\accent"5E#1$}\else
  \setbox7\hbox{\accent"5E#1}\penalty 10000\relax\fi\raise 1\ht7
  \hbox{\raise.1ex\hbox to 1\wd7{\hss.\hss}}\penalty 10000 \hskip-1\wd7\penalty
  10000\box7} \newcommand{\noop}[1]{}


\begin{thebibliography}{AHHV14}

\bibitem[Abe13]{MR3143708}
N.~Abe, \emph{On a classification of irreducible admissible modulo {$p$}
  representations of a {$p$}-adic split reductive group}, Compos. Math.
  \textbf{149} (2013), no.~12, 2139--2168.

\bibitem[AHHV14]{arXiv:1412.0737}
N.~Abe, G.~Henniart, F.~Herzig, and M.-F. Vigneras, \emph{{A} classification of
  irreducible admissible mod {$p$} representations of {$p$}-adic reductive
  groups}, arXiv:1412.0737.

\bibitem[BL95]{MR1361556}
L.~Barthel and R.~Livn{\'e}, \emph{Modular representations of {${\rm GL}_2$} of
  a local field: the ordinary, unramified case}, J. Number Theory \textbf{55}
  (1995), no.~1, 1--27.

\bibitem[GK]{Grosse-Klonne-pro-p-Iwahori-Hecke-and-Galois-rep}
E.~Gro{\ss}e-Kl\"onne, \emph{{F}rom pro-{$p$} {I}wahori-{H}ecke modules to
  {$(\phi,\Gamma)$}-modules {I}}, preprint.

\bibitem[IM65]{MR0185016}
N.~Iwahori and H.~Matsumoto, \emph{On some {B}ruhat decomposition and the
  structure of the {H}ecke rings of {${p}$}-adic {C}hevalley groups}, Inst.
  Hautes \'Etudes Sci. Publ. Math. (1965), no.~25, 5--48.

\bibitem[Oll10]{MR2728487}
R.~Ollivier, \emph{Parabolic induction and {H}ecke modules in characteristic
  {$p$} for {$p$}-adic {${\rm GL}_n$}}, Algebra Number Theory \textbf{4}
  (2010), no.~6, 701--742.

\bibitem[Oll12]{arXiv:1211.5366}
R.~Ollivier, \emph{{C}ompatibility between {S}atake and {B}ernstein-type
  isomorphisms in characteristic p}, to appear in Algebra Number Theory.

\bibitem[Tit66]{MR0206117}
J.~Tits, \emph{Normalisateurs de tores. {I}. {G}roupes de {C}oxeter \'etendus},
  J. Algebra \textbf{4} (1966), 96--116.

\bibitem[Viga]{Vigneras-prop}
M.-F. Vign\'eras, \emph{The pro-{$p$}-{I}wahori {H}ecke algebra of a {$p$}-adic
  group {I}}, to appear in Compositio mathematica.

\bibitem[Vigb]{Vigneras-prop-III}
M.-F. Vign\'eras, \emph{The pro-{$p$}-{I}wahori {H}ecke algebra of a {$p$}-adic
  group {III}}, to appear in Journal of the Institute of Mathematics of
  Jussieu.

\bibitem[Vig05]{MR2122539}
M.-F. Vign{\'e}ras, \emph{Pro-{$p$}-{I}wahori {H}ecke ring and supersingular
  {$\overline{\mathbf F}_p$}-representations}, Math. Ann. \textbf{331} (2005),
  no.~3, 523--556.

\bibitem[Vig14]{Vigneras-prop-II}
M.-F. Vign\'eras, \emph{The pro-{$p$}-{I}wahori {H}ecke algebra of a {$p$}-adic
  group {II}}, Muenster J. of Math. \textbf{7} (2014), no.~1, 364--379.

\end{thebibliography}
\end{document}